\documentclass[11pt,reqno,twoside]{article}
\usepackage{amssymb,amsfonts,amsthm,amsmath}
\usepackage{cite}

\usepackage[left=1 in,top=1 in,right=1 in,bottom=1 in]{geometry}

\numberwithin{equation}{section}

\newcommand{\Kzero}{\mathcal K_1}
\newcommand{\Kfirst}{\mathcal K_2}
\newcommand{\Kmore}{\bar {\mathcal K}_2}
\newcommand{\Kth}{\bar {\mathcal K}_1}
\newcommand{\Kfo}{\tilde{\mathcal K}_1}
\newcommand{\Kfv}{\tilde{\mathcal K}_2}

\newcommand{\Azero}{\mathcal A_1}
\newcommand{\Afirst}{\mathcal G_1}
\newcommand{\Amore}{\mathcal A_2}
\newcommand{\Ath}{\mathcal G_2}

\newcommand{\esssup}{\mathop{\mathrm{ess\,sup}}}

\newcommand{\beq}{\begin{equation}}
\newcommand{\eeq}{\end{equation}}
\newcommand{\beqs}{\begin{equation*}}
\newcommand{\eeqs}{\end{equation*}}
\newcommand{\ba}{\begin{array}}
\newcommand{\ea}{\end{array}}
\newcommand{\beas}{\begin{eqnarray*}}
\newcommand{\eeas}{\end{eqnarray*}}
\newcommand{\bea}{\begin{eqnarray}}
\newcommand{\eea}{\end{eqnarray}}
\newcommand{\bal}{\begin{align}}
\newcommand{\eal}{\end{align}}

\newcommand{\bals}{\begin{align*}}
\newcommand{\eals}{\end{align*}}

\newcommand{\R}{\ensuremath{\mathbb R}}

\newcommand{\N}{\ensuremath{\mathbb N}}

\newcommand{\norm}[1]{\| {#1} \|}

\newcommand{\bds}{\begin{displaystyle}}
\newcommand{\eds}{\end{displaystyle}}

 \def\longequals{\mathbin{=\kern-2pt=}}
 \def\eqdef{\mathbin{\buildrel \rm def \over \longequals}}

\def\varep{\varepsilon}
\def\ddt{\frac{d}{dt}}

\newtheorem{theorem}{Theorem}[section]

\newtheorem{lemma}[theorem]{Lemma}
\newtheorem{corollary}[theorem]{Corollary}
\newtheorem{proposition}[theorem]{Proposition}

\newtheorem{remark}[theorem]{Remark}

\theoremstyle{remark}

\title{Global Estimates for Generalized Forchheimer Flows of Slightly Compressible Fluids}

\author{Luan Hoang$^{a}$ and Thinh Kieu$^b$}

\date{\today}

\begin{document}
\maketitle

\begin{center}
\textit{$^a$Department of Mathematics and Statistics, Texas Tech University, Box 41042, Lubbock, TX 79409--1042, U. S. A.} \\
\textit{$^b$Department of Mathematics, University of North Georgia, Gainesville Campus, 3820 Mundy Mill Rd., Oakwood, GA 30566, U. S. A.}\\
Email addresses: \texttt{luan.hoang@ttu.edu, thinh.kieu@ung.edu}
\end{center}

\begin{center}
{\Large \textit{Dedicated to Professor Duong Minh Duc with gratitude.}}
\end{center}

\begin{abstract}
This paper is focused on the generalized Forchheimer flows of slightly compressible fluids in porous media. They  are reformulated as a degenerate parabolic equation for the pressure. The initial boundary value problem is studied with time-dependent Dirichlet boundary data.  The estimates up to the boundary and for all time are derived for the $L^\infty$-norm of the pressure, its gradient and time derivative. Large-time estimates are established to be independent of the initial data.
Particularly, thanks to the special structure of the pressure's nonlinear equation, the global gradient estimates are obtained in a relatively simple way, avoiding complicated calculations and a prior requirement of H\"older estimates.
\end{abstract}

\pagestyle{myheadings}\markboth{L. Hoang and T. Kieu}
{Global Estimates for Generalized Forchheimer Flows}

\section{Introduction}
\label{intro}

In studies of fluid dynamics in porous media, Darcy's law is ubiquitously used.
However this linear relation between the velocity and pressure gradient is not always valid.
The deviation from Darcy's law is well-known  when the Reynolds number increases \cite{Muskatbook,BearBook}. Such a deviation was even noticed in early works by Darcy \cite{Darcybook} and Dupuit \cite{Dupuit1857}.
Nonlinear alternatives were formulated by Forchheimer \cite{Forchh1901,ForchheimerBook}, and were studied  extensively afterward in physics and engineering, see \cite{Muskatbook,Ward64,BearBook,NieldBook,StraughanBook} and references therein. 
In contrast to the vast mathematical research of Darcy's flows, see e.g. \cite{VazquezPorousBook},
existing mathematical papers on Forchheimer flows are much fewer and came much later. 
Even rarer are the ones for compressible fluids.
(See, e.g., \cite{ABHI1,HI1,HI2} for more introduction to Forchheimer flows.) 
Generalized Forchheimer equations were proposed \cite{ABHI1,HI1,HI2,HIKS1,HKP1} in order to cover a general class of fluid flows in porous media formulated from experiments.
In our previous paper \cite{HK1}, we derive interior estimates for generalized Forchheimer flows of slightly compressible fluids. In this article, we focus on the spatially global estimates, i.e., on the entire domain.
Since the equation for pressure is of degenerate parabolic type, finding $L^\infty$-bounds for its gradient, in general, is a difficult task and requires much work, see, for e.g.,  \cite{VazquezSmoothBook,RVV2009,RVV2010,LD2000}. 
We will show that the gradient estimates for the particular flows under the current study  can be obtained in a relatively simple way thanks to their equations' specific structure.

In this paper,  we consider the initial boundary value problem for the pressure in a bounded domain with time-dependent Dirichlet boundary data.
Our goal is to obtain bounds for the pressure's gradient and time derivative in terms of initial and boundary data.
Aiming at studying the long-term dynamics of the system, we also emphasize the bounds for large time.
The standard technique from \cite{LadyParaBook68}, see also \cite{LiebermanPara96, DiDegenerateBook}, requires a H\"older estimate first, which, by itself, is not a small task.
Moreover, the estimates obtained in \cite{LadyParaBook68} for degenerate equations depend on the initial value of the gradient. This will result in long-time estimates for the gradient that are dependent on the initial data.
That means that the long-time dynamics cannot be reduced to a much smaller set, say, the global attractor (c.f. \cite{TemamDynBook,SellYouBook}).
In this paper, we will demonstrate that these two setbacks can be overcome without overcomplicated calculations.
First, we extend the techniques for parabolic equations in \cite{LadyParaBook68}   to our degenerate one in a suitable way. Such extension is possible due to the equation's special structure. Second, by doing analysis on the entire domain, we avoid the spatially local consideration, and hence, the requirement for the H\"older estimates. 
Finally, localizing the estimates in time and utilizing uniform Gronwall-type inequalities remove the estimates' dependence on the initial data.

The paper is organized as follows.
In section \ref{prelim}, we recall basic facts about generalized Forchheimer equations, and prove a global embedding of Ladyzhenskaya-Uraltseva type.  It is different from the versions in \cite{LadyParaBook68, HK1} which are localizations at an interior or a boundary point. 
This will contribute to simpler proofs for the global estimates in section \ref{PGradS}.
In section \ref{revision},  we review some relevant  estimates from our previous works \cite{HIKS1,HK1} which will be used repeatedly.  %
In section \ref{PSec}, we use De Giorgi's iteration to obtain (spatially) global estimates.
Time-local inequality \eqref{Ne1} is ``quasi-homogeneous" in terms of $L_{x,t}^\alpha$-norm of the solution.
This improves estimates in our previous works, and is similar to a recent improvement on interior estimates in \cite{Surnachev2012}. Explicit estimates in terms of initial and boundary data are established in Theorem \ref{Ntheo43}.
Section \ref{PGradBdry} is devoted to estimating the maximum of the gradient's modulus on the boundary.
Unlike the results in \cite{LadyParaBook68}, we are able to localize the estimate in time in Theorem \ref{theo51}, hence, it is independent of the initial gradient.
This is essential to finding the gradient bounds for large time, by combining it with certain uniform Gronwall-type estimates. This is demonstrated in Corollary \ref{Cor5}.
In section \ref{PGradS}, we establish global $L^s$-estimates for the gradient \emph{for all} $s>0$ and all positive time. This requires only moderate regularity for the initial data, say, $L^\alpha$ and $W^{1,2-a}$ with an appropriate $\alpha>0$. 
Section \ref{PGradInfty} contains $L^\infty$-estimates for the gradient.
Thanks to the special form of the resulting equations \eqref{um} for the gradient, it is possible to apply the De Giorgi technique, see Theorem \ref{GradUni}.
Section \ref{PTSec} has the same results as section \ref{PGradInfty} but for the time derivative. The main estimates are in Theorems \ref{ptthm} and \ref{ptlimthm}.

\section{Preliminaries}
\label{prelim}

Consider a fluid  in a porous medium  in space $\R^n$. For physics problem $n=3$, but here we consider any $n\ge 2$. Let $x\in \R^n$ and $t\in\R$  be the spatial and time variables.
The fluid flow has velocity $v(x,t)\in \R^n$, pressure $p(x,t)\in \R$ and density $\rho (x,t)\in [0,\infty)$.

The generalized Forchheimer equations studied in \cite{ABHI1,HI1,HI2,HIKS1,HKP1} are of the the form:
\beq\label{gForch} g(|v|)v=-\nabla p,\eeq
where $g(s)\ge 0$ is a function defined on $[0,\infty)$.
When $g(s)=\alpha, \alpha +\beta s,\alpha +\beta s+\gamma s^2,\alpha +\gamma_m s^{m-1}$, where $\alpha,\beta,\gamma,m,\gamma_m$ are empirical constants, we have Darcy's law, Forchheimer's two-term, three-term and power laws, respectively.

In this paper, we study the case when the function $g$ in \eqref{gForch} is a generalized polynomial with positive coefficients, that is,
\beq\label{gsa} g(s) =a_0s^{\alpha_0}+a_1 s^{\alpha_1}+\ldots +a_N s^{\alpha_N}\quad\text{for}\quad s\ge 0,\eeq
where  
$N\ge 1$, the powers $\alpha_0=0<\alpha_1<\ldots<\alpha_N$ are fixed real numbers (not necessarily integers), the coefficients $a_0,a_1,\ldots,a_N$ are positive.
This function $g(s)$ is referred to as the Forchheimer polynomial in equation \eqref{gForch}. 

From \eqref{gForch} one can solve $v$ implicitly in terms of $\nabla p$ and
derives a nonlinear version of Darcy's equation:
\beq\label{u-forma} v= -K(|\nabla p|)\nabla p,\eeq
where the function $K:[0,\infty)\to[0,\infty)$  is defined by
\beq\label{Kdef} K(\xi)=\frac1{g(s(\xi))},
\text{ with } s=s(\xi)\ge 0 \text{ satisfying } sg(s)=\xi, \ \text{ for }\xi\ge 0. \eeq

In addition to \eqref{gForch}, we have the continuity equation 
\beq\label{conti-eq} 
\phi\frac{\partial \rho}{\partial t}+\nabla\cdot(\rho v)=0,\eeq
where number $\phi\in(0,1)$ is the constant porosity.
Also, for slightly compressible fluids, the equation of state is 
\beq\label{slight-compress} 
\frac{d\rho}{dp}=\frac{\rho}{\kappa},\quad\text{with } \kappa=const.>0.
\eeq

From \eqref{u-forma}, \eqref{conti-eq} and \eqref{slight-compress} one derives a scalar equation for the pressure:
\beq\label{dafo-nonlin} 
\phi\frac{\partial p}{\partial t}=\kappa\nabla \cdot (K(|\nabla p|)\nabla p) + K(|\nabla p|)|\nabla p|^2.\eeq

On the right-hand side of \eqref{dafo-nonlin}, the constant $\kappa$ is very large  for most slightly compressible fluids in porous media \cite{Muskatbook}, hence we neglect its second term  and by scaling the time variable, we study the reduced equation
\beq\label{lin-p} 
\frac{\partial p}{\partial t} = \nabla\cdot (K(|\nabla p|)\nabla p).\eeq
Note that this reduction is commonly used in engineering.

Our aim is to study the initial boundary value problem for equation \eqref{lin-p} in a bounded domain.
Here afterward $U$ is a bounded, open, connected subset of $\mathbb{R}^n$, $n=2,3,\ldots$ with $C^2$ boundary $\Gamma=\partial U$. 
Throughout, 
the Forchheimer polynomial $g(s)$ in \eqref{gsa} is fixed.
The following number is frequently used in our  calculations:
\beq\label{ab} a=\frac{\alpha_N}{1+\alpha_N}\in(0,1).
\eeq

The function $K(\xi)$ in \eqref{Kdef}  has the following properties (c.f. \cite{ABHI1,HI1}): it is decreasing in $\xi$ mapping $\xi\in[0,\infty)$ onto $(0,1/a_0]$, and
\beq\label{K-est-3}
\frac{d_1}{(1+\xi)^a}\le K(\xi)\le \frac{d_2}{(1+\xi)^a},
\eeq
\beq\label{Kestn}
d_3(\xi^{2-a}-1)\le K(\xi)\xi^2\le d_2\xi^{2-a},
\eeq
\beq\label{K-est-2}
-a K(\xi)\le  K'(\xi)\xi\le 0,
\eeq
where $d_1$, $d_2$, $d_3$ are positive constants depending on $g$.
As in previous works, we use the function $H(\xi)$ defined by 
\beqs H(\xi)=\int_0^{\xi^2}K(\sqrt s)ds\quad \hbox{for } \xi\ge 0.\eeqs
It satisfies $K(\xi)\xi^2\leq H(\xi)\le 2K(\xi)\xi^2$, thus, by \eqref{Kestn},
\beq\label{Hcompare} d_3(\xi^{2-a}-1)\leq H(\xi)\le 2d_2\xi^{2-a}.\eeq


The following parabolic Poincar\'e-Sobolev inequalities are needed for our study. 
For each $T>0$, denote $Q_T=U\times (0,T)$.
We define a threshold exponent 
\beq\label{alphastar} 
\alpha_*=\frac{an}{2-a}.\eeq

\begin{lemma}[cf. \cite{HK1}, Lemma 2.1] \label{Sob4}
Assume
\beq\label{alcond}
\alpha \ge 2 \quad \text{and}\quad  \alpha>\alpha_*.
\eeq
Let 
\beq\label{expndef}
p=\alpha\Big(1+\frac{2-a}n\Big)-a.
\eeq
Then
\beq\label{nonzerobdn}
\|u\|_{L^p(Q_T)}\le C(1+\delta T)^{1/p}[[u]],
\eeq
where $\delta=1$ in general, $\delta=0$ in case $u$ vanishes on the boundary $\partial U$, and
\beq\label{udouble}
[[u]]=\esssup_{[0,T]}\|u(\cdot,t)\|_{L^\alpha(U)}+\Big(\int_0^T\int_U |u(x,t)|^{\alpha-2}|\nabla u(x,t)|^{2-a}dx dt\Big)^\frac 1{\alpha-a}.
\eeq
\end{lemma}

The next is a particular embedding with spatial weights from Lemma 2.4 of \cite{HKP1} (see inequality (2.28) with $m=2$ there).

\begin{lemma}[cf. \cite{HKP1}, Lemma 2.4]\label{WSobolev}
Given $W(x,t)>0$ on $Q_T$.   
Let $r$ be a number that satisfies 
\beq\label{rcond} \frac{2n}{n+2}< r<  2.\eeq 
Set
\beq\label{sstar}
\varrho =\varrho(r)\eqdef 4(1-1/r^*).
\eeq
Then 
\beq\label{Wei1}
\|u\|_{L^\varrho (Q_T)}
\le C  [[u]]_{2,W;T} \Big\{\delta T^{\frac 1\varrho}+  \esssup_{t\in[0,T]} \Big (\int_U  W(x,t)^{-\frac r{2-r}} \chi_{{\rm supp}\,u}(x,t)dx\Big )^\frac{2-r}{\varrho r}\Big \},
\eeq
where $\delta=1$ in general, $\delta=0$ in case $u$ vanishes on the boundary $\partial U$,  and 
\beq\label{normquad}
[[u]]_{2,W;T}=\esssup_{[0,T]} \|u(\cdot,t)\|_{L^2(U)}+\Big(\int_0^T \int_U W(x,t)|\nabla u(x,t)|^2 dx dt\Big)^\frac 1 2.
\eeq
\end{lemma}

Above, ${\rm supp} f$ denotes the support set of a function $f$, and $\chi_A$ denotes the characteristic functions of a set $A$.

The following embedding is a global version of Lemma 3.7 from \cite{HKP1} and Lemma 5.4 on page 93 in \cite{LadyParaBook68}.

\begin{lemma}\label{LUK} 
Suppose $w$ is a function on $\bar U$ that satisfies $|\nabla w|\le M$ on $\Gamma$.

Let $v=\max\{|\nabla w|^2-M^2,0\}$.
For each $s\geq1$, there exists a constant $C>0$ depending on $s$ such that for any $k\in \R$, 
\begin{equation}\label{LUineq}
\int_U K(|\nabla w|) v^{s+1} dx \le  C \max|w-k|^2\int_U K(|\nabla w|)|\nabla^2 w|^2 v^{s-1}  dx
+  CM^4 \int_U K(|\nabla w|)v^{s-1}  dx .
\end{equation}
\end{lemma}
\begin{proof} 
Let $I =  \int_U K(|\nabla w|) v^{s+1} dx$.
Note that $v=0$ on $\Gamma$.
First, we see that 
\begin{align*}
I&\le   \int_U K(|\nabla w|) v^s |\nabla w|^2 dx
=\sum_{i=1}^n\int_U K(|\nabla w|) v^s \partial_i w \partial_i (w-k)  dx.
\end{align*}
By integration by parts,
\begin{align*}
I
&= - \sum_{i=1}^n\int_U \partial_i (K(|\nabla w|) v^s \partial_i w )\cdot  (w-k) dx\\
& =  - \sum_{i,j=1}^n\int_U\Big( K'(|\nabla w|)\frac{\partial_i\partial_j w\partial_j w}{|\nabla w|} \Big) v^s \partial_i w\cdot (w-k) dx - \int_U K(|\nabla w|)v^s \Delta w  \cdot (w-k) dx\\
&\quad - s\sum_{i,j=1}^n \int_U K(|\nabla w|)\Big( v^{s-1} \partial_i\partial_j w \partial_j w\Big)\cdot  \partial_i w \cdot (w-k) \chi_{\{v>0\}} dx.
\end{align*}
From this and \eqref{K-est-2}, it follows that 
\begin{align*}
I&\le a \int_U K(|\nabla w|)|\nabla^2 w|  v^s |w-k| dx  
+ \int_U K(|\nabla w|) v^s |\Delta w|  |w-k| dx\\
&\quad + s \int_U K(|\nabla w|) v^{s-1} |\nabla^2  w|  |\nabla w|^2  |w-k| \chi_{\{v>0\}} dx\\
& \le  C \int_U K(|\nabla w|) v^s |\nabla^2 w|  |w-k| dx
+ C\int_U K(|\nabla w|) |\nabla^2  w| v^{s-1} (v+M^2)  |w-k| dx.
\end{align*}
Hence, 
\begin{align*}
I
& \le  C \int_U K(|\nabla w|) v^s |\nabla^2 w|  |w-k| dx
+ CM^2\int_U K(|\nabla w|) |\nabla^2  w| v^{s-1}  |w-k| dx\\
& = C \int_U K(|\nabla w|) v^{\frac {s+1}2} \cdot |\nabla^2 w| v^{\frac{s-1}2}   |w-k| dx
+ C \int_U K(|\nabla w|) v^{s-1} |\nabla^2  w|  |w-k| \cdot  M^2 dx.
\end{align*}
This last inequality and Cauchy's inequality imply that 
\begin{align*}
I
& \le  \int_U K(|\nabla w|) (\frac{v^{s+1}}2+C |\nabla^2 w|^2 v^{s-1}   |w-k|^2) dx
+ C \int_U K(|\nabla w|) v^{s-1} (|\nabla^2  w|^2  |w-k|^2+ M^4) dx\\
& =  \frac{1}{2} I
  +C \int_U K(|\nabla w|)v^{s-1} |\nabla^2 w|^2 |w-k|^2 dx
+  CM^4 \int_U K(|\nabla w|)v^{s-1}  dx. 	
\end{align*}
Therefore, we obtain \eqref{LUineq}.
\end{proof}

The following is a generalization of the convergence of fast decay geometry sequences in Lemma 5.6, Chapter II of \cite{LadyParaBook68}. It will be used in the De Giorgi iterations.

\begin{lemma}[cf. \cite{HKP1}, Lemma A.2] \label{multiseq} 
Let $\{Y_i\}_{i=0}^\infty$ be a sequence of non-negative numbers satisfying
\beqs 
Y_{i+1}\le \sum_{k=1}^m A_k B^i  Y_i^{1+\mu_k}, \quad 
i =0,1,2,\cdots,
\eeqs
where  $B>1$, $A_k>0$  and $\mu_k>0$ for $k=1,2,\ldots,m$.
Let $\mu=\min\{\mu_k:1\le k\le m\}$.
If $Y_0\le \min\{ (m^{-1} A_k^{-1} B^{-\frac 1 {\mu}})^{1/\mu_k} : 1\le k\le  m\}$
then $\lim_{i\to\infty} Y_i=0$.
\end{lemma}


\section{Previous results}
\label{revision}
We study the following initial boundary value problem (IBVP) for $p(x,t)$:
\beq\label{p:eq}
\begin{cases}
\begin{displaystyle}
\frac{\partial p}{\partial t} 
\end{displaystyle}
 = \nabla \cdot (K (|\nabla p|)\nabla p  )& \text {in }  U\times (0,\infty),\\
p(x,0)=p_0(x) &\text {in } U,\\
p(x,t)=\psi(x,t)& \text{on } \Gamma \times(0,\infty).
\end{cases}
\eeq

In order to deal with the non-homogeneous boundary condition, the data $\psi(x,t)$ with $x\in\Gamma$ and $t>0$ is extended to a function $\Psi(x,t)$ with $x\in \bar U$ and $t\ge 0$. Throughout, our results are stated in terms of $\Psi$ instead of $\psi$. Nonetheless, corresponding results in terms of $\psi$ can be retrieved as performed in \cite{HI1}. The function $\Psi$ is always assumed to have adequate regularities for all calculations in this paper.

\textbf{Solutions.}
It is proved in section 3 of \cite{HIKS1} that \eqref{p:eq} possesses a weak solution $p(x,t)$ for all $t>0$. 
It, in fact, has more regularity in spatial and time variables, see \cite{DiDegenerateBook}.
For the current study, we assume that solution $p(x,t)$ has sufficient regularities both in $x$ and $t$ variables such that our calculations hereafter can be performed legitimately. 
Specifically, we assume 
$p,\nabla p, p_t\in C(\bar U\times(0,\infty))$, the Hessian matrix of second spatial derivatives $\nabla^2 p\in C(U\times (0,\infty))$,
and the function $t\to p(\cdot,t)$ is continuous from $[0,\infty)$ to $W^{1,2-a}(U)\cap L^\alpha(U)$, for an appropriate $\alpha>0$ which will be determined for each type of estimate. Some particular estimates may require much less regularity such as the $L^\infty$-estimates in section \ref{PSec}.
The obtained estimates also hold without the $C^2$-requirement by using an approximation process, see \cite{LadyParaBook68}. However, to avoid further complications, we do not perform such approximation here.

\textbf{Generic constants notation.}
Hereafter, the symbols $C$ and $C'$ are used to denote positive numbers independent of the initial and boundary data, and the time variables $t$, $T_0$, $T$;
it may depend on many parameters, namely, exponents and coefficients of polynomial  $g$,  the spatial dimension $n$ and domain $U$, other involved exponents $\alpha$, $s$, etc. in calculations. The values of $C$ and $C'$ may vary from place to place, even on the same line.

\textbf{Functions and Norms.} Throughout, whenever unspecified, the norm Lebesgue or Sobolev norms mean for the whole domain $U$. Also, for a function $f(x,t)$, we use $t(t)$ to denote the functions $t\to f(\cdot,t)$. For example, $\|p(t)\|_{L^2}=\|p(\cdot,t)\|_{L^2(U)}$, $\max_{[0,T]}\|p(t)\|_{L^\infty}=\max_{t\in[0,T]}\|p(\cdot,t)\|_{L^\infty(U)}$.

We recall some relevant results from \cite{HIKS1}. For $\alpha\ge 1$, we define
\beq
A(\alpha,t)=\Big[\int_U|\nabla \Psi(x,t)|^\frac{\alpha(2-a)}{2} dx\Big]^\frac {2(\alpha-a)}{\alpha(2-a)} +\Big[\int_U  |\Psi_t(x,t)|^\alpha dx\Big]^\frac{\alpha-a}{\alpha(1-a)}
\eeq
for $t\ge 0$, and
\beq A(\alpha) = \limsup_{t\to\infty} A(\alpha,t)\quad \text {and}\quad \beta(\alpha)= \limsup_{t\to\infty}[A'(\alpha,t)]^-.
\eeq
Also, define for $\alpha>0$ the number
\beq
\widehat \alpha
=\max\big\{\alpha,2,\alpha_*\big\}.
\eeq

Whenever $\beta(\alpha)$ is in use, it is understood that the function $t\to A(\alpha,t)$ belongs to $C^1((0,\infty))$.

For a function $f:[0,\infty)\to\R$, we denote by $Env f$ a continuous and increasing function $F:[0,\infty)\to\R$ such that $F(t)\ge f(t)$ for all $t\ge 0$.

Let $p(x,t)$ be a solution to IBVP \eqref{p:eq}, and denote $\bar{p} = p-\Psi$.

\begin{theorem}[cf. \cite{HIKS1}, Theorem 4.3]\label{HIKS43}
Let $\alpha>0$.

{\rm (i)} For all $t\ge 0$,
 \beq\label{pbar:ineq1}
 \int_U|\bar{p}(x,t)|^\alpha dx \le C\Big(1+ \int_U|\bar{p}(x,0)|^{\widehat \alpha} dx
+ [ Env A(\widehat\alpha,t)]^\frac{\widehat\alpha}{\widehat\alpha-a}\Big ).
 \eeq

{\rm (ii)} If $A(\widehat \alpha)<\infty$ then
\beq\label{limsupPbar}
  \limsup_{t\rightarrow\infty} \int_U |\bar{p}(x,t)|^\alpha dx \le C\big (1+A(\widehat{\alpha})^\frac {\widehat{\alpha}} {\widehat{\alpha}-a} \big).
\eeq

{\rm (iii)} If $\beta(\widehat \alpha)<\infty$ then there is $T>0$ such that
 \beq\label{pbar:ineq2}
 \int_U|\bar{p}(x,t)|^\alpha dx \le C\big (1+\beta(\widehat{\alpha})^\frac {\widehat{\alpha}}{\widehat{\alpha}-2a} + A(\widehat{\alpha},t)^\frac{\widehat{\alpha}}{\widehat{\alpha}-a} \big )\quad \text{for all }t\ge T.
 \eeq
 
\end{theorem}

For gradient and time derivative estimates, we denote
\begin{align*}
G_1(t)&=\int_U |\nabla \Psi(x,t)|^2 dx +\Big[\int_U |\Psi_t(x,t)|^{r_0} dx   \Big]^\frac{2-a}{r_0(1-a)} 
 +\Big[\int_U |\Psi_t(x,t)|^{r_0}dx  \Big]^\frac{1}{r_0},\\
G_2(t)&=\int_U|\nabla\Psi_t(x,t)|^2 dx+\int_U|\Psi_t(x,t)|^2dx,\quad G_3(t)=G_1(t)+G_2(t).
\end{align*}
with $r_0=\frac{n(2-a)}{(2-a)(n+1)-n}$.
For $t\ge 0$, recall from (4.20) in \cite{HIKS1} and from (3.25) in \cite{HI1} that
\begin{align}
\label{t0} 
\int_0^t \int_U H(|\nabla p|)dx d\tau&\le C \int_U \bar{p}^2(x,0) dx + C\int_0^tG_1(\tau)d\tau,\\
\label{intpt}
\int_U H(|\nabla p|)(x,t)dx+\int_0^t \int_U|\bar p_t(x,\tau)|^2 dx d\tau &\le \int_U \big[H(|\nabla p(x,0)|)+\bar p^2(x,0)\big] dx +C\int_0^t G_3(\tau)d\tau. 
\end{align}

Below are estimates for gradient and time derivative when time is large.

\begin{theorem}[cf. \cite{HK1}, Corollary 3.3]\label{UGcor}
Let $\alpha\ge \widehat 2$.

{\rm (i)} For $t\ge 1$,
\beq\label{all1}
\int_{t-1}^t \int_U H(|\nabla p(x,\tau)|)dx d\tau\le C\Big(1+ \int_U|\bar p_0(x)|^{\alpha} dx
+ [ Env A(\alpha,t)]^\frac{\alpha}{\alpha-a} + \int_{t-1}^tG_1(\tau)d\tau\Big ),
\eeq
\beq\label{all2}
 \int_{t-1/2}^t \int_U \bar{p}_t^2(x,\tau) dx d\tau\le C\Big(1+ \int_U|\bar{p}_0(x)|^{\alpha} dx
+ [ Env A(\alpha,t)]^\frac{\alpha}{\alpha-a} + \int_{t-1}^tG_3(\tau)d\tau\Big ).
\eeq

{\rm (ii)} If $A(\alpha)<\infty$ then
\beq\label{lim3}
  \limsup_{t\rightarrow\infty} \int_{t-1}^t \int_U H(|\nabla p(x,\tau)|)dx d\tau \le C\Big (1+A(\alpha)^\frac {\alpha} {\alpha-a} +\limsup_{t\to\infty} \int_{t-1}^tG_1(\tau)d\tau\Big),
\eeq
\beq\label{lim4}
  \limsup_{t\rightarrow\infty}  \int_{t-1/2}^t \int_U \bar{p}_t^2(x,\tau) dx d\tau \le C\Big (1+A(\alpha)^\frac {\alpha} {\alpha-a} +\limsup_{t\to\infty} \int_{t-1}^tG_3(\tau)d\tau\Big).
\eeq

{\rm (iii)} If $\beta(\alpha)<\infty$ then there is $T>0$ such that one has  for all $t\ge T$ that 
  \beq\label{large3}
 \int_{t-1}^t \int_U H(|\nabla p(x,\tau)|)dx d\tau \le C\Big (1+\beta(\alpha)^\frac {\alpha}{\alpha-2a} + A(\alpha,t-1)^\frac{\alpha}{\alpha-a} +\int_{t-1}^tG_1(\tau)d\tau\Big ),
 \eeq
 \beq\label{large4}
\int_{t-1/2}^t \int_U \bar{p}_t^2(x,\tau) dx d\tau \le C\Big (1+\beta(\alpha)^\frac {\alpha}{\alpha-2a} + A(\alpha,t-1)^\frac{\alpha}{\alpha-a} +\int_{t-1}^tG_3(\tau)d\tau\Big ).
 \eeq
\end{theorem}


\section{$L^\infty$-estimates}
\label{PSec}

In this section we estimate the $L^\infty$-norm of the pressure. We will focus on estimates for $\bar p$ instead of $p$.
First, we give a local (in time) $L^\infty$-estimate which does not depend on the initial data's $L^\infty$-norm.

Let $p(x,t)$ be a solution to IBVP \eqref{p:eq} with given data $p_0(x)$ and $\psi(x,t)$.
Let $\bar p = p-\Psi$, then it satisfies
\beq\label{sys3}
\begin{cases}
\begin{displaystyle}
\frac{\partial \bar p}{\partial t}  
\end{displaystyle}
= \nabla  \cdot (K (|\nabla p|)\nabla p  )-\Psi_t &\text {in }  U\times (0,\infty),\\
\bar p(x,t)=0 &\text{on } \Gamma \times(0,\infty).
\end{cases}
\eeq

\begin{theorem}\label{Ntheo42}
Let $\alpha$ satisfy \eqref{alcond}. If $T_0\ge 0$, $T>0$ and $\theta\in(0,1)$ then 
\begin{multline}\label{Ne1}
\sup_{[T_0+\theta T,T_0+T]}\norm{\bar p(t)}_{L^\infty(U)}
\le C \Big\{ 
(\theta T)^{-\frac1{\delta_1}}\| \bar p\|_{L^\alpha(U\times (T_0,T_0+T)  )}
+ T^\frac{\alpha-2}{2\alpha(1+\delta_1)} \| \bar p\|_{L^\alpha(U\times (T_0,T_0+T)  )}^\frac{\delta_1}{1+\delta_1}\\
+ (\theta T)^{-\frac{1}{(\alpha-a)\delta_1}}\| \bar p\|_{L^\alpha(U\times (T_0,T_0+T)  )}^{\frac{\delta_3}{\delta_1}}
+ T^\frac{\alpha-2}{2(\alpha-a)(1+\delta_2)}\| \bar p\|_{L^\alpha(U\times (T_0,T_0+T)  )}^\frac{\delta_3}{1+\delta_2}\\
+ (T^\frac{\alpha-2}{2\alpha} \|\nabla \Psi\|_{L^{\alpha q_1}(U\times(T_0,T_0+T))} 
+T^\frac{\alpha-1}{\alpha} \| \Psi_t\|_{L^{\alpha q_1}(U\times(T_0,T_0+T))})^\frac{1}{1+\delta_3}  \| \bar p\|_{L^\alpha(U\times (T_0,T_0+T)  )}^\frac{\delta_2}{1+\delta_3}\\
+ (T^\frac{\alpha-2}{2\alpha} \|\nabla \Psi\|_{L^{\alpha q_1}(U\times(T_0,T_0+T))} +T^\frac{\alpha-1}{\alpha} \| \Psi_t\|_{L^{\alpha q_1}(U\times(T_0,T_0+T))})^\frac\alpha{(\alpha-a)(1+\delta_4)}    \| \bar p\|_{L^\alpha(U\times (T_0,T_0+T)  )}^\frac{\delta_4}{1+\delta_4}
\Big\}.
\end{multline}
where  $r_1=\alpha(1+(2-a)/n)-a,$  numbers $p_1$ and $q_1$ are conjugates of each other that satisfy 
\beqs
 1\le p_1<r_1/\alpha,\text{ or equivalently, } 1+\frac{1}{a(1/\alpha_*-1/\alpha)}<q_1=\frac{p_1}{p_1-1}\le\infty,
\eeqs
and $\delta_1$, $\delta_2$, $\delta_3$, $\delta_4$ are positive numbers defined by
\beq\label{deltadef}
\begin{aligned}
\delta_1&=1-\alpha/r_1,\quad &&\delta_2= \alpha/(\alpha-a) -\alpha/r_1, \\
\delta_3&= 1/p_1 - \alpha/r_1,\quad  &&\delta _4 = \alpha/((\alpha-a)p_1) - \alpha/r_1,
  \end {aligned}
\eeq
\end{theorem}
\begin{proof} Without loss of generality, we assume $T_0=0$. 
Let $k\ge 0$, define $\bar p^{(k)} =\max\{\bar p-k,0\}$, and denote by $\chi _k$ the characteristic function on the set ${\rm supp}\ \bar p^{(k)}$ - the support set of $\bar p^{(k)}$.   

Let $\zeta=\zeta(t)$ be a smooth function on $\R$ satisfying $0\le \zeta\le 1$ and $\zeta_t\ge 0$ on $\R$, and $\zeta=0$ on $(-\infty,0]$. 

Multiplying the partial differential equation (PDE) in \eqref{p:eq} by $|\bar p^{(k)}|^{\alpha-1} \zeta$, integrating over $U$ and using integration by parts, we have
\begin{align*}
& \frac1\alpha  \int_U \frac{\partial |\bar p^{(k)}|^\alpha}{\partial t} \zeta dx +(\alpha-1)\int_U K(|\nabla p|) (\nabla p\cdot \nabla\bar  p^{(k)}) |\bar p^{(k)}|^{\alpha-2} \zeta dx
= - \int_U \Psi_t |\bar p^{(k)}|^{\alpha-1} \zeta dx.
\end{align*}

For the second integral on the left-hand side,
\begin{align*}
& K(|\nabla p|) (\nabla p\cdot \nabla\bar  p^{(k)}) |\bar p^{(k)}|^{\alpha-2} 
= K(|\nabla p|) (\nabla \bar p\cdot \nabla\bar  p^{(k)}) |\bar p^{(k)}|^{\alpha-2} +K(|\nabla p|) (\nabla \Psi\cdot \nabla\bar  p^{(k)}) |\bar p^{(k)}|^{\alpha-2}.
\end{align*}

Note that $\nabla \bar p\cdot \nabla\bar  p^{(k)}=|\nabla\bar  p^{(k)}|^2$ and $|\nabla \Psi\cdot \nabla\bar  p^{(k)}|\le \frac12|\nabla \Psi |^2+\frac12 |\nabla\bar  p^{(k)}|^2$.
Then
\begin{align*}
& K(|\nabla p|) (\nabla p\cdot \nabla\bar  p^{(k)}) |p^{(k)}|^{\alpha-2} 
\ge K(|\nabla p|)|\bar p^{(k)}|^{\alpha-2}( |\nabla\bar  p^{(k)}|^2  - \frac12|\nabla \Psi |^2-\frac12 |\nabla\bar  p^{(k)}|^2 )\\
&=\frac12K(|\nabla p|) |\nabla\bar  p^{(k)}|^2 |\bar p^{(k)}|^{\alpha-2} - \frac12 K(|\nabla p|) |\nabla \Psi|^2 |\bar p^{(k)}|^{\alpha-2}.
\end{align*}

Therefore,
\beq\label{add2}
\begin{aligned}
& \frac1\alpha  \int_U \frac{\partial |\bar p^{(k)}|^\alpha}{\partial t} \zeta dx +\frac{\alpha-1}2 \int_U K(|\nabla p|) |\nabla\bar  p^{(k)}|^2 |\bar p^{(k)}|^{\alpha-2} \zeta dx\\
&\le  \frac{\alpha-1}2 \int_U K(|\nabla p|) |\nabla \Psi|^2 |\bar 
p^{(k)}|^{\alpha-2}\zeta dx + \int_U |\Psi_t| |\bar p^{(k)}|^{\alpha-1} \zeta 
dx\\
&\le  C \int_U  |\nabla \Psi|^2 |\bar p^{(k)}|^{\alpha-2}\zeta dx + \int_U 
|\Psi_t| |\bar p^{(k)}|^{\alpha-1} \zeta dx.
\end{aligned}
\eeq
The last inequality uses the fact that  $K(\xi)$ is bounded.

For the second integral on the left-hand side again, we have from \eqref{K-est-3}  and triangle inequality that
\begin{align*}
K(|\nabla p|) |\nabla\bar  p^{(k)}|^2
&\ge \frac{d_1  |\nabla\bar  p^{(k)}|^2}{(1+|\nabla p|)^a}
\ge \frac{d_1  |\nabla\bar  p^{(k)}|^2}{(1+|\nabla \bar p|+|\nabla \Psi|)^a}
=  \frac{d_1  |\nabla\bar  p^{(k)}|^2}{(1+|\nabla \bar p^{(k)}|+|\nabla \Psi|)^a}\\
&\ge \frac{(d_1/2) ( |\nabla\bar  p^{(k)}|+|\nabla \Psi|)^2-d_1|\nabla \Psi|^2}{(1+|\nabla \bar p^{(k)}|+|\nabla \Psi|)^a}\\
&\ge \frac{d_1}2\frac{( |\nabla\bar  p^{(k)}|+|\nabla \Psi|)^2}{(1+|\nabla \bar p^{(k)}|+|\nabla \Psi|)^a}-d_1|\nabla \Psi|^{2-a}.
\end{align*}

It is elementary to show for $\xi\ge0$ that
\beqs
\frac{\xi^2}{(1+\xi)^a}\ge \frac1{2^a}(\xi^{2-a}-1).
\eeqs

Thus,
\begin{align*}
K(|\nabla p|) |\nabla\bar  p^{(k)}|^2
&\ge \frac{d_1}{2^{a+1}}((|\nabla\bar  p^{(k)}|+|\nabla \Psi|)^{2-a}-1) - d_1|\nabla \Psi|^{2-a}\\
&\ge \frac{d_1}{2^{a+1}}|\nabla\bar  p^{(k)}|^{2-a}-\frac{d_1}{2^{a+1}} - d_1|\nabla \Psi|^{2-a}.
\end{align*}

Therefore, it follows this, \eqref{add2},  and Young's inequality that
\begin{align*}
& \int_U \frac{\partial |\bar p^{(k)}|^\alpha}{\partial t} \zeta dx +\int_U |\nabla\bar  p^{(k)}|^{2-a} |\bar p^{(k)}|^{\alpha-2} \zeta dx\\
&\le C\int_U (1+|\nabla \Psi|^{2-a} +|\nabla \Psi|^2) |\bar 
p^{(k)}|^{\alpha-2}\zeta dx + C \int_U |\Psi_t| |\bar p^{(k)}|^{\alpha-1} \zeta 
dx\\
&\le C\int_U (1+|\nabla \Psi|^2) |\bar p^{(k)}|^{\alpha-2}\zeta dx + C\int_U 
|\Psi_t| |\bar p^{(k)}|^{\alpha-1} \zeta dx.
\end{align*}

Applying the product rule to the first term, we have
\begin{align*}
&\ddt \int_U |\bar p^{(k)}|^\alpha \zeta dx+ \int_U |\nabla \bar p^{(k)}|^{2-a} |\bar p^{(k)}|^{\alpha-2} \zeta dx \\
&\le C\int_U |\bar p^{(k)}|^\alpha \zeta_t dx + C\int_U ((1  +|\nabla \Psi|^2) |\bar p^{(k)}|^{\alpha-2}\zeta +  |\Psi_t| |\bar p^{(k)}|^{\alpha-1} \zeta)dx.
\end{align*}
Applying Young's inequality to the integrand of the last integral yields
\begin{align*}
&\ddt \int_U |\bar p^{(k)}|^\alpha \zeta dx+ \int_U |\nabla \bar p^{(k)}|^{2-a} |\bar p^{(k)}|^{\alpha-2} \zeta dx 
\le C\int_U |\bar p^{(k)}|^\alpha \zeta_t dx \\
&+ \varep\int_U |\bar p^{(k)}|^\alpha \zeta  dx + C\varep^{1-\alpha/2}\int_U (1+|\nabla \Psi|^\alpha)\chi_k\zeta dx dt + C\varep^{1-\alpha}  \int_U |\Psi_t|^\alpha \chi_k\zeta dxdt. 
\end{align*}
Integrating form $0$ to $t$ for $t\in[0,T]$ and taking the supremum give
\begin{align*}
 &\sup_{[0,T]} \int_U |\bar p^{(k)}|^\alpha \zeta dx+\int_0^T\int_U |\nabla \bar p^{(k)}|^{2-a} |\bar p^{(k)}|^{\alpha-2} \zeta dx dt \le C \int_0^T\int_U |\bar p^{(k)}|^\alpha \zeta_t dx dt\\
&+ \varep T \sup_{[0,T]} \int_U |\bar p^{(k)}|^\alpha \zeta  dx  +C\varep^{1-\alpha/2}\int_0^T \int_U (1+|\nabla \Psi|^\alpha)\chi_k\zeta dx dt
 + C\varep^{1-\alpha}  \int_0^T\int_U |\Psi_t|^\alpha \chi_k\zeta dxdt.
\end{align*}
Selecting $\varep= 1/(2 T)$, we obtain
\begin{multline*}
  \sup_{[0,T]} \int_U |\bar p^{(k)}|^\alpha \zeta dx   + \int_0^T\int_U |\nabla \bar p^{(k)}|^{2-a} |\bar p^{(k)}|^{\alpha-2} \zeta dx dt \le  C\int_0^T\int_U |\bar p^{(k)}|^\alpha \zeta_t dx dt\\
 +C T^{\alpha/2-1}\int_0^T \int_U \chi_k\zeta dx dt
 +C T^{\alpha/2-1}\int_0^T \int_U |\nabla \Psi|^\alpha \chi_k\zeta dx dt
+ C T^{\alpha-1} \int_0^T\int_U |\Psi_t|^\alpha \chi_k\zeta dx dt.
\end{multline*}
Applying H\"older's inequality to the last two integrals gives
\beq\label{prena}
\begin{aligned}
&  \sup_{[0,T]} \int_U |\bar p^{(k)}|^\alpha \zeta dx   + \int_0^T\int_U |\nabla \bar p^{(k)}|^{2-a} |\bar p^{(k)}|^{\alpha-2} \zeta dx dt \le  C\int_0^T\int_U |\bar p^{(k)}|^\alpha \zeta_t dx dt\\
 & +C T^{\alpha/2-1}\int_0^T \int_U \chi_k\zeta dx dt\\
&+ C( T^{\alpha/2-1}\| |\nabla \Psi|^{\alpha}\|_{L^{q_1}(U\times(0,T))} +T^{\alpha-1}\| |\Psi_t|^{\alpha}\|_{L^{q_1}(U\times(0,T))}) \Big(\int_0^T\int_U  \chi_k\zeta dx dt\Big)^{1/p_1}.
\end{aligned}
\eeq

Let $M_0 >0$ be fixed which will be determined later.
For $i\ge 0$, define 
\beq\label{indexSeq2}
k_i= M_0(1-2^{-i}),\quad
t_i =\theta T( 1- 2^{-i}).
\eeq 
Then 
$t_0=0<t_1<\ldots<\theta T$ and
$\lim_{i\to\infty}t_i=\theta T$; $k_0=0<k_1<k_2<\ldots$ and $\lim_{i\to\infty}k_i=M_0$.
  
For $i,j\ge  0$, we denote 
\beq\label{setDef2} 
\mathcal Q_i =U\times (t_i,T)\quad\text { and } \quad A_{i,j} =\{(x,t)\in \mathcal Q_j:  p(x,t)>k_i\}, \quad  A_i=A_{i,i}.
\eeq

For each $i$, we use a cut-off function $\zeta_i(t)$ with $\zeta_i\equiv 1$ in $[t_{i+1},T]$ and 
$\zeta_i\equiv 0$ on $[0,t_{i}]$, and 
\beq\label{cutoffBound2}
|(\zeta_i)_t|\le \frac C{t_{i+1}-t_i} = \frac {C2^{i+1}}{\theta T}.\eeq
for some $C>0$. 
Applying \eqref{prena} with $k=k_{i+1}$ and $\zeta=\zeta_i$ gives
\beq\label{pren1}
\begin{aligned}
&  \sup_{[0,T]} \int_U |\bar p^{(k_{i+1})}|^\alpha \zeta_i dx   + \int_0^T\int_U |\nabla \bar p^{(k_{i+1})}|^{2-a}|\bar p^{(k_{i+1})}|^{\alpha-2} \zeta_i dx dt\\
&\le  \int_0^T\int_U |\bar p^{(k_{i+1})}|^\alpha \zeta_{it} dxdt
 + C T^{\alpha/2-1}\int_0^T \int_U \chi_{k_{i+1}}\zeta_i dx dt\\
&\quad + C (T^{\alpha/2-1} \|\nabla \Psi\|^{\alpha}_{L^{\alpha q_1}(U\times(0,T))}+ T^{\alpha-1} \| \Psi_t\|^{\alpha}_{L^{\alpha q_1}(U\times(0,T))}) \Big(\int_0^T\int_U  \chi_{k_{i+1}}\zeta_i dx dt\Big)^{1/p_1}.
\end{aligned}
\eeq
Define 
\beqs
F_i\eqdef   \sup_{[t_{i+1},T]} \int_U |\bar p^{(k_{i+1})}|^\alpha  dx   + \int_{t_i}^T\int_U |\nabla \bar p^{(k_{i+1})}|^{2-a}|\bar p^{(k_{i+1})}|^{\alpha-2}  dx dt.
\eeqs
Let
$$\mathcal E_1 = T^{\alpha/2-1}\quad \text{and}\quad \mathcal E_2= T^{\alpha/2-1} \|\nabla \Psi\|^{\alpha}_{L^{\alpha q_1}(U\times(0,T))}+T^{\alpha-1}\| \Psi_t\|^{\alpha}_{L^{\alpha q_1}(U\times(0,T))}.$$
Then \eqref{pren1} yields   
\beq\label{pren2}
\begin{aligned}
F_i &\le  \int_{t_i}^{t_{i+1}} \int_U |\bar p^{(k_{i+1})}|^\alpha \zeta_{it} dxdt
 +C\mathcal E_1\int_{t_{i+1}}^T \int_U \chi_{k_{i+1}} dx dt + C \mathcal E_2\Big(\int_{t_{i+1}}^T\int_U  \chi_{k_{i+1}} dx dt\Big)^{1/p_1}.\\
 &\le C2^i(\theta T)^{-1}  \|\bar p^{(k_{i+1})}\|_{L^\alpha(A_{i+1,i})}^\alpha
 +C\mathcal  E_1 |A_{i+1,i}|+C\mathcal  E_2 |A_{i+1,i}|^{1/p_1}.
\end{aligned}
\eeq
Since
$
 \|\bar  p^{(k_i)}\|_{L^\alpha(A_i)}\ge  \|\bar  p^{(k_i)}\|_{L^\alpha(A_{i+1,i})}\ge (k_{i+1}-k_i) |A_{i+1,i}|^{1/\alpha}, 
$
thus
\beq\label{a2}
  |A_{i+1,i}| \le (k_{i+1}-k_i)^{-\alpha} \|\bar  p^{(k_i)}\|_{L^\alpha(A_i)}^{\alpha} \le C 2^{\alpha i} M_0^{-\alpha}\|\bar  p^{(k_i)}\|_{L^\alpha(A_i)}^{\alpha}.
\eeq
This and \eqref{pren2} imply
\begin{multline}\label{Fn}
F_i \le  C2^i(\theta T)^{-1}  \|\bar p^{(k_{i+1})}\|_{L^\alpha(A_{i+1,i})}^\alpha
 +C\mathcal E_1  2^{\alpha i} M_0^{-\alpha}\| \bar p^{(k_i)}\|_{L^\alpha(A_i)}^{\alpha}+C\mathcal E_2  2^{\alpha i/p_1} M_0^{-\alpha/p_1}\| \bar p^{(k_i)}\|_{L^\alpha(A_i)}^{\alpha/p_1}\\
 \le   C 2^{\alpha i}\Big\{ \Big((\theta T)^{-1}+\mathcal E_1 M_0^{-\alpha}\Big) \|\bar  p^{(k_i)}\|_{L^\alpha(A_i)}^{\alpha}+ \mathcal E_2 M_0^{-\alpha/p_1}\| \bar p^{(k_i)}\|_{L^\alpha(A_i)}^{\alpha/p_1}\Big\}.
\end{multline}

Note that $r_1$ is the exponent defined in \eqref{expndef}. By Lemma~\ref{Sob4},
\beq\label{a1}
 \| \bar p^{(k_{i+1})}\|_{L^{r_1}(A_{i+1,i+1})}
\le C (F_i^{1/\alpha}+F_i^{1/(\alpha-a)}).
\eeq
H\"older's inequality gives
\beq\label{h1}
\|\bar  p^{(k_{i+1})}\|_{L^\alpha(A_{i+1,i+1})}
 \le \|\bar  p^{(k_{i+1})}\|_{L^{r_1}(A_{i+1,i+1})} |A_{i+1,i+1}|^{1/\alpha-1/r_1}
 \le \|\bar  p^{(k_{i+1})}\|_{L^{r_1}(A_{i+1,i+1})} |A_{i+1,i}|^{1/\alpha-1/r_1}.
\eeq
Combining \eqref{h1} with  \eqref{a2}, \eqref{Fn} and \eqref{a1} yields
\begin{align*}
 &\|\bar  p^{(k_{i+1})}\|_{L^\alpha(A_{i+1,i+1})}
\le C  B^i \Big \{ \Big((\theta T)^{-1}+\mathcal E_1 M_0^{-\alpha}\Big)^{1/\alpha} \|\bar  p^{(k_i)}\|_{L^\alpha(A_i)} + \mathcal E_2^{1/\alpha} M_0^{-1/p_1}\| \bar p^{(k_i)}\|_{L^\alpha(A_i)}^{1/p_1}\\
&\quad +\Big((\theta T)^{-1}+\mathcal E_1 M_0^{-\alpha}\Big)^{1/(\alpha-a)} \|\bar  p^{(k_i)}\|_{L^\alpha(A_i)}^{\alpha/(\alpha-a)}+ \mathcal E_2^{1/(\alpha-a)} M_0^{-\alpha/((\alpha-a)p_1)}\| \bar p^{(k_i)}\|_{L^\alpha(A_i)}^{\alpha/((\alpha-a)p_1)}
\Big\} \\
&\quad \cdot M_0^{-1+\alpha/r_1} \| p^{(k_i)}\|_{L^2(A_i)}^{1-\alpha/r_1},
\end{align*}
where $B=2^\alpha$.
Let $ Y_i=\|\bar  p^{(k_i)}\|_{L^\alpha(A_i)}.$ Then 
\beqs
Y_{i+1}\le  C B^i \Big (D_1 Y_i^{1+\delta_1}+D_2Y_i^{1+\delta_2} +D_3 Y_i^{1+\delta_3}+D_4 Y_i^{1+\delta_4}\Big ),
\eeqs
where 
\begin{align*}
D_1&= \big((\theta T)^{-1/\alpha}+\mathcal E_1^{1/\alpha} M_0^{-1}\big)\cdot M_0^{-1+\alpha/r_1}
    =(\theta T)^{-1}M_0^{-\delta_1} + T^{1/2-1/\alpha} M_0^{-1-\delta_1},\\
D_2&= \mathcal E_2^{1/\alpha} M_0^{-1/p_1}\cdot M_0^{-1+\alpha/r_1}
    = \mathcal E_2^{1/\alpha} M_0^{-1/p_1-\delta_1},\\
D_3&=\big((\theta T)^{-1/(\alpha-a)}+\mathcal E_1^{1/(\alpha-a)} M_0^{-\alpha/(\alpha-a)}\big)\cdot M_0^{-1+\alpha/r_1}
    =(\theta T)^{-\frac{1}{\alpha-a}}M_0^{-\delta_1} + T^\frac{\alpha-2}{2(\alpha-a)} M_0^{-\frac{\alpha}{\alpha-a}-\delta_1},\\
D_4&=\mathcal E_2^{1/(\alpha-a)} M_0^{-\alpha/((\alpha-a)p_1)}\cdot M_0^{-1+\alpha/r_1}
=\mathcal E_2^{1/(\alpha-a)} M_0^{-\frac\alpha{(\alpha-a)p_1}-\delta_1}.
\end{align*}

Take $M_0$ sufficiently large such that 
\beq\label{Y1}
Y_0\le C\min\{ D_1^{-1/\delta_1},D_2^{-1/\delta_2},D_3^{-1/\delta_3},D_4^{-1/\delta_4}\},
\eeq
or equivalently,
\beq\label{Y2}
D_j\le CY_0^{-\delta_j},\quad j=1,2,3,4. 
\eeq

Then by Lemma \ref{multiseq}, $\lim_{i\to\infty} Y_i=0$, consequently,
$\int_{\theta T}^T\int_U |\bar p^{(M_0)}|^\alpha dxdt=0,$
that is
\beqs
\bar p(x,t)\le M_0 \quad  \text{in } U\times [\theta T,T]. 
\eeqs
Repeating the argument above for $-p$,$-\psi$ instead of $p$, $\psi$, we obtain 
\beq\label{abspM} |\bar p(x,t)|\le M_0 \quad \text{in }U\times [\theta T,T].\eeq 

It remains to determine $M_0$. Since $Y_0\le \| \bar p\|_{L^\alpha(U\times (0,T)  )}$, we have sufficient conditions for \eqref{Y2}:
\begin{align*}
&(\theta T)^{-1}M_0^{-\delta_1}\le C \| \bar p\|_{L^\alpha(U\times (0,T)  )}^{-\delta_1},  
\quad T^{1/2-1/\alpha} M_0^{-1-\delta_1}\le C\| \bar p\|_{L^\alpha(U\times (0,T)  )}^{-\delta_1},\\
&(\theta T)^{-\frac{1}{\alpha-a}}M_0^{-\delta_1}\le C\| \bar p\|_{L^\alpha(U\times (0,T)  )}^{-\delta_3},\quad
T^\frac{\alpha-2}{2(\alpha-a)} M_0^{-\frac{\alpha}{\alpha-a}-\delta_1}\le C\| \bar p\|_{L^\alpha(U\times (0,T)  )}^{-\delta_3},\\
&\mathcal E_2^{1/\alpha} M_0^{-1/p_1-\delta_1}\le C\| \bar p\|_{L^\alpha(U\times (0,T)  )}^{-\delta_2},\quad 
\mathcal E_2^{1/(\alpha-a)} M_0^{-\frac\alpha{(\alpha-a)p_1}-\delta_1}
\le C\| \bar p\|_{L^\alpha(U\times (0,T)  )}^{-\delta_4}.
\end{align*}
Solving these inequalities gives
\begin{multline}\label{Mineq}
M_0
\ge C \max \Big\{ 
(\theta T)^{-1/\delta_1}\| \bar p\|_{L^\alpha(U\times (0,T)  )},\ 
T^\frac{\alpha-2}{2(1+\delta_1)\alpha} \| \bar p\|_{L^\alpha(U\times (0,T)  )}^\frac{\delta_1}{1+\delta_1},\
(\theta T)^{-\frac{1}{(\alpha-a)\delta_1}}\| \bar p\|_{L^\alpha(U\times (0,T)  )}^\frac{\delta_3}{\delta_1},\\
T^\frac{\alpha-2}{2(\alpha+\delta_1(\alpha-a))}\| \bar p\|_{L^\alpha(U\times (0,T)  )}^\frac{\delta_3(\alpha-a)}{\alpha+\delta_1(\alpha-a)},\
\mathcal E_2^\frac{1}{\alpha(1/p_1+\delta_1)} \| \bar p\|_{L^\alpha(U\times (0,T)  )}^\frac{\delta_2}{1/p_1+\delta_1},\
\mathcal E_2^\frac{p_1}{\alpha+\delta_1(\alpha-a)p_1}  \| \bar p\|_{L^\alpha(U\times (0,T)  )}^\frac{\delta_4(\alpha-a)p_1}{\alpha+\delta_1(\alpha-a)p_1}
\Big\}
\end{multline}
with an appropriate positive constant $C$. Choosing $M_0$ to be the right-hand side of \eqref{Mineq} with the sum replacing the maximum, we obtain from \eqref{abspM} that
\begin{multline}\label{Ne20}
\sup_{[\theta T,T]}\norm{\bar p(t)}_{L^\infty(U)}
\le C \Big\{ 
(\theta T)^{-1/\delta_1}\| \bar p\|_{L^\alpha(U\times (0,T)  )}
+ T^\frac{\alpha-2}{2(1+\delta_1)\alpha} \| \bar p\|_{L^\alpha(U\times (0,0+T)  )}^\frac{\delta_1}{1+\delta_1}\\
+ (\theta T)^{-\frac{1}{(\alpha-a)\delta_1}}\| \bar p\|_{L^\alpha(U\times (0,T)  )}^\frac{\delta_3}{\delta_1}
+ T^\frac{\alpha-2}{2(\alpha+\delta_1(\alpha-a))}\| \bar p\|_{L^\alpha(U\times (0,T)  )}^\frac{\delta_3(\alpha-a)}{\alpha+\delta_1(\alpha-a)}\\
(\mathcal E_2^{1/\alpha})^\frac{1}{1/p_1+\delta_1} \| \bar p\|_{L^\alpha(U\times (0,T)  )}^\frac{\delta_2}{1/p_1+\delta_1}+
(\mathcal E_2^{1/\alpha})^\frac{\alpha }{\alpha/p_1+\delta_1(\alpha-a)}  \| \bar p\|_{L^\alpha(U\times (0,T)  )}^\frac{\delta_4(\alpha-a)p_1}{\alpha+\delta_1(\alpha-a)p_1}
\Big\}.
\end{multline}
Rewriting the powers in \eqref{Ne20}, noticing that
\beqs
\frac1{p_1}+\delta_1=1+\delta_3,\quad
\frac{\alpha}{\alpha-a}+\delta_1=1+\delta_2,\quad
\frac{\alpha}{(\alpha-a)p_1}+\delta_1=1+\delta_4,
\eeqs
and also
\beqs
\mathcal E_2^{1/\alpha}\le C (T^\frac{\alpha-2}{2\alpha} \|\nabla \Psi\|_{L^{\alpha q_1}(U\times(T_0,T_0+T))} +T^\frac{\alpha-1}{\alpha} \| \Psi_t\|_{L^{\alpha q_1}(U\times(T_0,T_0+T))}),
\eeqs
we obtain \eqref{Ne1}. The proof is complete.
\end{proof}

\begin{remark}
In \eqref{Ne1}, the norm $\|\bar p\|_{L^\infty(U\times[T_0+\theta T,T_0+T])}$ is estimated by a sum of homogeneous terms in $\|\bar p\|_{L^\alpha(U\times(T_0,T_0+T))}$. Therefore it is appropriate for both small and large $\|\bar p\|_{L^\alpha(U\times(T_0,T_0+T))}$. This improves our previous interior versions in Theorem 4.1 of \cite{HK1} and Proposition 3.2 of \cite{HKP1}, where the bounds  contain an additive constant, which make them more suitable for large value $\|\bar p\|_{L^\alpha(U\times (T_0,T_0+T)}$. For doubly nonlinear parabolic equations of p-Laplacian type, the interior estimate of this kind was obtained in \cite{Surnachev2012} using Moser's iteration.
\end{remark}

To simplify our future estimates, we use the following  weaker version of Theorem \ref{Ntheo42}.

\begin{corollary}\label{Nsim}
Let $\alpha$, $T_0$, $T$, $\theta$, $r_1$, $p_1$, $q_1$, $\delta_1$, $\delta_2$, $\delta_3$, $\delta_4$ be as in Theorem \ref{Ntheo42}.
Then
\begin{multline}\label{esim}
\sup_{[T_0+\theta T,T_0+T]}\norm{\bar p(t)}_{L^\infty(U)}
\le C \Big(1+ (\theta T)^{-1/\delta_1}+T^{z_1}\Big)\\
\cdot \Big( 1+ \| \nabla \Psi\|_{L^{\alpha q_1}(U\times(T_0,T_0+T))}+ \| \Psi_t\|_{L^{\alpha q_1}(U\times(T_0,T_0+T))}\Big)^{z_2}\\
\cdot \Big(1+   \| \bar p\|_{L^\alpha(U\times (T_0,T_0+T)  )}\Big)^{z_3},
\end{multline}
where
\beq\label{z123}
z_1=\frac{\alpha-1}{(\alpha-a)(1+\delta_4)},\
z_2=\frac\alpha{(\alpha-a)(1+\delta_4)},\
z_3=\max\Big\{1,\frac{\delta_2}{1+\delta_3}\Big\}.
\eeq
\end{corollary}
\begin{proof}
We apply Young's inequality in \eqref{Ne1} to the terms involving $(\theta T)^{-1}$, $T$, $\|\nabla \Psi\|_{L^{\alpha q_1}}$, $\|\Psi_t\|_{L^{\alpha q_1}}$, $\|\bar p\|_{L^\alpha}$.

First, we note that $\delta_2>\delta_1\ge \delta_3$ and $\delta_2\ge \delta_4>\delta_3$.

$\bullet$ For the power of $(\theta T)^{-1}$, we have $1/\delta_1>1/((\alpha-a)\delta_1)$.

$\bullet$ The largest power of $T$ is
\begin{align*}
z_1=\max\Big\{\frac{\alpha-2}{2\alpha(1+\delta_1)},
\frac{\alpha-2}{2(\alpha-a)(1+\delta_2)},
\frac{\alpha-1}{\alpha (1+\delta_3)} ,
\frac{\alpha-1}{(\alpha-a)(1+\delta_4)}
\Big\}.
\end{align*}
Note that
\begin{align*}
&\frac{\alpha-2}{2\alpha(1+\delta_1)} <\frac{\alpha-2}{2(\alpha+\delta_1(\alpha-a))}=\frac{\alpha-2}{2(\alpha-a)(1+\delta_2)}
\le \frac{\alpha-1}{(\alpha-a)(1+\delta_4)},\\
&\frac{\alpha-1}{\alpha (1+\delta_3)} =\frac{\alpha-1}{\alpha (1/p_1+\delta_1)} 
<\frac{\alpha-1}{(\alpha-a)(\frac{\alpha}{(\alpha-a)p_1}+\delta_1)}=\frac{\alpha-1}{(\alpha-a)(1+\delta_4)}.
\end{align*}
Hence $z_1$ is as in \eqref{z123}.

$\bullet$ The largest power of $\|\nabla \Psi\|_{L^{\alpha q_1}}$  and $\|\Psi_t\|_{L^{\alpha q_1}}$ is
\beqs
z_2=\max\Big\{\frac{1}{1+\delta_3} ,
\frac\alpha{(\alpha-a)(1+\delta_4)}\Big\}.
\eeqs
Since
$\frac{1}{1+\delta_3}=\frac{1}{1/p_1+\delta_1} < \frac\alpha{(\alpha-a)(\frac{\alpha}{(\alpha-a)p_1}+\delta_1)}=\frac\alpha{(\alpha-a)(1+\delta_4)},
$
we have $z_2$ as in \eqref{z123}.

$\bullet$ The largest power of $\|\bar p\|_{L^\alpha}$ is
\beqs
z_3=\Big\{1,
\frac{\delta_1}{1+\delta_1},
\frac{\delta_2}{1+\delta_3},
\frac{\delta_3}{\delta_1},
\frac{\delta_3}{1+\delta_2},
\frac{\delta_4}{1+\delta_4}\Big\}.
\eeqs
Note that
$\frac{\delta_1}{1+\delta_1},\frac{\delta_3}{\delta_1}, \frac{\delta_3}{1+\delta_2},\frac{\delta_4 }{\delta_4+1}<1.$
Hence $z_3$ is as in \eqref{z123}.
Therefore, we obtain \eqref{esim} from \eqref{Ne1}.
\end{proof}

We then derive the $L^\infty$-estimates in terms of the problem's data.

\begin{theorem}\label{Ntheo43} 
Assume $\alpha$ satisfies \eqref{alcond}. Let $q_1$, $\delta_1$ be defined as in Theorem \ref{Ntheo42}, and $z_1$, $z_2$, $z_3$ be defined as in Corollary \ref{Nsim}.

{\rm (i)} If $0<t\le 3$ then
\beq\label{Ne2}
\norm{\bar p(t)}_{L^\infty(U)}
\le C t^{-\frac{1}{\delta_1}}(1 + \|\bar{p}_0\|_{L^\alpha(U)} + [ Env A(\alpha,t)]^\frac1{\alpha-a})^{z_3} (1+\|\nabla \Psi\|_{L^{\alpha q_1}(U\times(0,t))}+\|\Psi_t\|_{L^{\alpha q_1}(U\times(0,t))})^{z_2}.
\eeq

If $t\ge 1$ then
\beq\label{Nsmall}
\norm{\bar p(t)}_{L^\infty(U)}\le C(1+   \|\bar{p}_0\|_{L^\alpha(U)} + [ Env A(\alpha,t)]^\frac1{\alpha-a})^{z_3}( 1+ \|\nabla \Psi\|_{L^{\alpha q_1}( U\times (t-1,t))}+\|\Psi_t\|_{L^{\alpha q_1}( U\times (t-1,t))})^{z_2}. 
\eeq

{\rm (ii)} If $ A(\alpha)<\infty$ then
\beq\label{NlimI}
 \limsup_{t\rightarrow\infty (U)} \norm{\bar p(t)}_{L^\infty} \le C(1+   A(\alpha)^\frac1{\alpha-a})^{z_3}(1+\limsup_{t\to\infty}( \|\nabla \Psi(t)\|_{L^{\alpha q_1}( U\times (t-1,t))}+\|\Psi_t(t)\|_{L^{\alpha q_1}( U\times (t-1,t))}))^{z_2}.
\eeq 

{\rm (iii)} If $\beta(\alpha)<\infty$ then there is $T>0$ such that
 \beq\label{NPbarI}
\norm{\bar p(t)}_{L^\infty(U)}\le C(1+  \beta(\alpha)^\frac1{\alpha-2a} + \|A(\alpha,\cdot)\|_{L^\frac\alpha{\alpha-a}(t-1,t)}^\frac1{\alpha-a} )^{z_3}( 1 +\|\nabla \Psi\|_{L^{\alpha q_1}( U\times (t-1,t))}+\|\Psi_t\|_{L^{\alpha q_1}( U\times (t-1,t))})^{z_2}
  \eeq
 for all $t\ge T$.
\end{theorem}

\begin{proof}
(i) Note that $\alpha=\widehat\alpha$. 
Let $t\in(0,3]$. Applying \eqref{esim} for $T=t$ and $\theta=1/2$, we have
 \beq\label{pb1}
 \|\bar p(t)\|_{L^\infty}\le C(1+t^{-\frac{1}{\delta_1}} )(1+ \|\nabla \Psi\|_{L^{\alpha q_1}(U\times(0,t))}+ \|\Psi_t\|_{L^{\alpha q_1}(U\times(0,t))})^{z_2}(1+\|\bar p\|_{L^\alpha(U\times(0,t))})^{z_3}.
\eeq
Using \eqref{pbar:ineq1} to estimate $\|\bar p\|_{L^\alpha(U\times(0,t))}$ in \eqref{pb1}, we obtain \eqref{Ne2}.   

For $t\ge 1$, applying \eqref{esim}  with $T_0=t-1$, $T=1$ and $\theta=1/2$ we obtain 
\beq\label{Npinf}
\norm{\bar p(t)}_{L^\infty}\le C( 1 +\|\nabla \Psi\|_{L^{\alpha q_1}( U\times (t-1,t))}+\|\Psi_t\|_{L^{\alpha q_1}( U\times (t-1,t))})^{z_2}(1+ \| \bar p\|_{L^\alpha( U\times (t-1,t) )})^{z_3}. 
\eeq
Again using \eqref{pbar:ineq1} with noticing that the function $Env A(\alpha,t)$ increasing in $t$, then we obtain \eqref{Nsmall} from \eqref{Npinf}.   
 
(ii) From \eqref{Npinf} we have
\begin{multline*}
\limsup_{t\to\infty} \norm{\bar p(t)}_{L^\infty}\le C( 1 +\limsup_{t\to\infty}( \|\Psi_t\|_{L^{\alpha q_1}( U\times (t-1,t))}+ \|\Psi_t\|_{L^{\alpha q_1}( U\times (t-1,t))}))^{z_2}\\
\cdot (1+ \limsup_{t\to\infty} \| \bar p\|_{L^\alpha( U\times (t-1,t) )})^{z_3}. 
\end{multline*}
By \eqref{limsupPbar},    
\beq\label{pb2}
\limsup_{t\to\infty}  \| \bar p\|_{L^\alpha( U\times (t-1,t) )}^\alpha \le \limsup_{t\to\infty} \int_U |\bar p(x,t)|^\alpha dx\le  C(1+A(\alpha))^{\alpha/(\alpha -a)}.
\eeq
Thus \eqref{NlimI} follows. 

(iii) Using \eqref{pbar:ineq2} we have for large $t$ that
\beq\label{careful}
\int_{t-1}^t\int_U |\bar p(x,\tau)|^\alpha dxd\tau \le C\Big(1+\beta(\alpha)^\frac{\alpha}{\alpha-2a}+\int_{t-1}^t A(\alpha,\tau)^\frac{\alpha}{\alpha-a}d\tau\Big).
\eeq
Therefore, \eqref{NPbarI} follows from \eqref{Npinf} and \eqref{careful}.  
\end{proof}

\section{Gradient estimates on the boundary}
\label{PGradBdry}

In this section, we estimate the maximum of $|\nabla p|$ on the boundary $\Gamma\times(0,\infty)$.
It will be used in estimates of $L^s$-norms for the gradient in section \ref{PGradS} and its $L^\infty$-norm in section \ref{PGradInfty}.
We extend Ladyzhenskaya-Uraltseva's technique \cite{LadyParaBook68} in order to have better estimates for large time.

We rewrite  the PDE in  \eqref{p:eq} for $p$ in the non-divergence form as
\beq\label{eq4p}
 p_t- \sum_{i,j=1}^n\mathcal{A}_{ij}p_{x_ix_j}=0,
\eeq
where
\beqs
\mathcal{A}_{ij}=\mathcal{A}_{ij}(x,t)\eqdef K(|\nabla p(x,t)|)\delta_{ij}+\frac{K'(|\nabla p(x,t)|)}{|\nabla p(x,t)|} p_{x_i}(x,t)p_{x_j}(x,t),\quad i,j=1,2,\ldots,n.
\eeqs

Thanks to \eqref{K-est-3}--\eqref{K-est-2} we find that
\beqs
|\mathcal{A}_{ij}|\le (1+a)K(|\nabla p|),
\eeqs
\beqs
(1-a)K(|\nabla p|)|y|^2\le \sum_{i,j=1}^n\mathcal{A}_{ij}y_iy_j\le K(|\nabla p|)|y|^2, \quad \forall y\in \R^n.
\eeqs
By \eqref{K-est-3}, we have more explicit relations:
\beq\label{Abound}
|\mathcal{A}_{ij}|\le (1+a)d_2(1+|\nabla p|)^{-a},
\eeq
\beq\label{PSD}
(1-a)d_1(1+|\nabla p|)^{-a}|y|^2\le \sum_{i,j=1}^n\mathcal{A}_{ij}y_iy_j\le d_2(1+|\nabla p|)^{-a}|y|^2, \quad \forall y\in \R^n.
\eeq

We will establish boundary estimates in case the boundary is flat first. 
For general boundary, we will  flatten it out and hence transform equation \eqref{eq4p} to a different, but similar, one.
To prepare for this transformation, we consider a more general PDE:
\beq\label{ee1}
 {\mathcal L} p \eqdef  p_t-\sum_{i,j=1}^n \tilde {\mathcal{A}}_{ij}p_{x_ix_j}+\sum_{i=1}^n \tilde b_i  p_{x_i}=0,
\eeq
where $\tilde {\mathcal{A}}_{ij}=\tilde {\mathcal{A}}_{ij}(x,t)$, $\tilde b_i=\tilde b_i(x,t)$ for $i,j=1,2,\ldots, n$.

Define for $\xi\ge 0$ the function
\beqs
\tilde K(\xi)=\frac1{(1+\xi)^a}.
\eeqs
Similar to \eqref{Abound} and \eqref{PSD}, we assume that $\tilde{\mathcal A}_{ij}$ satisfy
\beq\label{hyp1}
(\sum_{i,j=1}^n |\tilde {\mathcal{A}}_{ij}|^2)^{1/2}\le c_1\tilde K(|\nabla p|),
\eeq
\beq\label{Pos}
c_2\tilde K(|\nabla p|)|y|^2\le \sum_{i,j=1}^n\tilde {\mathcal{A}}_{ij}y_iy_j\le c_3\tilde K(|\nabla p|)|y|^2, \quad \forall y\in \R^n,
\eeq
where $c_1, c_2,c_3$ are positive constants.    
In addition, we also assume that 
\beq\label{hyp3}
|\tilde b(x,t)|\le c_4 \tilde K(|\nabla p|)
\text{ with }\tilde b=(\tilde b_1, \ldots, \tilde b_n   ).
\eeq
for some positive constant $c_4.$
Note, particularly when $y=e_k$, a unit vector in the standard basis of $\mathbb R^n$,  we have from \eqref{Pos} that 
\beq\label{hyp2}
\tilde {\mathcal{A}}_{kk}\ge c_2\tilde K(|\nabla p|),\quad k=1,2,\ldots,n.
\eeq

Denote the ball in $\mathbb R^n$ centered at the origin with radius $R$ by $B_R$, and the upper half ball by $B_R^+$, i.e., $B_R^+=B_R\cap \{ x=(x',x_n):x'\in\R^{n-1},\ x_n>0\}$.
Also, denote $\Gamma_R=\partial B_R^+\cap\{x_n=0\}$-the flat portion of the boundary of $B_R^+$.

We start with estimates on a flat boundary.

\begin{proposition}\label{prop5}
Let $T>0$ and $R>0$. Let $p(x,t)\in C^{1,0}(\bar B_R^+\times [0,T])\cap C^{2,1}( B_R^+\times (0,T])$
be a solution of \eqref{ee1} in $B_R^+\times (0,T]$, with $\tilde {\mathcal A}_{ij}$ and $\tilde b_i$ satisfying \eqref{hyp1}--\eqref{hyp3} in $\bar B_R^+\times [0,T]$. 
Then
\begin{multline}\label{gradp-0}
\max_{\Gamma_{R/2}\times [\theta T,T]}|\nabla p(x,t)| 
\le C(1+R^{-2})e^{C'R}(1+(\theta T)^{-1})^{\mu_0}\exp \big(C'\max_{\bar B_R^+\times [\theta T/2,T]} |\bar p| \big) \\
\cdot(1+\max_{\bar B_R^+ \times [0,T]}|\Psi_t|^{\mu_0}+\max_{\bar B_R^+\times [0,T]}|\nabla \Psi|^2+\max_{\bar B_R^+\times [0,T]}|\nabla^2 \Psi|),
\end{multline}
for any $\theta\in(0,1)$, where $C$, $C'$ are positive numbers independent of $R$, $T$, $\theta$, and 
\beq \label{mu0}
\mu_0=\frac2{2-a}.
\eeq 
\end{proposition}
\begin{proof}
(a) Let $v=-1+e^{\kappa_0\bar p }$ with the constant $\kappa_0>0$ chosen later. Then 
\beq\label{vBdr}
v=0  \text{ on } \partial B_R^+\cap \{x_n=0\}.
\eeq
For $i,j=1,2,\ldots,n$, we have 
\beq\label{v_tgradv}
v_t=\kappa_0 e^{\kappa_0\bar p}\bar p_t,\quad 
v_{x_i}=\kappa_0 e^{\kappa_0\bar p}\bar p_{x_i},\quad
v_{x_ix_j}=\kappa_0 e^{\kappa_0\bar p}\bar p_{x_ix_j}
+\kappa_0^2 e^{\kappa_0\bar p}\bar p_{x_i}\bar p_{x_j}.
\eeq
Substituting  \eqref{v_tgradv}   into  \eqref{ee1} gives
\beqs
 \kappa_0^{-1}e^{-\kappa_0\bar p}v_t+\Psi_t
-\sum_{i,j=1}^n \tilde {\mathcal{A}}_{ij} \Big\{ \kappa_0^{-1}e^{-\kappa_0\bar p}v_{x_ix_j}+\Psi_{x_ix_j} - \kappa_0\bar p_{x_i}\bar p_{x_j}\Big\}
+\tilde b\cdot (\kappa_0^{-1} e^{-\kappa_0 \bar p}\nabla v+\nabla \Psi)=0.
\eeqs
Multiplying by $\kappa_0e^{\kappa_0\bar p}$, it shows that $v(x,t)$ solves the equation 
\beqs
v_t   -\sum_{i,j=1}^n\tilde {\mathcal{A}}_{ij} v_{x_ix_j} + \tilde b\cdot \nabla v= 
-\sum_{i,j=1}^n\tilde {\mathcal{A}}_{ij}  \kappa_0^2e^{\kappa_0\bar p}\bar p_{x_i}\bar p_{x_j} 
+\kappa_0e^{\kappa_0\bar p}(\sum_{i,j=1}^n\tilde {\mathcal{A}}_{ij} \Psi_{x_ix_j}-\Psi_t-\tilde b\cdot\nabla \Psi ).
\eeqs
Thus, by \eqref{Pos} and Cauchy-Schwarz inequality, we have   
\beqs
{\mathcal L} v \le 
- c_2\kappa_0^2e^{\kappa_0\bar p} \tilde K(|\nabla p|)|\nabla \bar p|^2
+\kappa_0 e^{\kappa_0\bar p}\Big[(\sum_{i,j=1}^n |\tilde {\mathcal{A}}_{ij}|^2)^{1/2} (\sum_{i,j=1}^n|\Psi_{x_ix_j}|^2)^{1/2}
+|\tilde b| |\nabla \Psi| +|\Psi_t|\Big] .
\eeqs
Using inequality $|\nabla \bar p|^2=|\nabla p-\nabla \Psi|^2\ge |\nabla p|^2/2-|\nabla \Psi|^2$, and inequalities \eqref{hyp1}, \eqref{hyp3}, we get
\beqs
\begin{split}
{ \mathcal L} v
&\le 
- \frac{c_2}{2}\kappa_0^2e^{\kappa_0\bar p}\tilde K(|\nabla p|)|\nabla p|^2+c_2\kappa_0^2e^{\kappa_0\bar p}\tilde K(|\nabla p|)|\nabla \Psi|^2\\
&\quad  + \kappa_0 e^{\kappa_0\bar p}\Big[\tilde K(|\nabla p|) (c_1 |\nabla^2 \Psi| + c_4 |\nabla \Psi|) +|\Psi_t|\Big].
\end{split}
\eeqs
Let 
\begin{align*}
M_1=\max_{\bar B_R^+\times[0,T]}( c_1|\nabla^2 \Psi|+c_4|\nabla \Psi|),\quad 
M_2= \max_{\bar B_R^+\times[0,T]}|\Psi_t|, \quad
M_3=c_2\max_{\bar B_R^+\times[0,T]} |\nabla \Psi|^2.
\end{align*} 
Then
\beq\label{Bv_t}
{\mathcal L} v\le - \frac{c_2}{2}\kappa_0^2e^{\kappa_0\bar p}\tilde K(|\nabla p|)|\nabla p|^2
+\kappa_0(M_3\kappa_0+  M_1)e^{\kappa_0\bar p}\tilde K(|\nabla p|) 
+ \kappa_0e^{\kappa_0\bar p}M_2.
\eeq
Let $\varep>0$. We write $M_2= \tilde K(|\nabla p|)\cdot M_2 (1+|\nabla p|)^a $ and estimate
\beqs
M_2(1+|\nabla p|)^a\le M_2(1+|\nabla p|^a) \le M_2 + \varep^{-a/(2-a)}M_2^{2/(2-a)}+\varep|\nabla p |^2.
\eeqs
 From this and \eqref{Bv_t}, we find that 
\beq\label{lv1}
\begin{split}
{\mathcal L} v
&\le - \big(\frac{c_2}{2}\kappa_0- \varep\Big)\kappa_0e^{\kappa_0\bar p}\tilde K(|\nabla p|)|\nabla p|^2\\
&\quad +\kappa_0e^{\kappa_0\bar p}\tilde K(|\nabla p|) (\kappa_0M_3  +  M_1+  M_2+ \varep^{-a/(2-a)}M_2^{2/(2-a)}).
\end{split}
\eeq

(b) We localize the above calculations. Let $\zeta=\zeta(x,t)\in [0,1]$ be a cut-off function on $B_R\times(0,T)$ with $\zeta=0$ on $(B_R\times [\theta T/2,T]) \cup (\partial B_R\times[0,T])$, and  $\zeta(x,t)=1$ on $B_{R/2}\times (\theta T,T)$, and satisfy
\beqs
|\nabla \zeta|\le \frac{C_0}R, \quad |\nabla^2 \zeta|\le \frac{C_0}{R^2}, \quad |\zeta_t|\le \frac{C_0}{\theta T},
\eeqs 
where $C_0>0$ is independent of $R$, $\theta$, and $T$.
Denote by $\chi=\chi(t)$  the characteristic function of $[\theta T/T,T]$.

Let $w=v\zeta^2$. Then
 \beq \label{Lw}
{\mathcal L} w=\zeta^2 {\mathcal L} v + v {\mathcal L} (\zeta^2) - 2 \zeta(\tilde{\mathcal A} \nabla v)\cdot \nabla \zeta.
\eeq

For the last term, by \eqref{hyp1} and \eqref{v_tgradv}:
\begin{align*}
|2\zeta (\tilde{\mathcal A} \nabla v)\cdot \nabla \zeta |
&\le 2c_1\zeta \tilde K(|\nabla p|) |\nabla v| |\nabla \zeta| \le C_1  \zeta\kappa_0 e^{\kappa_0\bar p}\tilde K(|\nabla p|)|\nabla \bar p|\\
&\le \kappa_0 e^{\kappa_0\bar p}\tilde K(|\nabla p|)\big[C_1  \zeta|\nabla p| \chi+ C_1  \zeta|\nabla \Psi|\big],
\end{align*}
where $C_1=2c_1$.
Thus
\beq\label{Agrad}
|2\zeta (\tilde{\mathcal A} \nabla v)\cdot \nabla \zeta | \le \kappa_0 e^{\kappa_0\bar p}\tilde K(|\nabla p|)  \big( \varep \zeta^2|\nabla p|^2 + \varep^{-1} C_1^2 \chi +C_1 \zeta |\nabla\Psi| \big).
\eeq

For the second term on the right-hand side of \eqref{Lw},
\beq\label{term2}
|v{\mathcal L} (\zeta^2)|
\le |v|2\zeta |\zeta_t| +|v|(\sum_{i,j=1}^n|\tilde{\mathcal A}_{ij}| |(\zeta^2)_{x_ix_j}| +2\zeta  |\tilde b| |\nabla \zeta|)
\le |v|2\zeta |\zeta_t| +C_2|v|\chi \tilde K(|\nabla p|),
\eeq
where $C_2=C_{3}(1+R^{-2})$, with $C_3>0$  independent of $R$, $\theta$, $T$. 
Treating the term $2|v|\zeta |\zeta_t|$ in \eqref{term2} the same way as we did for $M_2$ in \eqref{Bv_t}:
\begin{align*}
2|v|\zeta |\zeta_t| 
&\le C_4|v|\zeta  \le C_4|v| \zeta \tilde K(|\nabla p|) (1+|\nabla p|^a) 
\le \tilde K(|\nabla p|)(C_4|v|\zeta +C_4|v|\zeta|\nabla p|^a),
\end{align*}
where
$C_4=2 C_0(\theta T)^{-1}.$
Note that
\beqs
C_4|v|\zeta|\nabla p|^a = C_4|v|\zeta^{1-a} [\varep\kappa_0e^{\kappa_0\bar p}]^{-a/2} \cdot  [\varep^{1/2}\kappa_0^{1/2}e^{\kappa_0\bar p/2} \zeta|\nabla p|]^a.
\eeqs
Applying  Young's inequality with power $2/(2-a)$ and $2/a$ yields
\beq\label{vzt}
2|v|\zeta |\zeta_t| 
\le  \tilde K(|\nabla p|)\Big[C_4|v|\zeta +(C_4|v|\zeta^{1-a})^{2/(2-a)}(\varep\kappa_0e^{\kappa_0\bar p})^{-a/(2-a)}\Big] +\varep\kappa_0e^{\kappa_0\bar p}\tilde K(|\nabla p|)\zeta^2 |\nabla p|^2 .
\eeq
It follows from \eqref{term2} and \eqref{vzt} that 
\begin{multline}\label{vLz}
|v|{\mathcal L}(\zeta^2)
\le   \varep\kappa_0e^{\kappa_0\bar p}\tilde K(|\nabla p|)\zeta^2 |\nabla p|^2 \\
+\tilde K(|\nabla p|)\Big[ C_4|v|\zeta +(C_4|v|\zeta^{1-a})^{2/(2-a)}(\varep\kappa_0)^{-a/(2-a)} e^{\kappa_0|\bar p|a/(2-a)}+C_2|v|\chi \Big].
\end{multline}

Using \eqref{lv1}, \eqref{vLz} and \eqref{Agrad} in \eqref{Lw}, we obtain
\beq\label{lw2}
\begin{aligned}
\mathcal L w
&\le - \zeta^2\big(\frac{c_2}{2}\kappa_0- 3\varep\Big)\kappa_0e^{\kappa_0\bar p}\tilde K(|\nabla p|)|\nabla p|^2\\
&\quad +\kappa_0e^{\kappa_0\bar p} \tilde K(|\nabla p|)\Big[ \zeta^2(\kappa_0M_3  +  M_1+  M_2+ M_2^{2/(2-a)}\varep^{-a/(2-a)})+C_1^2\varep^{-1}\chi+ C_1 \zeta |\nabla\Psi| \Big] \\
&\quad + \tilde K(|\nabla p|)[C_2 \chi |v| +C_4|v|\chi +(C_4|v|\chi)^{2/(2-a)}(\varep\kappa_0)^{-a/(2-a)}e^{\kappa_0|\bar p|a/(2-a)}] .
\end{aligned}
\eeq

Choose $\varep=1$ and $\kappa_0= 8/c_2$. Then 
\beq\label{lw22}
\begin{aligned}
\mathcal L w
&\le \kappa_0\chi e^{\kappa_0|\bar p|} \tilde K(|\nabla p|)\Big[ \kappa_0M_3  +  M_1+  M_2+ M_2^{2/(2-a)}+C_1^2+ C_1 |\nabla\Psi| \Big] \\
&\quad + \tilde K(|\nabla p|)[C_2 \chi |v| +C_4|v|\chi +(C_4|v|\chi)^{2/(2-a)}\kappa_0^{-a/(2-a)}e^{\kappa_0|\bar p|a/(2-a)}] .
\end{aligned}
\eeq
Let $M_4= \max_{\bar B_R^+\times[\theta T/2,T]} |\bar p|$. 
Note that 
$\chi e^{\kappa_0|\bar p|}, \chi |v|\le e^{\kappa_0 M_4}.$
Define  
\begin{multline*}
M_5 =\kappa_0e^{\kappa_0M_4}(\kappa_0M_3  +  M_1+ c_0 M_2+M_2^{2/(2-a)}+C_1^2+C_1\max_{\bar B_R^+\times[0,T]} |\nabla \Psi|)\\
+e^{\kappa_0 M_4(2+a)/(2-a)} (1+\kappa_0^{-a/(2-a)})(C_2+C_4+C_4^\frac{2}{2-a}).
\end{multline*}
We obtain $\mathcal L w\le M_5\tilde K(|\nabla p|)$.
Note that
\begin{multline*}
M_5\le 
\kappa_0e^{\kappa_0M_4}(\kappa_0M_3  +  M_1+ c_0 M_2+M_2^{2/(2-a)}+C_1^2+C_1\max_{\bar B_R^+\times[0,T]} |\nabla \Psi|)\\
+e^{\kappa_0 M_4(2+a)/(2-a)} (1+\kappa_0^{-a/(2-a)})(1+C_2)2(1+C_4^\frac{2}{2-a})
\le M_6,
\end{multline*}
where $M_6$ is explicitly defined by
\beq\label{MM}
M_6= C(1+R^{-2}) e^{\kappa_0M_4(2+a)/(2-a)} (1+(\theta T)^{-\frac{2}{2-a}})(1+\max_{\bar B_R^+\times [0, T]  }|\Psi_t|^\frac2{2-a}+\max_{\bar B_R^+\times [0, T]  }|\nabla\Psi|^2+\max_{\bar B_R^+\times [0, T]  }|\nabla^2\Psi|),
\eeq
with $C>0$ independent of $R$, $T$, $\theta$. Therefore,
\beq\label{Lw8}
\mathcal L w\le M_6\tilde K(|\nabla p|).
\eeq

(c)  Let $\tilde w = w + \lambda e^{-\mu x_n}$, for $\lambda,\mu>0$. Then
\begin{align*}
{\mathcal L} \tilde w
&={\mathcal L} w +\lambda{ \mathcal L }(e^{-\mu x_n})
\le  M_6\tilde K(|\nabla p|) -\lambda\mu^2\tilde {\mathcal{A}}_{nn} e^{-\mu x_n} -\lambda \mu \tilde b_n e^{-\mu x_n}\\
&\le  M_6\tilde K(|\nabla p|) -\lambda \mu(c_2\mu - c_4)\tilde K(|\nabla p|)e^{-\mu x_n}.
\end{align*}
Choose $\mu=(1+c_4)/c_2$ and 
\beq\label{lamchoice}
 \lambda =M_6 e^{\mu R}/\mu.
\eeq
Then on $B_R^+\times(0,T]$,
${\mathcal L}\tilde w\le 0.$
Since $p(x,t)\in C^{1,0}(\bar U\times[0,T])$, the function  $\tilde K(|\nabla p|)$ is bounded below by a positive number. Then by \eqref{Pos},  $\mathcal L$ is a parabolic operator.
Therefore, the maximum principle for operator $\mathcal L$ implies 
\beqs
\max_{\bar B_R^+\times[0,T]}\tilde w =\max_{\partial_p ( B_R^+\times(0,T))} \tilde w.
\eeqs
Here $\partial_p$ denotes the parabolic boundary.

When $(x,t)\in (\bar B_R^+ \times \{0\} )\cup ((\partial B_R^+\setminus \Gamma_R)\times(0,T])$, we have $\zeta(x,t)=0$ hence
\beqs \tilde w(x,t)=0+\lambda e^{-\mu x_n}\le \lambda.\eeqs

When $(x,t)\in \Gamma_R\times (0,T]$: $v(x,t)=0$ hence
$\tilde w(x,t) = \lambda.$

Thus,
\beqs
\max_{\bar B_R^+\times[0,T]}\tilde w =\max_{\partial_p ( B_R^+\times(0,T))} \tilde w \le \lambda= \tilde w(x,t), \quad \forall (x,t)\in \Gamma_R\times (0,T].
\eeqs
Hence, we have $\tilde w_{x_n}\le 0$ on $\Gamma_R\times(0,T]$, equivalently, $w_{x_n}\le  \mu\lambda$ on $\Gamma_R\times(0,T]$.
Note that $w=v$ on $\Gamma_{R/2}\times[\theta T,T]$. Thus, on $\Gamma_{R/2}\times[\theta T,T]$ we have
\beqs
\kappa_0 (p_{x_n}-\Psi_{x_n})=v_{x_n}=w_{x_n}\le  \mu\lambda,
\eeqs
which implies
\beq\label{p1}
 p_{x_n}\le \Psi_{x_n} + \kappa_0^{-1} \mu\lambda.
\eeq
Replacing $p$ by $-p$, and $\Psi$ by $-\Psi$, we obtain again \eqref{p1} with $-p_{x_n}$ and $-\Psi_{x_n}$ in place of $p_{x_n}$ and $\Psi_{x_n}$.
Therefore, we have on $\Gamma_{R/2}\times[\theta T,T]$ that
\beqs
| p_{x_n}|\le C(| \Psi_{x_n}|+ \kappa_0^{-1} \mu\lambda).
\eeqs
Combining with the fact  $p_{x_i}=\Psi_{x_i}$ on $\Gamma_{R}\times(0,T]$ for $i<n$, we assert that
\beq\label{gb5}
|\nabla p| \le C(|\nabla \Psi|+ \kappa_0^{-1} \mu\lambda)\text { on } \Gamma_{R/2}\times[\theta T,T].
\eeq
Using \eqref{lamchoice} and \eqref{MM} in \eqref{gb5}, we obtain
\begin{multline}
\max_{\Gamma_{R/2}\times[\theta T,T]}|\nabla p(x,t)| 
\le C(1+R^{-2})e^{C'R}(1+(\theta T)^{-1})^{\mu_0}\exp \big(C'\max_{\bar B_R^+\times [\theta T/2,T]} |\bar p| \big) \\
\cdot(1+\max_{\bar B_R^+ \times [0,T]}|\Psi_t|^\frac2{2-a}+\max_{\bar B_R^+\times [0,T]}|\nabla \Psi|^2+\max_{\bar B_R^+\times [0,T]}|\nabla^2 \Psi|),
\end{multline}
thus proving \eqref{gradp-0}.
\end{proof}

The general domain and boundary are treated in the next theorem.

\begin{theorem} \label{theo51}
Let $p(x,t)$ be a solution of \eqref{p:eq}.
Then for any $T_0> 0$, $T>0$, and $\theta\in(0,1)$, 
\begin{multline}\label{gb}
\max_{\Gamma\times [T_0+\theta T,T_0+T]}|\nabla p(x,t)| 
\le C(1+(\theta T)^{-1})^{\mu_0}\exp \big(C'\max_{\bar U\times [T_0+\theta T/2,T_0+T]} |\bar p| \big) \\
\cdot(1+\max_{\bar U \times [T_0,T_0+T]}|\Psi_t|^\frac2{2-a}+\max_{\bar U\times [T_0,T_0+T]}|\nabla \Psi|^2+\max_{\bar U\times [T_0,T_0+T]}|\nabla^2 \Psi|),
\end{multline}
where $C$, $C'$ are positive numbers independent of $T_0$, $T$, and $\theta$. 
\end{theorem}
\begin{proof}
By replacing $p(x,t)$ with $p(x,T_0+t)$, we can assume, without loss of generality, that $T_0=0$ and $p\in C^{2,1}(\bar U\times [0,T])$.
Let  $x_0\in \partial U$. There exist an open neighborhood $V$ of $x_0$, a radius $R>0$ and $C^{2}$-bijections $y=\Phi(x):U\cap V\to B_{2R}^+$ and $x=\Upsilon(y):B_{2R}^+\to  U\cap V$ such that $\Phi=\Upsilon^{-1}$,
\beq \label{ig}
\Phi( \Gamma\cap V)=B_{2R}\cap \{y_n=0\},
\eeq   and 
\beq \label{Ctwo}
\|\Phi\|_{C^2(U\cap V)},\|\Upsilon\|_{C^2(B_{2R}^+)}\le c_0 \quad \text{for some }c_0>0.
\eeq

Define $\tilde p(y)=p ( \Upsilon (y) ) $ and $\tilde \Psi(y)=\Psi(\Upsilon(y))$. 
We use the fact that $p(x,t)$ is a solution of \eqref{eq4p}.
Simple calculations give 
  \beqs
 p_{x_i} = \sum_{k=1}^n \tilde p_{y_k}\Phi_{x_i}^{k}, \quad   p_{x_ix_j} = \sum_{k,l=1}^n \tilde p_{y_ky_l}\Phi_{x_i}^{k}\Phi_{x_j}^{l} +  \sum_{k=1}^n \tilde p_{y_k}\Phi_{x_ix_j}^{k}
\eeqs
with $\Phi=(\Phi^k)_{k=1,\ldots,n}.$ Thus, 
\beqs
p_t - \sum_{i,j=1}^n {\mathcal  A}_{ij}p_{x_xx_j}
 = \tilde p_t - \sum_{k,l=1}^n \tilde {\mathcal  A}_{kl} \tilde p_{y_k y_l}  +\sum_{k=1}^n \tilde b_k \tilde p_{y_k} 
 \eqdef {\mathcal L} \tilde p,
\eeqs
where 
\beqs
\tilde {\mathcal  A}_{kl} = \sum_{i,j=1}^n\mathcal A_{ij}\Phi_{x_i}^{k}\Phi_{x_j}^{l}, \quad \tilde b_k = -\sum_{i,j=1}^n  \mathcal A_{ij} \Phi^k_{x_ix_j}.  
\eeqs
Therefore, $\tilde p(y,t)$ satisfies the equation
\beq\label{tileq}
\mathcal L \tilde p(y,t)=0 \text{ on } B_R^+\times(0,T),\quad \tilde p(y,t)=\tilde \Psi(y,t) \text{ for } y_n=0.
\eeq

Now we check the conditions \eqref{hyp1}, \eqref{Pos} and \eqref{hyp3}. 
By chain rule and \eqref{Ctwo}, 
\beq\label{echange}
\begin{aligned}
&|\nabla_y \tilde p(y,t)| \le  \tilde c_0  |\nabla p(x,t)|,\quad |\nabla p(x,t)|\le  \tilde c_0  |\nabla_y \tilde p(y,t)|,\\
&|\nabla_y^2 \tilde \Psi|\le  \tilde c_0 (|\nabla^2 \Psi|+|\nabla \Psi|),\quad
|\nabla^2 \Psi|\le  \tilde c_0  (|\nabla_y^2 \tilde \Psi|+|\nabla_y \tilde \Psi|),
\end{aligned}
\eeq
for some $ \tilde c_0 \ge 1$.
We have for any $k,l=1,2,\ldots,n$ that
\begin{align*}
|\tilde {\mathcal  A}_{kl}| 
&= |\sum_{i,j=1}^n\mathcal A_{ij}\Phi_{x_i}^{k}\Phi_{x_j}^{l}| \le n^2\tilde c_0^2\max_{1\le i,j\le n}|\mathcal A_{ij}| 
\le  n^2 \tilde c_0^2(1+a)d_2 (1+|\nabla p|)^{-a} \\
&\le  n^2 \tilde c_0^2(1+a)d_2 (1+\tilde c_0^{-1}|\nabla_y \tilde p|)^{-a}
\le   n^2 \tilde c_0^2(1+a)d_2 \tilde c_0^a(1+|\nabla_y \tilde p|)^{-a}.
\end{align*}
This yields \eqref{hyp1}.
For all $\xi\in \mathbb R^n,$
\beq\label{B1}
\sum_{k,l=1}^n \tilde {\mathcal  A}_{kl} \xi_k\xi_l = \sum_{k,l=1}^n  \sum_{i,j=1}^n\mathcal A_{ij}\Phi_{x_i}^{k}\Phi_{x_j}^{l}\xi_k\xi_l =   \sum_{i,j=1}^n\mathcal A_{ij}\eta_i\eta_j  
\eeq
with $\eta =  ( D\Phi)\xi$, i.e.,  $\xi =(D\Upsilon)\eta$. 
By \eqref{PSD} we have  
\beq\label{B2}
(1-a)d_1(1+|\nabla p|)^{-a}|\eta|^2\le \sum_{i,j=1}^n \mathcal{A}_{ij}\eta_i\eta_j\le d_2(1+|\nabla p|)^{-a}|\eta|^2.
\eeq 
Note that  $\tilde c_0^{-1}|\xi| \le |\eta|\le \tilde c_0|\xi|$.
It follows from \eqref{B2} that
\beqs
\sum_{k,l=1}^n \tilde{\mathcal{A}}_{kl}\xi_k\xi_l
\le \tilde c_0^2 d_2(1+\tilde c_0^{-1}|\nabla_y\tilde p|)^{-a}|\xi|^2
\le \tilde c_0^2 d_2\tilde c_0^a(1+|\nabla_y\tilde p|)^{-a}|\xi|^2,
\eeqs
\beqs
\sum_{k,l=1}^n \tilde{\mathcal{A}}_{kl}\xi_k\xi_l
\ge (1-a) \tilde c_0^{-2}d_1(1+\tilde c_0|\nabla_y \tilde p|)^{-a} |\xi|^2\\
\ge (1-a) \tilde c_0^{-2}d_1\tilde c_0^{-a}(1+|\nabla_y \tilde p|)^{-a} |\xi|^2 ,
\eeqs
hence \eqref{Pos} is satisfied.
Next, we bound $ |\tilde b_k|$ for each $k=1,2,\ldots,n$ by
 \begin{align*}
|\tilde b_k| 
&\le \sum_{i,j=1}^n  |\mathcal A_{ij} \Phi^k_{x_ix_j}|
\le c_0 n^2 (1+a) d_2(1+|\nabla p|)^{-a} \\
&\le c_0 n^2(1+a) d_2(1+\tilde c_0^{-1}|\nabla_y\tilde p|)^{-a} 
\le c_0 n^2(1+a) d_2\tilde c_0^a (1+|\nabla_y\tilde p|)^{-a}.
\end{align*}
Hence we have \eqref{hyp3}.

Applying estimate \eqref{gradp-0}  from Proposition \ref{prop5} to $\tilde p$, $\tilde \Psi$ and equation \eqref{tileq}, we obtain
\begin{multline}\label{gradp-2}
\max_{\Gamma_{R/2}\times [\theta T,T]}|\nabla_y\tilde p(y,t)| 
\le \tilde C_R(1+(\theta T)^{-1})^{\mu_0}\exp \big(C'\max_{\bar B_R^+\times [\theta T/2,T]} |\tilde p-\tilde \Psi| \big) \\
\cdot(1+\max_{\bar B_R^+ \times [0,T]}|\tilde \Psi_t|^{\mu_0}+\max_{\bar B_R^+\times [0,T]}|\nabla_y \tilde \Psi|^2+\max_{\bar B_R^+\times [0,T]}|\nabla_y^2 \tilde \Psi|).
\end{multline}
Let $\rho>0$ sufficiently  small such that $U\cap B_\rho(x_0)\subset \Upsilon(B_{R/2}^+)$.
Thanks to \eqref{ig}, \eqref{gradp-2} and \eqref{echange}, it follows  that
\begin{multline}
\max_{(\Gamma\cap \bar B_\rho(x_0))\times [\theta T,T]}|\nabla p(x,t)| 
\le C_R(1+(\theta T)^{-1})^{\mu_0}\exp \big(C'\max_{\bar U\times [\theta T/2,T]} |\bar p| \big) \\
\cdot(1+\max_{\bar U \times [0,T]}|\Psi_t|^{\mu_0}+\max_{\bar U\times [0,T]}|\nabla \Psi|^2+\max_{\bar U\times [0,T]}(|\nabla^2 \Psi|+|\nabla \Psi|)).
\end{multline} 
By Cauchy's inequality for the last $|\nabla \Psi|$, we obtain
\begin{multline}\label{gradp-3}
\max_{(\Gamma\cap \bar B_\rho(x_0))\times [\theta T,T]}|\nabla p(x,t)| 
\le C_R(1+(\theta T)^{-1})^{\mu_0}\exp \big(C'\max_{\bar U\times [\theta T/2,T]} |\bar p| \big) \\
\cdot(1+\max_{\bar U \times [0,T]}|\Psi_t|^{\mu_0}+\max_{\bar U\times [0,T]}|\nabla \Psi|^2+\max_{\bar U\times [0,T]}|\nabla^2 \Psi|).
\end{multline} 
By using a finite open covering of $\Gamma$, we obtain the desired estimate \eqref{gb} from \eqref{gradp-3}.
The proof is complete.
\end{proof}

The bounds of $\|\nabla p(t)\|_{L^\infty(\Gamma)}$, in fact, can be expressed in terms of the initial and boundary data as follows.

\begin{corollary}\label{Cor5}
Let $p(x,t)$ be a solution of \eqref{p:eq}, and let $\alpha$ satisfy \eqref{alcond}.

{\rm (i)} If $0<t\le 3$ then
\begin{multline}\label{b1}
\|\nabla p(t)\|_{L^\infty(\Gamma)}
\le C t^{-\mu_0}(1+\max_{[t/4,t]} (\|\Psi_t\|_{L^\infty}^{\mu_0}+\|\nabla \Psi\|_{L^\infty}^2+\|\nabla^2\Psi\|_{L^\infty})) \\
\cdot \exp\Big\{C' t^{-\frac{1}{\delta_1}}(1 + \|\bar{p}_0\|_{L^\alpha(U)})^{z_3}(1 + [ Env A(\alpha,t)]^\frac1{\alpha-a})^{z_3}(1+\|\nabla \Psi\|_{L^{\alpha q_1}(U\times(0,t))}+\|\Psi_t\|_{L^{\alpha q_1}(U\times(0,t))})^{z_2}\Big\}.
\end{multline}

If $t> 1$ then
\begin{multline}\label{b2}
\|\nabla p(t)\|_{L^\infty(\Gamma)}
\le C (1+\max_{[t-1,t]} (\|\Psi_t\|_{L^\infty}^{\mu_0}+\|\nabla \Psi\|_{L^\infty}^2+\|\nabla^2\Psi\|_{L^\infty})) \\
\cdot \exp\Big\{C' (1 + \|\bar{p}_0\|_{L^\alpha(U)})^{z_3} (1+ [ Env A(\alpha,t)]^\frac1{\alpha-a})^{z_3}(1+\|\nabla \Psi\|_{L^{\alpha q_1}(U\times(t-1,t))}+\|\Psi_t\|_{L^{\alpha q_1}(U\times(t-1,t))})^{z_2}\Big\}.
\end{multline}

{\rm (ii)} If $A(\alpha)<\infty$ then
\begin{multline}\label{b3}
 \limsup_{t\to\infty} \|\nabla p(t)\|_{L^\infty(\Gamma)}
\le C (1+ \limsup_{t\to\infty}  (\|\Psi_t\|_{L^\infty}^{\mu_0}+\|\nabla \Psi\|_{L^\infty}^2+\|\nabla^2\Psi\|_{L^\infty})) \\
\cdot \exp\Big\{C' (1+   A(\alpha)^\frac1{\alpha-a})^{z_3}(1+\limsup_{t\to\infty}( \|\nabla \Psi(t)\|_{L^{\alpha q_1}( U\times (t-1,t))}+ \|\Psi_t(t)\|_{L^{\alpha q_1}( U\times (t-1,t))}))^{z_2}\Big\}.
\end{multline}

{\rm (iii)} If $\beta(\alpha)<\infty$ then there is $T>0$ such that for all $t\ge T$
\begin{multline}\label{b4}
\|\nabla p(t)\|_{L^\infty(\Gamma)}
\le C (1+\max_{[t-1,t]} (\|\Psi_t\|_{L^\infty}^{\mu_0}+\|\nabla \Psi\|_{L^\infty}^2+\|\nabla^2\Psi\|_{L^\infty})) \\
\cdot \exp\Big\{C' (1+  \beta(\alpha)^\frac1{\alpha-2a} + \|A(\alpha,\cdot)\|_{L^\frac\alpha{\alpha-a}(t-1,t)}^\frac1{\alpha-a} )^{z_3}( 1 +\|\nabla \Psi\|_{L^{\alpha q_1}( U\times (t-1,t))}+\|\Psi_t\|_{L^{\alpha q_1}( U\times (t-1,t))})^{z_2}\Big\}.
\end{multline}
\end{corollary}
\begin{proof}
{\rm (i)} For $0<t\le 3$, applying Theorem \ref{theo51} to $T_0=t/4$, $T=3t/4$, $\theta=1/3$ and using estimate \eqref{Ne2} we have
\begin{multline*}
\|\nabla p(t)\|_{L^\infty(\Gamma)}
\le C t^{-\mu_0}(1+\max_{[t/4,t]} (\|\Psi_t\|_{L^\infty}^{\mu_0}+\|\nabla \Psi\|_{L^\infty}^2+\|\nabla^2\Psi\|_{L^\infty})) \\
\cdot \exp\Big\{C' t^{-\frac{1}{\delta_1}}(1 + \|\bar{p}_0\|_{L^\alpha(U)} + [ Env A(\alpha,t)]^\frac1{\alpha-a})^{z_3}(1+\|\nabla \Psi\|_{L^{\alpha q_1}(U\times(0,t))}+\|\Psi_t\|_{L^{\alpha q_1}(U\times(0,t))})^{z_2}\Big\}.
\end{multline*}
Then \eqref{b1} follows.

When $t> 1$, applying Theorem \ref{theo51} to $T_0=t-1$, $T=1$, $\theta=1/2$ gives
\beq\label{g1}
\|\nabla p(t)\|_{L^\infty(\Gamma)}
\le C (1+\max_{[t-1,t]} (\|\Psi_t\|_{L^\infty}^{\mu_0}+\|\nabla \Psi\|_{L^\infty}^2+\|\nabla^2\Psi\|_{L^\infty})) 
\cdot \exp\Big\{C'\max_{[t-3/4,t]}\|\bar p\|_{L^\infty}\Big\}.
\eeq
Using estimate \eqref{Nsmall} we have
\begin{multline*}
\|\nabla p(t)\|_{L^\infty(\Gamma)}
\le C (1+\max_{[t-1,t]} (\|\Psi_t\|_{L^\infty}^{\mu_0}+\|\nabla \Psi\|_{L^\infty}^2+\|\nabla^2\Psi\|_{L^\infty})) \\
\cdot \exp\Big\{C' (1 + \|\bar{p}_0\|_{L^\alpha(U)} + [ Env A(\alpha,t)]^\frac1{\alpha-a})^{z_3}(1+\|\nabla \Psi\|_{L^{\alpha q_1}(U\times(t-1,t))}+\|\Psi_t\|_{L^{\alpha q_1}(U\times(t-1,t))})^{z_2}\Big\}.
\end{multline*}
Then \eqref{b2} follows.

{\rm (ii)} If $ A(\alpha)<\infty$ then taking  limit superior of \eqref{g1} and using \eqref{NlimI} yield
\begin{multline*}
 \limsup_{t\to\infty} \|\nabla p(t)\|_{L^\infty(\Gamma)}
\le C (1+ \limsup_{t\to\infty}  (\|\Psi_t\|_{L^\infty}^{\mu_0}+\|\nabla \Psi\|_{L^\infty}^2+\|\nabla^2\Psi\|_{L^\infty})) \\
\cdot \exp\Big\{C' (1+   A(\alpha)^\frac1{\alpha-a})^{z_3}(1+\limsup_{t\to\infty} (\|\nabla \Psi(t)\|_{L^{\alpha q_1}( U\times (t-1,t))}+\|\Psi_t(t)\|_{L^{\alpha q_1}( U\times (t-1,t))}))^{z_2}\Big\}.
\end{multline*}
Thus we obtain \eqref{b3}.

{\rm (iii)} Using \eqref{NPbarI} in \eqref{g1} gives
\begin{multline*}
\|\nabla p(t)\|_{L^\infty(\Gamma)}
\le C (1+\max_{[t-1,t]} (\|\Psi_t\|_{L^\infty}^{\mu_0}+\|\nabla \Psi\|_{L^\infty}^2+\|\nabla^2\Psi\|_{L^\infty})) \\
\cdot \exp\Big\{C' (1+  \beta(\alpha)^\frac1{\alpha-2a} + \|A(\alpha,\cdot)\|_{L^\frac\alpha{\alpha-a}(t-1,t)}^\frac1{\alpha-a} )^{z_3}( 1 +\|\nabla \Psi\|_{L^{\alpha q_1}( U\times (t-1,t))}+\|\Psi_t\|_{L^{\alpha q_1}( U\times (t-1,t))})^{z_2}\Big\}.
\end{multline*}
 for all $t\ge T$ with some $T>0$. This proves \eqref{b4}.
\end{proof}

\section{$L^s$-estimates for the gradient}
\label{PGradS}

In this section, we estimate the pressure gradient in $L^s$-norm for any $0<s<\infty$.
Throughout this section, $p(x,t)$ is a solution of \eqref{p:eq}.
First, we establish the basic step for the Ladyzhenskaya-Uraltseva iteration.

\begin{lemma} \label{new.it-L} Let $s> 0$, $T_0\ge 0$, and $T> T'\ge 0$.
Define 
\beqs M_b=\max_{\Gamma\times[T_0+T',T_0+T]}|\nabla p|\quad\text{and}
\quad v=\max\{|\nabla p|^2-M_b^2,0\}.
\eeqs

Let $\zeta(t)$ be a smooth cut-off function on $[T_0,T_0+T]$ with $\zeta=0$ on $[T_0,T_0+T']$.
Then
\beq\label{eLs2}
\sup_{[T_0,T_0+T]}  \int_U v^{s+1}(x,t) \zeta dx
+\int_{T_0}^T\int_U K(|\nabla p|) |\nabla^2 p|^2   v^s  \zeta dx dt
\le  C  \int_{T_0}^T \int_U v^{s+1}|\zeta_t| dx dt.
\eeq
\end{lemma}
\begin{proof} Without loss of generality, assume $T_0=0$. 
Note that $v\zeta =0$ on $\Gamma\times [0,T]$.
Denote $\chi=\chi_{\{v>0\}}$.
Multiplying the equation \eqref{p:eq} by 
$-\nabla \cdot (v^s \zeta \nabla p )$, integrating the resultant over $U$ and using integration by parts,
we obtain
\beq\label{eqgrad}
 \frac{1}{2s+2} \ddt \int_U v^{s+1} \zeta dx
= -\sum_{i,j=1}^n\int_U \partial_j (K(|\nabla p|)\partial_i p) \, \partial_i (v^s \partial_j p \zeta) dx +\frac 1{2s+2} \int_U v^{s+1} \zeta_t dx .
\eeq
Using the product rule for the first term in the right hand side of \eqref{eqgrad} (see Lemma 3.6 of \cite{HKP1}) we rewrite above equation as     
\begin{align*}
& \frac{1}{2s+2} \ddt \int_U v^{s+1} \zeta dx
= -\sum_{i,j,l=1}^n\int_U \Big[\partial_{y_l} (K(|y|) y_i)\Big|_{y=\nabla p} \partial_j\partial_l p\Big]\partial_j \partial_i p  v^s  \zeta dx\\
&\quad -s\sum_{i,j=1}^n \sum_{l,m=1}^n\int_U \Big[\partial_{y_l} (K(|y|) y_i)\Big|_{y=\nabla p} \partial_j\partial_l p \Big]\partial_j p   \, (v^{s-1} \partial_i\partial_m p\partial_m p\chi)\, \zeta dx+\frac 1{2s+2} \int_U v^{s+1} \zeta_t dx.
\end{align*}	
We denote the three terms  on the right-hand side by $I_1$, $I_2$,  and $I_3$. 
It follows from the calculations in Lemma 3.6 of \cite{HKP1} that  
\begin{align*}
 I_1&\le -(1-a)\sum_{j=1}^n \int_U K(|\nabla p|) |\nabla(\partial_j p)|^2   v^s  \zeta dx,\\
 I_2&\le  -(1-a)s\int_U K(| \nabla p|)\Big | \nabla \Big(\frac12|\nabla p|^2\Big)\Big |^2  v^{s-1} \chi  \zeta dx\le 0.
\end{align*}
Combining these estimates, we find that 
\beq\label{gradid}
\frac{1}{2s+2} \ddt \int_U v^{s+1} \zeta dx
+(1-a)  \int_U K(|\nabla p|) |\nabla^2 p|^2   v^s  \zeta dx
\le  \frac 1{2s+2} \int_U v^{s+1}\zeta_t dx.
\eeq
Inequality \eqref{eLs2} follows directly by integrating \eqref{gradid} from $0$ to $T$.
\end{proof}

In the following proposition, we iterate the inequality in Lemma \ref{new.it-L} in order to estimate $W^{1,s}$-norm of $p$ in term of its $W^{1,2-a}$ and $L^\infty$ norms.

\begin{lemma}\label{Claimlem}
Let $s\ge 1$, $T_0\ge 0$, $T>0$, and $0<\theta'<\theta<1$. Define
\beqs
M_b =\max_{[T_0+\theta' T,T_0+T]}\|\nabla p\|_{L^\infty(\Gamma)}\quad \text{and}\quad v=\max\{|\nabla p|^2-M_b^2,0\}.
\eeqs
Then
\beq\label{new2}
\int_{T_0+\theta T}^{T_0+T} \int_{U} K(|\nabla p|) v^{s} dx dt 
\le C T \mathcal U(s) +Cd_0^\frac{s-1}{2-a} \int_{T_0+\theta' T}^{T_0+T} \int_{U} (1+|\nabla p|^{2-a}) dxdt,
\eeq
where constant $C>0$ is independent of $T_0$, $T$, $\theta$, and $\theta'$,
\beq\label{Udef}
\mathcal U(s)=\begin{cases}
      0,&\text{if } 1\le s\le 3-a,\\
      d_0+d_0^{\frac{s-1}{2-a}+\frac12},&\text{if }s>3-a,
     \end{cases}
\eeq
\beq \label{d0def}
d_0=N_0^4((\theta-\theta') T)^{-2} + M_b^4,\quad\text{with} \quad N_0=\sup_{[T_0+\theta' T,T_0+T]}\|p\|_{L^\infty(U)}.
\eeq 
\end{lemma}
\begin{proof}
Without loss of generality, assume $T_0=0$.
The proof consists of three steps. 

\textbf{Step 1.} Let $\zeta(t)$ be the cut-off function with $\zeta=0$ for $t\le \theta' T$ and $\zeta=1$ for $t\ge \theta T$. For $s\ge 0$,  by applying Lemma \ref{LUK}  with $k=0$, and $s+1$ in place of $s$, multiplying \eqref{LUineq} by $\zeta^2(t)$ and integrating from $0$ to $T$, we  have 
\begin{align*} 
\int_{0}^T \int_U K(|\nabla p|) v^{s+2}\zeta^2 dx dt 
\le  C \max |p |^2 & \int_{0}^T\int_U K(|\nabla p|)|\nabla^2 p|^2 v^s \zeta^2dx dt  +  CM_b^4\int_{0}^T \int_U K(|\nabla p|)v^s \zeta^2 dx dt .
\end{align*}
To estimate the first integral on the right-hand side, we use \eqref{eLs2} with $\zeta^2$ in place of $\zeta$, and find that      
\beq\label{ladyural4}
\begin{aligned} 
&\int_{0}^T \int_U K(|\nabla p|) v^{s+2} \zeta^2dx dt  \le  CN_0^2\int_{0}^T \int_U   v^{s+1}\zeta \zeta_t dx dt+CM_b^4\int_{0}^T \int_U K(|\nabla p|)v^s \zeta^2dx dt .
\end{aligned}
 \eeq 
 Now using Young's inequality we have 
 \begin{multline*} 
\int_{0}^T \int_U K(|\nabla p|) v^{s+2}\zeta^2dx dt  \le  \frac 12 \int_{0}^T \int_U K(|\nabla p|) v^{s+2}\zeta^2 dx dt + CN_0^4\int_{0}^T \int_U  K^{-1}(|\nabla p|) v^{s}\zeta_t^2 dx dt\\
+CM_b^4\int_{\theta T}^T \int_U K(|\nabla p|)v^s \zeta^2dx dt .
\end{multline*}
 Thus  
  \beq\label{ev}
\begin{aligned} 
\int_{0}^T \int_U K(|\nabla p|) v^{s+2}\zeta^2 dx dt \le CN_0^4\int_{0}^T \int_U  K^{-1}(|\nabla p|) v^{s}\zeta_t^2 dx dt+CM_b^4\int_{\theta T}^T \int_U K(|\nabla p|)v^s \zeta^2dx dt.
\end{aligned}
 \eeq 
 Note that 
\begin{align*}
  K^{-1}(|\nabla p|)v^s &\le C K(|\nabla p|)(1+ |\nabla p|)^{2a}v^s \le C K(|\nabla p|)(1+  (v+M_b^2)^a )v^s\\
                                  & \le C K(|\nabla p|)(1+ v^{s+a}+ M_b^{2a}v^s)\le C K(|\nabla p|)(v^{s+a}+ 1+M_b^{2(s+a)}),
\end{align*}
and 
\beqs                              
 K(|\nabla p|)v^s \le C K(|\nabla p|)(1+v^{s+a}).
\eeqs
Thus \eqref{ev} implies 
 \beqs
\begin{aligned} 
\int_{0}^T \int_U K(|\nabla p|) v^{s+2}\zeta^2 dx dt 
&\le C[N_0^4((\theta-\theta')T)^{-2}+M_b^4] \int_{0}^T \int_UK(|\nabla p|)v^{s+a} \chi_{\{\zeta>0\}}dx dt\\
&\quad + C \int_{0}^T \int_U [N_0^4((\theta-\theta')T)^{-2}(1+ M_b^{2(s+a)}) +M_b^4 ]\chi_{\{\zeta>0\}}dxdt.
\end{aligned}
\eeqs 
Therefore,
 \beqs
\begin{aligned} 
\int_{\theta T}^T \int_U K(|\nabla p|) v^{s+2}dx dt 
&\le C[N_0^4((\theta-\theta')T)^{-2}+M_b^4] \int_{\theta' T}^T \int_UK(|\nabla p|)v^{s+a} dx dt\\
&\quad + CT[N_0^4((\theta-\theta')T)^{-2}(1+ M_b^2)^{s+a} +M_b^4 ].
\end{aligned}
 \eeqs 
 Let $N_1=TN_0^4((\theta-\theta')T)^{-2}$ and $N_2=Td_0$. Then
 \beq\label{ev1}
\begin{aligned} 
\int_{\theta T}^T \int_U K(|\nabla p|) v^{s+2}dx dt 
&\le Cd_0 \int_{\theta' T}^T \int_UK(|\nabla p|)v^{s+a} dx dt
 + CN_1 M_b^{2(s+a)} +CN_2.
\end{aligned}
 \eeq 
Equivalently, for $s\ge 2$
 \beq\label{ev10}
\begin{aligned} 
\int_{\theta T}^T \int_U K(|\nabla p|) v^{s}dx dt 
&\le Cd_0 \int_{\theta' T}^T \int_UK(|\nabla p|)v^{s-(2-a)} dx dt
 + CN_1 (1+ M_b^2)^{s-(2-a)} +CN_2.
\end{aligned}
 \eeq 


{\bf Step 2.} We prove \eqref{new2} for $s\in[1,3-a]$.
First,
\beq\label{Kv2}
\int_{\theta T}^T\int_{U} K(|\nabla p|)v dx dt\le \int_{\theta T}^T\int_{U} K(|\nabla p|)|\nabla p|^2 dx dt \le \int_{\theta T}^T\int_{U} (1+|\nabla p|^{2-a}) dx dt.
\eeq
 This yields \eqref{new2} for $s=1$. 

Second, let  $s=1-a$ in \eqref{ev}, 
\begin{align*}
&  \int_{\theta T}^T\int_{U} K(|\nabla p|)v^{3-a} dxdt  
 \le CN_0^4((\theta-\theta') T)^{-2}\int_{\theta' T}^T\int_{U}  (1+|\nabla p|)^{a}v^{1-a}  dx dt+ C M_b^4 \int_{\theta' T}^T\int_{U} v^{1-a} dxdt\\
&\le  CN_0^4((\theta-\theta') T)^{-2}\int_{0}^T\int_{U}  (1+|\nabla p|)^{a}|\nabla p|^{2(1-a)}  dx dt+ C M_b^4 \int_{0}^T\int_{U} |\nabla p|^{2(1-a)} dxdt.
\end{align*}
Therefore,
\beq\label{Kv7}
  \int_{\theta T}^T\int_{U} K(|\nabla p|)v^{3-a} dxdt  \le  Cd_0\int_{0}^T\int_{U}  (1+|\nabla p|^{2-a})  dx dt
\eeq
This implies \eqref{new2} when $s=3-a$. 

Third, when $s \in (1,3-a)$, let $\beta$ be the number in $(0,1)$ such that 
$\frac 1 {s} =\frac {1-\beta}{1}+\frac\beta {3-a}$. 
Then, by interpolation inequality:  
\begin{align*}
\Big(\int_{\theta T}^T\int_{U} K(|\nabla p|)v^{s} dx dt\Big)^{\frac 1 {s}} &\le \Big(\int_{\theta T}^T\int_{U} K(|\nabla p|) v dx dt\Big)^{1-\beta} \Big(\int_{\theta T}^T\int_{U} K(|\nabla p|) v^{3-a}dx dt\Big)^{\frac \beta {3-a}}.
\end{align*}
Using \eqref{Kv7} to estimate the last double integral, we obtain
\beqs
\int_{\theta T}^T\int_{U} K(|\nabla p|)v^{s} dx dt
\le Cd_0^{\beta s/(3-a) } \Big(\int_{0}^T\int_{U}  (1+|\nabla p|^{2-a})  dx dt \Big)^{s(1-\beta+\frac\beta{3-a})}.
\eeqs
Note that $\beta s/(3-a)  = (s-1)/(2-a) $. Then\beq\label{Kgrad24}
\int_{\theta T}^T\int_{U} K(|\nabla p|)v^{s} dx dt
\le Cd_0^\frac{s-1}{2-a}\int_{0}^T\int_{U}  (1+|\nabla p|^{2-a})  dx dt .
\eeq
Therefore, \eqref{Kv2} and \eqref{Kgrad24}  imply \eqref{new2} for $s\in[1,3-a]$. 

{\bf Step 3.} When $s>3-a$, let $m\in \N$ such that 
\beqs \frac {s-(3-a)}{2-a}\le m<\frac {s-(3-a)}{2-a}+1,
\eeqs 
then $s- m(2-a)\in (1,3-a]$. For each $k =1,2,\ldots, m$, let $\theta_k=\theta-(\theta-\theta')k/m$ and $\tilde Q_k = U \times (\theta_kT, T)$. Also, let $\zeta_k(t)$ be a smooth cut-off function 
 which is equal to one on $\tilde Q_k $ and zero on $Q_T\setminus \tilde Q_{k-1}$. There is a positive constant $c>0$ depending on $U$, such that 
 $|\zeta_{k,t}|\le mc[(\theta-\theta') T]^{-1}$ , for all $k = 1, 2,\ldots, m.$ 

From \eqref{ev10} we have 
\begin{align*}
&\int_{\theta T}^T\int_{U} K(|\nabla p|) v^{s} dx dt 
\le Cd_0^m\int_{\theta_m T}^T \int_UK(|\nabla p|)v^{s-m(2-a)} dx dt\\
&\quad + C N_1\Big[M_b^{2(s-(2-a))}+d_0M_b^{2(s-2(2-a))}+d_0^2M_b^{2(s-3(2-a))}+\ldots+d_0^{m-1}M_b^{2(s-m(2-a))}\Big]\\
&\quad + C N_2(1+d_0+d_0^2+\ldots+d_0^{m-1}).
\end{align*}
Using inequality $ \sum_{i=0}^{m-1} z^i \le (1+z)^{m-1},  \forall z>0$, we infer
\begin{align*}
&\int_{\theta T}^T\int_{U} K(|\nabla p|) v^{s} dx dt 
\le Cd_0^m\int_{\theta_m T}^T \int_UK(|\nabla p|)v^{s-m(2-a)} dx dt\\
&\quad + C N_1M_b
^{2(s-(2-a))}[1+d_0M_b^{-2(2-a)}]^{m-1} + CTd_0(1+d_0)^{m-1}.
\end{align*}
Therefore we obtain   
\begin{multline*}
\int_{\theta T}^T\int_{U} K(|\nabla p|) v^{s} dx dt 
\le Cd_0^m\int_{\theta_m T}^T \int_UK(|\nabla p|)v^{s-m(2-a)} dx dt\\
 + C N_1M_b^{2(s-(2-a))}[1+d_0M_b^{-2(2-a)}]^{m-1} + CTd_0(1+d_0)^{m-1}.
\end{multline*}

Since $s-m(2-a)\in(1-a,3-a]$, estimating the double integral on the right-hand side  by \eqref{Kgrad24} gives
\begin{multline*}
\int_{\theta T}^T \int_{U} K(|\nabla p|) v^{s} dx dt 
\le C N_1M_b^{2(s-(2-a))}[1+d_0M_b^{-2(2-a)}]^{m-1} + CTd_0(1+d_0)^{m-1}\\ 
+Cd_0^{\frac{s-m(2-a)-1}{2-a}+m} \int_{\theta_m T}^T \int_{U} (1+|\nabla p|^{2-a}) dxdt.
\end{multline*}

Since $m\le \frac{s-1}{2-a}$ and $\frac{s-m(2-a)-1}{2-a}+m =\frac{s-1}{2-a}$, then 
\begin{multline*}
\int_{\theta T}^T \int_{U} K(|\nabla p|) v^{s} dx dt 
\le C N_1M_b^{2(s-(2-a))}[1+d_0M_b^{-2(2-a)}]^{\frac{s-1}{2-a}-1} + CTd_0(1+d_0)^{\frac{s-1}{2-a}-1}\\ 
+Cd_0^\frac{s-1}{2-a} \int_{\theta_m T}^T \int_{U} (1+|\nabla p|^{2-a}) dxdt
\le C N_1(M_b^{2(s-(2-a))} +d_0^{\frac{s-1}{2-a}-1} M_b^2) + CT(d_0+d_0^\frac{s-1}{2-a})\\ 
+Cd_0^\frac{s-1}{2-a} \int_{\theta_m T}^T \int_{U} (1+|\nabla p|^{2-a}) dxdt.
\end{multline*}
Note that
\begin{align*}
&N_1(M_b^{2(s-(2-a))} +d_0^{\frac{s-1}{2-a}-1} M_b^2) + T(d_0+d_0^\frac{s-1}{2-a})
\le T(d_0 d_0^{2(s-(2-a))/4} +d_0 d_0^{\frac{s-1}{2-a}-1} d_0^{1/2} + d_0+d_0^\frac{s-1}{2-a})\\
&=T(d_0^\frac{s+a}2+d_0^{\frac{s-1}{2-a}+\frac12}+d_0+d_0^\frac{s-1}{2-a})
\le CT(d_0^\frac{s+a}2+d_0^{\frac{s-1}{2-a}+\frac12}+d_0)\le CT(d_0+d_0^{\frac{s-1}{2-a}+\frac12}),
\end{align*}
since $1<\frac{s+a}2<\frac{s-1}{2-a}+\frac12$ for $s>3-a$.
So,
\beq\label{new5}
\int_{\theta T}^T \int_{U} K(|\nabla p|) v^{s} dx dt 
\le C T \mathcal U(s) +Cd_0^\frac{s-1}{2-a} \int_{\theta' T}^T \int_{U} (1+|\nabla p|^{2-a}) dxdt.
\eeq
This completes the proof of \eqref{new2} for all $s\ge 1$.
\end{proof}

From the estimates of $K(|\nabla p|)v^s$ is Lemma \ref{Claimlem}, we derive direct estimates for $|\nabla p|^s$.

\begin{proposition}\label{Prop53}
Let $T_0\ge 0$, $T>0$ and $\theta\in (0,1)$. If $s\ge  2-a$ then
 \beq\label{new0}
\int_{T_0+\theta T}^{T_0+T} \int_{U}  |\nabla p|^{s} dx dt 
\le C T (1+\mathcal D^{\tilde s}) +  C \mathcal D^{\frac{s}{2-a}-1}\int_{T_0}^{T_0+T} \int_{U} |\nabla p|^{2-a} dxdt,
\eeq
and if $s>2$ then
\beq\label{new1}
\sup_{[T_0+\theta T,T_0+T]} \int_U |\nabla p|^s dx
\le C\theta^{-1}(1+\mathcal D^{\tilde s})
+ C (\theta T)^{-1}\mathcal D^{\frac{s}{2-a} -1}\int_{T_0}^{T_0+T} \int_{U} |\nabla p|^{2-a} dxdt,
\eeq
where constant $C>0$ is independent of $T_0$, $T$, and $\theta$,
\beq\label{stil}
\tilde s =
\begin{cases}
\max\{\frac{s+a}2,\frac{s}{2-a}-1\},&\text{if }2-a\le s\le 3(2-a),\\
\frac{s}{2-a},&\text{if }s>3(2-a),
\end{cases}
\eeq
and 
\beq \label{Ddef}
\mathcal D =(\theta T)^{-1} \sup_{[T_0+\theta T/2,T_0+T]}\|p\|_{L^\infty(U)}^2 + \sup_{[T_0+\theta T/2,T_0+T]}|\nabla p|^2_{L^\infty(\Gamma)}.
\eeq 
\end{proposition}
\begin{proof} 
Without loss of generality, assume $T_0=0$. Let $M_b$, $v$ and $d_0$ be defined as in Lemma \ref{Claimlem}.

{\bf Proof of \eqref{new0}.}  From the definition of the function $v$ and inequality $|a-b|^\gamma \ge 2^{-\gamma} a^\gamma -b^\gamma$ for $a,b\ge 0, \gamma>0$, we have
\beqs
\begin{split}
  \int_{\theta T}^T\int_{U} K(|\nabla p|)|\nabla p|^{2s} dxdt  
\le \int_{\theta T}^T \int_{U} K(|\nabla p|) v^s dxdt
+CTM_b^{2s}.
\end{split}
\eeqs

We apply \eqref{new2} with $\theta'=\theta/2$ to estimate the integral on the right-hand side. It results in 
\beq\label{new6}
\int_{\theta T}^T \int_{U} K(|\nabla p|) |\nabla p|^{2s} dx dt 
\le C T \mathcal U(s)+CTM_b^{2s} +Cd_0^\frac{s-1}{2-a} \int_{0}^T \int_{U} (1+|\nabla p|^{2-a}) dxdt.
\eeq
For $s>3-a$, then by definition \eqref{Udef} and then Young's inequality, it follows
\begin{multline*}
\int_{\theta T}^T \int_{U} K(|\nabla p|) |\nabla p|^{2s} dx dt 
\le C T (d_0+d_0^{\frac{s-1}{2-a}+\frac12} + d_0^{s/2} +d_0^\frac{s-1}{2-a}) + C d_0^\frac{s-1}{2-a}\int_{0}^T \int_{U} |\nabla p|^{2-a} dx\\
\le C T (1+d_0^{\frac{s-1}{2-a}+\frac12}) +  C d_0^\frac{s-1}{2-a}\int_{0}^T \int_{U} |\nabla p|^{2-a} dxdt.
\end{multline*}
Estimating $K(|\nabla p|)$ from below by \eqref{Kestn}, we get
\beq\label{gb1}
\int_{\theta T}^T \int_{U} |\nabla p|^{2s-a} dx dt 
\le C T (1+d_0^{\frac{s-1}{2-a}+\frac12}) +  C d_0^\frac{s-1}{2-a}\int_{0}^T \int_{U} |\nabla p|^{2-a} dxdt.
\eeq
Replacing $2s-a$ by $s>2(3-a)-a=3(2-a)$ in \eqref{gb1} we obtain
\beq\label{new7}
\int_{\theta T}^T |\nabla p|^{s} dx dt 
\le C T (1+d_0^{ \frac{s}{2(2-a)} } ) +  C d_0^{ \frac{s-(2-a)}{2(2-a)}} \int_{0}^T \int_{U} |\nabla p|^{2-a} dxdt.
\eeq

For $1\le s\le 3-a$, we have from \eqref{new6} that
\beqs
\int_{\theta T}^T \int_{U} K(|\nabla p|) |\nabla p|^{2s} dx dt 
\le CT(d_0^{s/2} + d_0^\frac{s-1}{2-a}) +Cd_0^\frac{s-1}{2-a} \int_{0}^T \int_{U} |\nabla p|^{2-a} dxdt.
\eeqs
Using property \eqref{Kestn} again yields
\beqs
\int_{\theta T}^T \int_{U} |\nabla p|^{2s-a} dx dt 
\le CT(1+d_0^{s/2} + d_0^\frac{s-1}{2-a}) +Cd_0^\frac{s-1}{2-a} \int_{0}^T \int_{U} |\nabla p|^{2-a} dxdt .
\eeqs
Replacing $2s-a$ by $s\in [2-a,3(2-a)]$ gives
\begin{multline}\label{new8}
\int_{\theta T}^T \int_{U} |\nabla p|^{s} dx dt 
\le CT(1+d_0^\frac{s+a}{4} + d_0^\frac{s-(2-a)}{2(2-a)}) +Cd_0^\frac{s-(2-a)}{2(2-a)} \int_{0}^T \int_{U} |\nabla p|^{2-a} dx dt\\
\le CT(1+d_0^{\max\{\frac{s+a}{4},\frac{s-(2-a)}{2(2-a)} \} }) +Cd_0^\frac{s-(2-a)}{2(2-a)} \int_{0}^T \int_{U} |\nabla p|^{2-a} dx dt.
\end{multline}

Note that $d_0^{1/2} \le C\mathcal D$. Then combing \eqref{new7}, for the case $s>3(2-a)$, with \eqref{new8}, for the case  $s\in [2-a,3(2-a)]$, we obtain \eqref{new0} for all $s\ge 2-a$.
   
{\bf Proof of \eqref{new1}.}  For $s>2$, we have
from \eqref{eLs2} with $\theta'=\theta/2$ that 
\begin{align*}
\sup_{[0,T]} \int_U |\nabla p|^s \zeta dx
&\le C\sup_{[0,T]} \int_U (v^{s/2} +M_b^{s}) \zeta dx
\le CM_b^{s} +\int_0^T \int_U v^{s/2} \zeta_t dxdt\\
&\le   C  M_b^{s} +C\int_0^T \int_U |\nabla p|^{s/2}|\zeta_t| dxdt.
\end{align*}
Choose $\zeta(t)$ such that $\zeta=0$ for $t\le 3\theta T/4$, and $\zeta =1$ for $t\ge \theta T$. 
Then
\beq\label{new4}
\sup_{[\theta T,T]} \int_U |\nabla p|^s \zeta dx
\le  C d_0^{s/4}
+ C (\theta T)^{-1}\int_{\frac {3\theta T}4}^T \int_U |\nabla p|^{s} dx dt,
\eeq
with $d_0$ being defined by using, again, $\theta'=\theta/2$.
We apply \eqref{new0} with $\theta$ being replaced by $\theta_1=3\theta/4$ and $d_0$, hence $\mathcal D$, defined by using $\theta'=2\theta_1/3$.  We obtain
\begin{align*}
\sup_{[\theta T,T]} \int_U |\nabla p|^s dx
&\le  C \mathcal D^{s/2}+C\theta^{-1}(1+\mathcal D^{\tilde s})
+ C (\theta T)^{-1}\mathcal D^{\frac{s}{2-a}-1} \int_{0}^T \int_{U} |\nabla p|^{2-a} dxdt \\
&\le C\theta^{-1}(1+\mathcal D^{\tilde s})
+ C (\theta T)^{-1}\mathcal D^{\frac{s}{2-a}-1} \int_{0}^T \int_{U} |\nabla p|^{2-a} dxdt.
\end{align*}
Thus, we obtain \eqref{new1}. This completes the proof of the theorem.
\end{proof}


Now, we combine Proposition \ref{Prop53} with estimates in section \ref{revision} to express the bounds for the gradient's $L^s$-norms in terms of the initial and boundary data.

In the following $\alpha$ always satisfies \eqref{alcond}. 
Let $p_1$, $q_1$ be fixed and $\delta_1$ be defined as in Theorem \ref{Ntheo42}.
Also, let $z_1$, $z_2$, $z_3$ be defined as in Corollary \ref{Nsim}, and $\mu_0$ be defined by \eqref{mu0}.
We define a number of constants and quantities that will be used for the rest of the paper:
\beq\label{kap1}
\kappa_1=\max\{2\mu_0,1+2/\delta_1\},
\eeq
\beq\label{K1}
\Kzero(t)=1+ [Env A(\alpha,t)]^\frac{1}{\alpha-a},\quad 
\Kth(t)=1+  \beta(\alpha)^\frac {1}{\alpha-2a} + \sup_{[t-2, t]}A(\alpha,\cdot)^\frac{1}{\alpha-a},
\eeq
\begin{multline}\label{K2}
\Kfirst(t)=1+\sup_{[0,t]} (\|\Psi\|_{L^\infty}+\|\Psi_t\|_{L^\infty}^{\mu_0}+\|\nabla \Psi\|_{L^\infty}^2+ \|\nabla^2 \Psi\|_{L^\infty})\\
+(\|\nabla \Psi\|_{L^{\alpha q_1}(U\times(0,t))}+\|\Psi_t\|_{L^{\alpha q_1}(U\times(0,t))})^{z_2},
\end{multline}
\begin{multline}\label{bK2}
\Kmore(t)=1+\sup_{[t-2,t]} (\|\Psi\|_{L^\infty}+\|\Psi_t\|_{L^\infty}^{\mu_0}+\|\nabla \Psi\|_{L^\infty}^2+ \|\nabla^2 \Psi\|_{L^\infty})\\
+(\|\nabla \Psi\|_{L^{\alpha q_1}(U\times(t-2,t))}+\|\Psi_t\|_{L^{\alpha q_1}(U\times(t-2,t))})^{z_2}.
\end{multline}


\begin{theorem}\label{gradsthm}
Let $s> 2$. 

{\rm (i)} If $0<t\le 3$ then 
\begin{multline}\label{es5}
\int_{U} |\nabla p(x,t)|^s dx 
\le Ct^{-1-\tilde s \kappa_1} (1 + \|\bar{p}_0\|_{L^\alpha(U)})^{2\tilde s z_3+2} \Kzero(t)^{2\tilde s z_3} \Kfirst(t)^{2\tilde s}
\Big( 1+\int_0^t G_1(\tau)d\tau\Big)\\
\cdot \exp\Big\{C' t^{-1/\delta_1}(1 + \|\bar{p}_0\|_{L^\alpha(U)})^{z_3} \Kzero(t)^{z_3}(1+\|\nabla \Psi\|_{L^{\alpha q_1}(U\times(0,t))}+\|\Psi_t\|_{L^{\alpha q_1}(U\times(0,t))})^{z_2}\Big\}.
\end{multline}

{\rm (ii)} If $t> 2$ then
\begin{multline}\label{es6}
\int_{U} |\nabla p(x,t)|^s dx 
\le C(1 + \|\bar{p}_0\|_{L^\alpha(U)})^{2\tilde s z_3+\alpha}
\Kzero(t)^{2\tilde s z_3+\alpha} \Kmore(t)^{2\tilde s}
\Big(1+\int_{t-1}^t G_1(\tau)d\tau\Big)\\
\cdot \exp\Big\{C' (1 + \|\bar{p}_0\|_{L^\alpha(U)}^{z_3})\Kzero(t)^{z_3}
 (1+ \|\nabla \Psi\|_{L^{\alpha q_1}( U\times (t-2,t))}+ \|\Psi_t\|_{L^{\alpha q_1}( U\times (t-2,t))})^{z_2} \Big\}.
\end{multline}
\end{theorem}
\begin{proof}
(i)   Let $t\in(0,3]$. Applying \eqref{new1} to $T_0=0$, $T=t$ and $\theta=1/2$ gives
 \beqs
\int_U |\nabla p(x,t)|^s dx
\le C\Big( 1+\mathcal D^{\tilde s}+ C t^{-1}\mathcal D^{\frac{s}{2-a} -1}\int_{0}^{t} \int_{U} |\nabla p|^{2-a} dxd\tau\Big),
\eeqs
where $\mathcal D$ is defined by \eqref{Ddef}. Since $\tilde s \ge \frac{s}{2-a}-1$, then
\beq\label{new11}
\int_U |\nabla p(x,t)|^s dx
\le Ct^{-1}( 1+\mathcal D)^{\tilde s}\Big(1+\int_{0}^{t} \int_{U} |\nabla p|^{2-a} dxd\tau\Big),
\eeq
Note that
 \begin{align*}
\mathcal D 
&\le Ct^{-1} \sup_{[t/4,t]} \|p\|_{L^\infty}^2 + \sup_{[t/4,t]}  \|\nabla p\|^2_{L^\infty(\Gamma)}\\
&\le Ct^{-1} \sup_{[t/4,t]} \|\bar p\|_{L^\infty}^2 +Ct^{-1} \sup_{[t/4,t]} \|\Psi\|_{L^\infty}^2+ \sup_{[t/4,t]}  \|\nabla p\|^2_{L^\infty(\Gamma)}.
\end{align*}
To estimate $\|\bar p\|_{L^\infty}$ we use  \eqref{Ne2}.
To estimate    $\|\nabla p\|_{L^\infty(\Gamma)}$, we apply \eqref{b1} of Corollary \ref{Cor5}.
We obtain
\beqs
\mathcal D\le C(t^{-1-2/\delta_1}R_1^2+t^{-1}R_2^2 + t^{-2\mu_0}R_3^2 e^{C' t^{-1/\delta_1}R_1})
\eqdef C\bar{\mathcal D},
\eeqs
where
\beqs
R_1=(1 + \|\bar{p}_0\|_{L^\alpha})^{z_3}\Kzero(t)^{z_3}(1+\|\nabla \Psi\|_{L^{\alpha q_1}(U\times(0,t))}+\|\Psi_t\|_{L^{\alpha q_1}(U\times(0,t))})^{z_2},
\eeqs
\beqs
 R_2=\sup_{[0,t]}\|\Psi\|_{L^\infty} ,\quad R_3=1+\sup_{[0,t]}( \|\Psi_t\|_{L^\infty}^{\mu_0}+\|\nabla \Psi\|_{L^\infty}^2+ \|\nabla^2 \Psi\|_{L^\infty}).
\eeqs
By relation \eqref{Hcompare}, we have $|\nabla p|^{2-a}\le C(1+H(|\nabla p|))$, hence we  can use \eqref{t0}  to estimate the last integral in \eqref{new11}.
Substituting these estimates  into \eqref{new11} gives
 \beq\label{es1}
\int_{U} |\nabla p(x,t)|^s dx 
\le Ct^{-1}\bar {\mathcal D}^{\tilde s} \Big( 1+\|\bar p_0\|_{L^2}^2 +\int_0^t G_1(\tau)d\tau\Big),
\eeq
hence
\beq\label{es2}
\int_{U} |\nabla p(x,t)|^s dx 
\le Ct^{-1-\tilde s \kappa_1} (R_1+R_2+R_3)^{2\tilde s} R_4 e^{ C' t^{-1/\delta_1}R_1},
\eeq
where 
$R_4=(1+\|\bar p_0\|_{L^2})^2 (1+\int_0^t G_1(\tau)d\tau)$.
For the sum $R_1+R_2+R_3$, we clearly see that $R_2,R_3\le \Kfirst(t)$, and 
\beqs
R_1\le C(1 + \|\bar{p}_0\|_{L^\alpha(U)})^{z_3}  \Kzero(t)^{z_3}\Kfirst(t).
\eeqs
Therefore
\begin{multline}\label{Grad-est}
\int_{U} |\nabla p(x,t)|^s dx 
\le Ct^{-1-\tilde s \kappa_1} (1+\|\bar p_0\|_{L^2})^2 (1 + \|\bar{p}_0\|_{L^\alpha})^{2\tilde s z_3}[\Kzero(t)^{ z_3}\Kfirst(t)]^{2\tilde s}
\Big( 1+\int_0^t G_1(\tau)d\tau\Big)\\
\cdot \exp\Big\{C t^{-1/\delta_1}R_1\Big\}.
\end{multline}
Then by H\"older's inequality, we can combine the powers of $L^2$ and $L^\alpha$ norms of $\bar p_0$, and obtain \eqref{es5}.

(ii)  Let $t> 2$. Applying \eqref{new1} with $T_0=t-1$, $T=1$, $\theta=1/2$, 
  \begin{multline}\label{new10}
\int_U |\nabla p(x,t)|^s dx
\le C(1+\mathcal D^{\tilde s})
+ C \mathcal D^{\frac{s}{2-a} -1}\int_{t-1}^{t} \int_{U} |\nabla p|^{2-a} dxd\tau\\
\le C(1+\mathcal D)^{\tilde s}\Big[1+\int_{t-1}^{t} \int_{U} |\nabla p|^{2-a} dxd\tau\Big],
\end{multline}
where $\mathcal D$ is defined in \eqref{Ddef}. We have
\beq \label{Dlarge}
  \begin{aligned}
   \mathcal D 
 &\le C \sup_{[t-1, t]}  \Big[\norm{p}_{L^\infty} +  \norm{\nabla p}_{L^\infty(\Gamma)}  \Big]^2\\
 &\le C  \sup_{[t-1, t]} \Big[ \norm{\bar p}_{L^\infty} +  \norm{\Psi}_{L^\infty} + \norm{\nabla p}_{L^\infty(\Gamma)}  \Big]^2.
 \end{aligned}
  \eeq
Similar to part (i), but using \eqref{Nsmall} and \eqref{b2} instead of \eqref{Ne2} and \eqref{b1},
we have
  \beq\label{D1}
 \mathcal D \le C(R_5^2 +R_6^2+R_7^2e^{C'R_5}),
 \eeq
 where
 \begin{align*}
 R_5&=(1+ \|\bar{p}_0\|_{L^\alpha} )^{z_3}\Kzero(t)^{z_3}(1+ \|\nabla \Psi\|_{L^{\alpha q_1}( U\times (t-2,t))}+ \|\Psi_t\|_{L^{\alpha q_1}( U\times (t-2,t))})^{z_2} ,\\
 R_6&=\sup_{[t-1, t]}\|\Psi\|_{L^\infty},\quad
 R_7=1+\max_{[t-2,t]} ( \|\Psi_t\|_{L^\infty}^{\mu_0}+\|\nabla \Psi\|_{L^\infty}^2+\|\nabla^2 \Psi\|_{L^\infty}).
 \end{align*}
 Combining  \eqref{new10}, \eqref{D1}  and \eqref{all1} gives
 \begin{multline}\label{larget10}
\int_{U} |\nabla p(x,t)|^s dx 
\le C (R_5^2 +R_6^2+R_7^2e^{C'R_5})^{\tilde s} \Big( 1+\|\bar p_0\|_{L^\alpha}^\alpha +Env A(\alpha,t)^\frac{\alpha}{\alpha-a}+\int_{t-1}^t G_1(\tau)d\tau\Big)\\
\le C (R_5 +R_6+R_7e^{C'R_5})^{2\tilde s} ( 1+\|\bar p_0\|_{L^\alpha})^\alpha \Kzero(t)^\alpha\Big(1+\int_{t-1}^t G_1(\tau)d\tau\Big).
\end{multline}
Note that $R_6,R_7\le \Kmore(t)$ and
\beqs
 R_5\le C(1+ \|\bar{p}_0\|_{L^\alpha} )^{z_3} \Kzero(t)^{z_3}\Kmore(t)^{z_2}.
\eeqs
Therefore, similar to part (i), we obtain \eqref{es6} from \eqref{larget10}.
\end{proof}

\begin{remark}
{\rm (a)} Regarding the regularity requirement for the initial data $p_0(x)$ in estimates of $ \|\nabla p\|_{L^s}$ in later time $t>0$, the right-hand side of \eqref{es5} and \eqref{es6} only need $\|p_0\|_{L^\alpha}$. This shows the regularity gain  of the solution when $t>0$.

{\rm (b)} The $W^{1,s}$-estimates established above hold for $s>n$, hence the H\"older estimates  follow thanks to Morrey's embedding. This is opposite to Ladyzhenskaya-Uraltseva's proofs in \cite{LadyParaBook68} which establish  H\"older  continuity first and use it to obtain $W^{1,s}$-estimates, \underline{but not for all $s$}. 
The reason for this simplification is our equation's special structure, see \eqref{lin-p}--\eqref{Hcompare}.   
\end{remark}

For large time estimates, the bounds in Theorem \ref{gradsthm} can be established to be independent of the initial data as in the following. 

\begin{theorem}\label{theo55}
  
{\rm (i)} If $A(\alpha)<\infty$ then
\begin{multline}\label{LsupGradp-s}
\limsup_{t\to\infty}\int_{U} |\nabla p(x,t)|^s dx
 \le C\Azero^{2\tilde s z_3+\alpha} \Amore^{2\tilde s} \Afirst \\
\cdot\exp\Big\{  C'\Azero^{z_3}(1+\limsup_{t\to\infty}( \|\nabla \Psi(t)\|_{L^{\alpha q_1}( U\times (t-1,t))}+ \|\Psi_t(t)\|_{L^{\alpha q_1}( U\times (t-1,t))}))^{z_2}
\Big\} ,
\end{multline}
where
\beq\label{tilA1}
\Azero=1+A(\alpha)^\frac {1} {\alpha-a},\quad
\Afirst=1+\limsup_{t\to\infty} \int_{t-1}^tG_1(\tau)d\tau,
\eeq
\begin{multline}\label{tilA2}
\Amore=1+ \limsup_{t\to\infty} (\|\Psi(t)\|_{L^\infty}+\|\Psi_t(t)\|_{L^\infty}^\frac2{2-a}+\|\nabla \Psi(t)\|_{L^\infty}^2+\|\nabla^2 \Psi(t)\|_{L^\infty}) \\
+ \limsup_{t\to\infty}(\|\nabla \Psi(t)\|_{L^{\alpha q_1}( U)}+\|\Psi_t(t)\|_{L^{\alpha q_1}( U)})^{z_2}.
\end{multline}

{\rm (ii)} If $\beta(\alpha)<\infty$ then there is $T>0$ such that for all $t>T$,
\begin{multline}\label{Gradp-sL}
\int_{U} |\nabla p(x,t)|^s dx \le C\Kth(t)^{2\tilde s z_3+\alpha}
\Kmore(t)^{2\tilde s}\Big (1+\int_{t-1}^tG_1(\tau)d\tau\Big )\\
\cdot\exp\Big\{  C'\Kth(t)^{z_3} (1+\|\nabla \Psi(t)\|_{L^{\alpha q_1}( U\times (t-2,t))}+\|\Psi_t(t)\|_{L^{\alpha q_1}( U\times (t-2,t))})^{z_2}
\Big\}.
\end{multline}
\end{theorem}
\begin{proof} (i) Taking the limit superior as $t\to\infty$  in \eqref{new10}  and \eqref{Dlarge} gives
  \beq\label{Grads0}
  \begin{aligned}
  &\limsup_{t\to\infty}\int_{U}|\nabla p(x,t)|^s dx dt\\
 & \le  C(1+\limsup_{t\to\infty}(\norm{p}_{L^\infty(U)}+\norm{\Psi}_{L^\infty(U)}  + \norm{\nabla p}_{L^\infty(\Gamma)} ))^{2\tilde s} \cdot (1+\limsup_{t\to\infty} \int_{t-1}^{t} \int_{U} |\nabla p|^{2-a} dxd\tau).
  \end{aligned}
  \eeq
Applying \eqref{NlimI}, \eqref{b3} and \eqref{lim3} gives
\begin{multline*}
\limsup_{t\to\infty}\int_{U} |\nabla p(x,t)|^s dx \\
\le C\Big[\Azero^{z_3}(1+\limsup_{t\to\infty} (\|\nabla \Psi(t)\|_{L^{\alpha q_1}( U\times (t-1,t))}+\|\Psi_t(t)\|_{L^{\alpha q_1}( U\times (t-1,t))}))^{z_2}\\
+\limsup_{t\to\infty}\norm{\Psi}_{L^\infty(U)} 
+\limsup_{t\to\infty}(1+\|\Psi_t\|_{L^\infty}^\frac2{2-a}+\|\nabla \Psi\|_{L^\infty}^2+ \|\nabla^2 \Psi\|_{L^\infty})\\
\cdot \exp\{C'\Azero^{z_3}(1+\limsup_{t\to\infty} \|\Psi_t(t)\|_{L^{\alpha q_1}( U\times (t-1,t))})^{z_2}
\}\Big]^{2\tilde s}
\times
 \Big (1+\Azero^\alpha +\limsup_{t\to\infty} \int_{t-1}^tG_1(\tau)d\tau\Big).
\end{multline*}
Then estimate \eqref{LsupGradp-s} follows. 

(ii)  By combining \eqref{new10}, \eqref{Dlarge} with \eqref{NPbarI}, \eqref{b4}  and \eqref{large3}, we have
\begin{multline*}
\int_{U} |\nabla p(x,t)|^s dx \\
\le C \Big\{
[1+  \beta(\alpha)^\frac 1{\alpha-2a} + \norm{A(\alpha,\cdot)}_{L^{\frac\alpha{\alpha -a}}(t-2,t) }^\frac 1{\alpha-a} \big ]^{z_3} [1+\|\nabla \Psi\|_{L^{\alpha q_1} U\times (t-2,t))}+\|\Psi_t\|_{L^{\alpha q_1} U\times (t-2,t))}]^{z_2} \\
+  \sup_{[t-1,t]} \|\Psi\|_{L^{\infty}(U)}
+ (1+\sup_{[t-2,t]}(\|\Psi_t\|_{L^\infty}^{\mu_0}+\|\nabla \Psi\|_{L^\infty}^2+ \|\nabla^2 \Psi\|_{L^\infty}))\\
\cdot \exp\{C'
[1+  \beta(\alpha)^\frac 1{\alpha-2a} + \norm{A(\alpha,\cdot)}_{L^{\frac\alpha{\alpha -a}}(t-2,t) }^\frac 1{\alpha-a} \big ]^{z_3} [1+\|\nabla \Psi\|_{L^{\alpha q_1} U\times (t-2,t))}+\|\Psi_t\|_{L^{\alpha q_1} U\times (t-2,t))}]^{z_2} \}\Big\}^{2\tilde s} \\
\cdot \Big (1+\beta(\alpha)^\frac {\alpha}{\alpha-2a} + A(\alpha,t-1)^\frac{\alpha}{\alpha-a} +\int_{t-1}^tG_1(\tau)d\tau\Big )
\end{multline*}
for all $t\ge T$ with some $T>2$.
Note that $\norm{A(\alpha,\cdot)}_{L^{\frac\alpha{\alpha -a}}(t-2,t) }^\frac 1{\alpha-a} \le 2^{1/\alpha}\sup_{[t-2,t]} A(\alpha, \tau)^{1/(\alpha-a)}$.
Then
\begin{multline*}
\int_{U} |\nabla p(x,t)|^s dx \\
\le C \Big\{
\Kth(t)^{z_3} \Kmore(t)+ \Kmore(t) \exp\{C' \Kth(t)^{z_3} [1+\|\nabla \Psi\|_{L^{\alpha q_1} U\times (t-2,t))}+\|\Psi_t\|_{L^{\alpha q_1} U\times (t-2,t))}]^{z_2} \}\Big\}^{2\tilde s} \\
\cdot \Kth(t)^\alpha \Big (1+ \int_{t-1}^tG_1(\tau)d\tau\Big ).
\end{multline*}
Therefore, inequality \eqref{Gradp-sL} follows.
\end{proof}

\begin{remark}
In case $\Psi=\Psi(x)$, we want to investigate the long-time dynamics by using the notion of global attractors, see, e.g.,  \cite{TemamDynBook,SellYouBook}. However, the independence of estimates  \eqref{LsupGradp-s} and \eqref{Gradp-sL} on the initial data does not yet prove the existence of a absorbing ball. Therefore the existence of the global attractor is still an open problem.
\end{remark}

\section{$L^\infty$-estimates for the gradient}
\label{PGradInfty}

In this section, we obtain  global $L^\infty$-estimates for the gradient of pressure.
Let $p(x,t)$ be a solution of IBVP \eqref{p:eq}.
For each $m=1,2,\ldots,n,$ denote $u_m=p_{x_m}$ and $u=(u_1,u_2,\ldots,u_n)=\nabla p$. 
We have
\beq\label{dum}
\frac{\partial u_{m}}{\partial t}=\partial_m  ( \nabla \cdot (  K(|u|)  u) )= \nabla \cdot (K(|u|) \partial_m u) +\nabla \cdot \Big[ K'(|u|) \frac{\sum_{i} u_i \partial_m u_i }{|u|} u\Big].          
\eeq
Since $\partial_i u_m=\partial_m u_i$, we have 
$\partial_m u = ( \partial_m u_1,\ldots, \partial_m u_n ) =( \partial_1 u_m, \ldots, \partial_n u_m) =
\nabla u_m$, and
$\sum_{i} u_i \partial_m u_i=\sum_{i} u_i \partial_i u_m=u\cdot \nabla u_m$.
Therefore, we rewrite \eqref{dum} as
\beq \label{um}
\frac{\partial u_{m}}{\partial t} = \nabla\cdot(K(|u|) \nabla u_m) +\nabla \cdot \Big[ K'(|u|) \frac{u \cdot \nabla u_m  }{|u|} u\Big].
\eeq
We consider this as  a linear equation for $u_m$ with variable coefficients, and write
\beq\label{uA}
\frac{\partial u_{m}}{\partial t} = \nabla\cdot \mathcal A(x,t,\nabla u_m),
\eeq
where
\beqs
\mathcal A(x,t,\xi)=K(|u(x,t)|)\xi+K'(|u(x,t)|) \frac{u(x,t) \cdot \xi }{|u|} u(x,t).
\eeqs
Using properties \eqref{Kestn} and \eqref{K-est-2}, one can verify that
\beq\label{A1}
|\mathcal A\mathcal (x,t,\xi)|\le (1+a)K(|u(x,t)|)|\xi|,
\eeq
\beq\label{A2}
\mathcal A(x,t,\xi)\cdot \xi \ge (1-a)K(|u(x,t)|)|\xi|^2.
\eeq

We will apply De Giorgi's technique to equation \eqref{um}. 
In the following, we fix a number $s_0$ such that $r=s_0$ satisfies \eqref{rcond}. 
Note that $s_0^*>2$.
We will also use $s_j$ for $j\ge 1$ to denote some exponents that depend on $s_0$ but are independent of $\alpha$.
Let 
\beq\label{s1}
s_1=(1-2/s_0^*)^{-1}>1.
\eeq

\begin{theorem}\label{GradUni} For any $T_0\ge 0$, $T>0$, and $\theta\in (0,1)$, if $t \in [T_0+\theta T,T_0+T]$ then    
\beq\label{gradInf}
\norm{\nabla p(t)}_{L^\infty(U)}\le C(1+(\theta T)^{-1})^\frac{s_1+1}{2} \lambda^\frac{s_1}2  \norm{\nabla p}_{L^2(U\times (T_0+\theta T/2,T_0+T))} +C\sup_{ [T_0+\theta T/2,T_0+T]}\|\nabla p\|_{L^\infty(\Gamma)}, 
\eeq
where 
\beq\label{lambda}
\lambda=\lambda(T_0, T,\theta) = \Big( \int_{T_0+\theta T/2}^{T_0+T} \int_U (1+|\nabla p|)^{\frac {a s_0}{2-s_0}}dx dt  \Big)^\frac{2-s_0}{s_0},
\eeq
and constant $C>0$ is independent of $T_0$, $T$, $\theta$.
\end{theorem}
\begin{proof} Without loss of generality, assume $T_0=0$.  
Denote
$M_b=\sup_{\Gamma\times [\theta T/2,T]}|\nabla p| .$

Fix $m\in\{1,2,\ldots,n\}$.
We will show for $t\in[\theta T,T]$ that
\beq\label{grad519}
\norm{p_{x_m}(t)}_{L^\infty(U)}\le C(1+(\theta T)^{-1})^{\frac{s_1+1}2}\lambda^{\frac {s_1}2} \|p_{x_m}\|_{L^2(V\times (T_0+\theta T/2,T_0+T))}+2M_b.
\eeq

 Let $\zeta(t)$ be a cut-off function with $\zeta(t)=0$ for $t\le \theta T/2$. We define for $k\ge M_b,$ $ u_m^{(k)} =\max\{u_m-k,0\},$
then $u_m^{(k)}=0$ on $\Gamma$. 

Multiplying \eqref{uA} by $u_m^{(k)} \zeta$ and integrating over $U$, using properties \eqref{A1} and \eqref{A2}, we obtain
\beq\label{cut5}
 \frac 12\frac d{dt}\int_U |u_m^{(k)} \zeta|^2 dx
\le  \int_U |u_m^{(k)}|^2 \zeta \zeta_t dx -(1-a)\int_U  K(|u|)|\nabla u_m^{(k)} \zeta|^2  dx.
\eeq
Integrating \eqref{cut5} from $0$ to $t$ for $t\in [0,T]$, and then taking supremum in $t$ give 
\beq\label{cut6}
\max_{[0,T]}\int_U |u_m^{(k)} \zeta|^2  dx  
+C\int_0^T \int_U  K(|u|)|\nabla u_m^{(k)}|^2 \zeta dxdt\le \int_0^T\int_U |u_m^{(k)}|^2 \zeta \zeta_t dx. 
\eeq
Let 
\beq\label{nu2} 
\nu_0=4(1-1/s_0^*)>2.
\eeq
Applying Lemma~\ref{WSobolev} to function $u_m^{(k)} \zeta$ which vanishes on the boundary, weight $W =K(|u|)$ and exponents $r=s_0$, $\varrho=\varrho(s_0)=\nu_0$,  we have 
\begin{align*}
\norm{ u_m^{(k) }\zeta}_{L^{\nu_0}(Q_T)}  &\le C\Big[ \esssup_{t\in [0,T]} \norm{u_m^{(k) }\zeta }_{L^2(U)}+\Big(\int_0^T \int_U K(|u|)|\nabla (u_m^{(k) }\zeta) |^2 dx dt\Big)^\frac 1 2\Big]\\
&\quad\cdot   \Big[\int_0^T\int_{{\rm supp} \zeta} K(|u|)^{-\frac{s_0}{2-s_0}}dx  dt \Big]^{\frac {2- s_0}{\nu_0 s_0} }.
\end{align*}
Using \eqref{Kestn}, we have $K(|u|)^{-\frac{s_0}{2-s_0}}\le C(1+|u|)^{\frac{a s_0}{2-s_0}}$, hence 
\beq\label{cut7}
\norm{ u_m^{(k) }\zeta}_{L^{\nu_0}(Q_T)} \le C\lambda^{1/\nu_0} \Big[ \max_{t\in [0,T]} \int_U |u_m^{(k) }\zeta|^2 dx +\int_0^T \int_U K(|u|)|\nabla (u_m^{(k) }\zeta) |^2 dx dt\Big]^\frac12.
\eeq
By \eqref{cut7}, \eqref{cut6} and the boundedness of function $K(\cdot)$ we find that 
 \beq\label{umk-Ls*}
 \begin{aligned}
\norm{ u_m^{(k) }\zeta}_{L^{\nu_0}(Q_T)} 
\le C\lambda^{1/\nu_0} \Big(  \int_0^T\int_U |u_m^{(k)}|^2 \zeta|\zeta_t| dx\Big)^{\frac12}. 
\end{aligned}
\eeq
Let $M_0\ge M_b$. For $i\ge 0$, let $t_i=\theta T (1-1/2^{i+1})$, $k_i=2M_0(1-2^{i+1})$, and let $\mathcal Q_i$, $\zeta_i(t)$ be the same as in the proof of Theorem \ref{Ntheo42}.
For $i,j\ge 0$, define
 \beqs
A_{i,j}=\{(x,t):u_m(x,t)>k_i,\ t\in(t_j,T)\}, \quad  A_i=A_{i,i},
\eeqs

Let $Y_i=\| u_m^{(k_i)} \|_{L^2(A_{i,i})}$.
Applying \eqref{umk-Ls*} with  $k=k_{i+1}$  and $\zeta=\zeta_i$ and then using the same arguments as in the proof of Theorem \ref{Ntheo42} (see detailed calculations in Theorem 5.6 of \cite {HK1}), we obtain
$$Y_{i+1}\le D B^i Y_i^{1+\nu_3} \quad\text{for all } i\ge 0,$$
where $\nu_3=1-2/\nu_0$, $B =4$ and
$D=C(1+\frac 1{\theta T})^{1/2} \lambda^{1/\nu_0} M_0^{-\nu_3}.$

We now determine $M_0$ so that $ Y_0 \leq D^{-1/\nu_3}B^{-1/{\nu_3^2}}$.
This condition is met if
$$ M_0\ge C  \Big[\lambda^{1/\nu_0} (1+\frac 1{\theta T})^{1/2} \Big]^{1/\nu_3} Y_0=C  \lambda^{s_1/2} (1+\frac 1{\theta T})^\frac{1+s_1}{2} Y_0.$$
Since 
$
Y_0= \norm{u_m^{(k_0)}}_{L^2{(A_{0,0})}}\le \| u_m\|_{L^2(U\times (\theta T/2,T))}
$ and $M_0 \ge M_b$, it suffices to choose $M_0$ as
\beq\label{M0} 
M_0 = C \lambda^{s_1/2}(1+\frac 1{\theta T})^\frac{s_1+1}{2} \|u_m\|_{L^2(U\times (\theta T/2,T))}+M_b.
\eeq
Then Lemma \ref{multiseq} gives 
$\displaystyle{\lim_{i\to\infty}}Y_i=0$. 
Hence, 
$\int_{\theta T}^T\int_U |u_m^{(2M_0)}|^2 dxdt=0.$
Thus, $u_m(x,t)\le M_0$ a.e. in $U\times (\theta T,T)$.
Replace $u_m,u$ by $-u_m,-u$ and 
use the same argument we obtain 
\beq\label{M0um}
|u_m(x,t)|\le 2M_0 \quad  \text{a.e. in } U\times (\theta T,T). 
\eeq
By the choice of $M_0$
  we obtain from \eqref{M0um} that
\beqs
|u_m(x,t)|\le C (1+\frac 1{\theta T})^\frac{s_1+1}{2} \lambda^{s_1/2} \|u_m\|_{L^2(U\times (\theta T/2,T))}  +2M_b
\eeqs
for all $m=1,\ldots,n.$ Then \eqref{grad519} follows, and the proof is compete.
\end{proof}

We will combine Theorem \ref{GradUni} with the high integrability of $\nabla p$ in section \ref{PGradS} to obtain the $L^\infty$-estimates.  For the rest of this paper, we fix the following constants
\beq\label{s2def}
 s_2>\max\Big\{2,\frac{as_0}{2-s_0}\Big\}
 \quad\text{and}\quad
 s_3=s_1(2-s_0)/s_0+1,
\eeq
where $s_0$ and $s_1$ are as in \eqref{s1}. 
We use the same notation as in \eqref{kap1}--\eqref{bK2}, in Theorem \ref{theo55}, and also denote
\beq
\kappa_2=1+s_1+{\tilde s}_2 \kappa_1 s_3,\quad 
\kappa_3={\tilde s}_2 z_3+\alpha/2,
\eeq
\beq
\Kfo(t)= 1+\beta(\alpha)^\frac {1}{\alpha-2a} + \sup_{[t-3, t]}A(\alpha,\cdot)^\frac{1}{\alpha-a},
\eeq
\begin{multline}\label{tilK2}
\Kfv(t)=1+  \sup_{ [t-3,t]} (\|\Psi(t)\|_{L^\infty}+\|\Psi_t(t)\|_{L^\infty}^{\mu_0}+\|\nabla \Psi(t)\|_{L^\infty}^2+\|\nabla^2 \Psi(t)\|_{L^\infty})\\
+(\|\nabla \Psi\|_{L^{\alpha q_1}( U\times (t-3,t))}+\|\Psi_t\|_{L^{\alpha q_1}( U\times (t-3,t))})^{z_2}.
\end{multline}
We recall that the definition of $\tilde s$ is given by \eqref{stil}.

\begin {theorem}\label{theo57}
{\rm (i)} If $0<t\le 3$ then
\begin{multline}\label{Gradps}
\norm{\nabla p(t)}_{L^\infty(U)}
\le C t^{-\kappa_2/2}
 (1+ \|\bar{p}_0\|_{L^\alpha(U)})^{({\tilde s}_2 z_3+1)s_3} \Kzero(t)^{{\tilde s}_2 z_3 s_3} 
 \Kfirst(t)^{{\tilde s}_2 s_3}
\Big( 1+\int_0^t G_1(\tau)d\tau\Big)^{ s_3/2}\\
\cdot \exp\Big[C t^{-1/\delta_1}(1 + \|\bar{p}_0\|_{L^\alpha(U)}^{z_3})\Kzero(t)^{z_3} (1+\|\nabla \Psi\|_{L^{\alpha q_1}(U\times(0,t))}+\|\Psi_t\|_{L^{\alpha q_1}(U\times(0,t))})^{z_2}\Big].
\end{multline}

{\rm (ii)} If $t> 3$ then
\begin{multline}\label{GradpL}
\norm{\nabla p(t)}_{L^\infty(U)}
\le C(1+ \|\bar{p}_0\|_{L^\alpha(U)})^{\kappa_3 s_3} 
\Kzero(t)^{\kappa_3 s_3}
 \Kfv(t)^{{\tilde s}_2 s_3}
\Big( 1+\int_{t-2}^t G_1(\tau)d\tau\Big)^{ s_3/2}\\
\cdot\exp\Big\{C (1 + \|\bar{p}_0\|_{L^\alpha(U)}^{z_3})\Kzero(t)^{z_3}
 (1+ \|\nabla \Psi\|_{L^{\alpha q_1}( U\times (t-2,t))}+ \|\Psi_t\|_{L^{\alpha q_1}( U\times (t-2,t))})^{z_2} \Big\}.
\end{multline}
\end{theorem}
\begin{proof}
First, we note in \eqref{gradInf}  that
\begin{align*}
&\lambda\le C\Big( \int_{T_0+\theta T/2}^{T_0+T} \int_U (1+|\nabla p|)^ {s_2} dx dt  \Big)^{\frac{2-s_0}2},\\
&\norm{\nabla p}_{L^2(U\times (T_0+\theta T/2,T_0+T))} \le C\Big( \int_{T_0+\theta T/2}^{T_0+T} \int_U (1+|\nabla p|)^ {s_2} dx dt  \Big)^{\frac12}.
\end{align*}
Hence, we have from \eqref{gradInf} that
\begin{multline}\label{gradmain}
\sup_{[T_0+\theta T,T_0+T]}\norm{\nabla p(t)}_{L^\infty(U)}\le C (1+(\theta T)^{-1})^{\frac{1+s_1}{2}} 
\Big( \int_{T_0+\theta T/2}^{T_0+T} \int_U (1+|\nabla p|)^ {s_2} dx dt  \Big)^{ s_3/2}\\
+C\sup_{[T_0+\theta T/2,T_0+T]}\|\nabla p\|_{L^\infty(\Gamma)}.
\end{multline}

(i) Let $t\in(0,3]$,  applying \eqref{gradmain} with $T_0=0$, $T=t$, and $\theta=1/2$, we obtain 
\beq\label{gradmain2}
\norm{\nabla p(t)}_{L^\infty(U)}\le C t^{-\frac{1+s_1}{2}} \Big( \int_{t/4}^t \int_U (1+|\nabla p|)^ {s_2} dx dt  \Big)^{ s_3/2}+C\sup_{[t/4,t]}\|\nabla p\|_{L^\infty(\Gamma)}.
\eeq
Using \eqref{es5}  with $s= s_2$ and \eqref{b1} we obtain
\begin{multline*} 
\norm{\nabla p(t)}_{L^\infty(U)}
\le C t^{-\frac{1+s_1}{2}}\Big\{ t\cdot t^{-1-{\tilde s}_2 \kappa_1} (1+ \|\bar{p}_0\|_{L^\alpha(U)})^{2({\tilde s}_2 z_3+1)} 
\Kzero(t)^{2{\tilde s}_2 z_3}
\Kfirst(t)^{2{\tilde s}_2}
\Big( 1+\int_0^t G_1(\tau)d\tau\Big)\\
\cdot\exp\Big[C' t^{-1/\delta_1}(1 + \|\bar{p}_0\|_{L^\alpha(U)})^{z_3} \Kzero(t)^{z_3}(1+\|\nabla \Psi\|_{L^{\alpha q_1}(U\times(0,t))}+\|\Psi_t\|_{L^{\alpha q_1}(U\times(0,t))})^{z_2}\Big]\Big\}^{ s_3/2}\\
+ C  t^{-\mu_0} \Kfirst(t) \exp\Big[C' t^{-1/\delta_1}(1 + \|\bar{p}_0\|_{L^\alpha(U)})^{z_3} \Kzero(t)^{z_3}(1+\|\nabla \Psi\|_{L^{\alpha q_1}(U\times(0,t))}+\|\Psi_t\|_{L^{\alpha q_1}(U\times(0,t))})^{z_2}\Big].
\end{multline*}
Hence,
\begin{multline} \label{gradmain3}
\norm{\nabla p(t)}_{L^\infty(U)}
\le C (t^{-\frac{1+s_1+{\tilde s}_2 \kappa_1 s_3}{2}} + t^{-\mu_0})
 (1+ \|\bar{p}_0\|_{L^\alpha(U)})^{({\tilde s}_2 z_3+1)s_3} 
 \Kzero(t)^{{\tilde s}_2 z_3 s_3} \\
\cdot  (\Kfirst(t)^{{\tilde s}_2 s_3}+\Kfirst(t))
\Big( 1+\int_0^t G_1(\tau)d\tau\Big)^{ s_3/2}\\
\cdot \exp\Big[C' t^{-1/\delta_1}(1 + \|\bar{p}_0\|_{L^\alpha(U)})^{z_3} \Kzero(t)^{z_3}(1+\|\nabla \Psi\|_{L^{\alpha q_1}(U\times(0,t))}+\|\Psi_t\|_{L^{\alpha q_1}(U\times(0,t))})^{z_2}\Big].
\end{multline}
Concerning the power of $t^{-1}$, we note that $s_3, {\tilde s}_2\ge 1$  and 
$\frac{{\tilde s}_2 \kappa_1 s_3}{2}\ge \frac{2{\tilde s}_2 \mu_0  s_3}{2}={\tilde s}_2 \mu_0s_3\ge \mu_0$. 
Then \eqref{Gradps} follows \eqref{gradmain3}.

(ii) Consider $t> 3$. Applying \eqref{gradmain} with $T_0=t-1$, $T=1$, $\theta=1/2$ yields 
 \beq\label{gradlarge}
\norm{\nabla p(t)}_{L^\infty(U)}\le C \Big( 1+\int_{t-3/4}^t \int_U |\nabla p|^{s_2} dx dt  \Big)^{ s_3/2}+C\sup_{[t-3/4,t]}\|\nabla p\|_{L^\infty(\Gamma)}.
\eeq
Thanks to \eqref{es6} with $s= s_2$, and \eqref{b2}, and also noting that $\sup_{[t-3/4,t]}\Kmore(\cdot)\le \Kfv(t)$, we obtain
 \begin{multline*}
\norm{\nabla p(t)}_{L^\infty(U)}
\le C \Big[(1 + \|\bar{p}_0\|_{L^\alpha(U)})^{2\kappa_3}
\Kzero(t)^{2\kappa_3} \Kfv(t)^{2\tilde s_2}
\Big(1+\int_{t-2}^t G_1(\tau)d\tau\Big)\\
\cdot \exp\Big\{C (1 + \|\bar{p}_0\|_{L^\alpha(U)}^{z_3})\Kzero(t)^{z_3}
 (1+ \|\nabla \Psi\|_{L^{\alpha q_1}( U\times (t-2,t))} + \|\Psi_t\|_{L^{\alpha q_1}( U\times (t-2,t))})^{z_2}\Big\}\Big]^{ s_3/2}\\
 +C \Kmore(t)  \exp\Big\{C' (1 + \|\bar{p}_0\|_{L^\alpha(U)})^{z_3}\Kzero(t)(1+ \|\nabla \Psi\|_{L^{\alpha q_1}( U\times (t-2,t))}+\|\Psi_t\|_{L^{\alpha q_1}(U\times(t-2,t))})^{z_2}\Big\}.
\end{multline*}
Then \eqref{GradpL} follows.
\end{proof}

Combining \eqref{gradlarge} with Theorem \ref{theo55}, we have the following asymptotic estimates.

\begin{theorem}\label{theo58}
{\rm (i)} If $A(\alpha)<\infty$ then  
\begin{multline}\label{ptwGradp}
\limsup_{t\to\infty}\norm{\nabla p(t)}_{L^\infty(U)}
\le C \Azero^{\kappa_3 s_3} \Amore^{{\tilde s}_2 s_3} \Afirst^{ s_3/2}\\
\cdot\exp\Big\{C' \Azero^{z_3}
 (1+ \limsup_{t\to\infty}(\|\nabla \Psi\|_{L^{\alpha q_1}(U\times{t-1,t})}+\|\Psi_t\|_{L^{\alpha q_1}(U\times{t-1,t})}))^{z_2} \Big\}.
\end{multline}

{\rm (ii)} If $\beta(\alpha)<\infty$  then there is $T>0$ such that for all $t>T$,
\begin{multline}\label{ptwGradpL}
\norm{\nabla p(t)}_{L^\infty(U)}
\le C  \Kfo(t)^{\kappa_3 s_3}  \Kfv(t)^{{\tilde s}_2 s_3} \Big (1+\int_{t-3/2}^tG_1(\tau)d\tau\Big )^{ s_3/2} \\
\cdot \exp\Big\{C'\Kfo(t)^{z_3}(1+\|\nabla \Psi(t)\|_{L^{\alpha q_1}( U\times (t-3,t))}+\|\Psi_t(t)\|_{L^{\alpha q_1}( U\times (t-3,t))})^{z_2}
\Big\}.
\end{multline}
\end{theorem}
\begin{proof}
Taking the limit superior as $t\to\infty$ of \eqref{gradlarge}, and using estimates  \eqref{LsupGradp-s} and \eqref{b3}  give
\begin{multline*}
\limsup_{t\to\infty}\norm{\nabla p(t)}_{L^\infty(U)}
\le C \Big[\Azero^{2\kappa_3} \Amore^{2\tilde s_2} \Afirst \\
\cdot\exp\Big\{ C'\Azero^{z_3}(1+\limsup_{t\to\infty} (\|\nabla \Psi(t)\|_{L^{\alpha q_1}( U\times (t-1,t))}+\|\Psi_t(t)\|_{L^{\alpha q_1}( U\times (t-1,t))}))^{z_2}
\Big\}\Big]^{s_3/2}\\
+C\Amore  \exp\{C'\Azero^{z_3}(1+\limsup_{t\to\infty}(\|\nabla \Psi(t)\|_{L^{\alpha q_1}( U\times (t-1,t))}+ \|\Psi_t(t)\|_{L^{\alpha q_1}( U\times (t-1,t))}))^{z_2}
\}.
\end{multline*}
Then \eqref{ptwGradp} follows.

(ii) Combining \eqref{gradlarge} with estimate \eqref{Gradp-sL} for $s=s_2$,    and with estimate \eqref{b4} we have 
 \begin{multline*}
\norm{\nabla p(t)}_{L^\infty(U)}\le C\Big\{ \sup_{[t-3/4,t]} \Kth(\cdot)^{2\kappa_3}
\sup_{[t-3/4,t]} \Kmore(\cdot)^{2{\tilde s}_2}
\Big (1+\int_{t-7/4}^tG_1(\tau)d\tau\Big ) \\
\cdot \exp\{ C'\sup_{[t-3/4,t]} \Kth(\cdot)^{z_3} (1+\|\nabla \Psi(t)\|_{L^{\alpha q_1}( U\times (t-11/4,t))}+\|\Psi_t(t)\|_{L^{\alpha q_1}( U\times (t-11/4,t))})^{z_2}
\}  \Big\}^{ s_3/2}\\
+C(1+\sup_{[t-7/4,t]} (\|\Psi_t\|_{L^\infty}^\frac2{2-a}+\|\nabla \Psi\|_{L^\infty}^2+ \|\nabla^2 \Psi\|_{L^\infty}))\\
\cdot \exp\{C
(1+ \beta(\alpha)^{\frac {1}{\alpha-2a} } + \norm{A(\alpha,\cdot)}_{L^{\frac\alpha{\alpha-a}}(t-7/4,t)}^\frac{1}{\alpha-a})^{z_3} (1+\|\nabla \Psi(t)\|_{L^{\alpha q_1}( U\times (t-7/4,t))}+\|\Psi_t(t)\|_{L^{\alpha q_1}( U\times (t-7/4,t))})^{z_2}\}
\end{multline*}
for all $t\ge T$ with some $T>3$.
This leads to \eqref{ptwGradpL}.     
\end{proof}

\section{$L^\infty$-estimates for the time derivative}
\label{PTSec}
We now estimate the $L^\infty$-norm of the pressure's time derivative.
Let $p(x,t)$ be a solution of IBVP \eqref{p:eq}, and let $q= p_t$. Then
\beq\label{eqt} \frac{\partial   q}{\partial t}=\nabla \cdot \big(K(|\nabla p|)\nabla  p\big)_t .\eeq 

\begin{proposition} \label{ptInf}  If $T_0\ge 0$, $T>0$ and $\theta\in(0,1)$, then
\beq\label{ptbound}
\sup_{[T_0+\theta T,T_0+T]}\| p_t\|_{L^\infty(U)}\le C\lambda^\frac{s_1}{2} (1+(\theta T)^{-1} )^\frac{s_1+1}{2} \| p_t\|_{L^2(U\times (T_0+\theta T/2,T_0+T))}+\max_{\Gamma\times [T_0,T_0+T]} |\psi_t|,
\eeq
where $s_1$ and $\lambda=\lambda(T_0,T,\theta)$ are defined in Theorem \ref{GradUni},
and constant $C>0$ is independent of $T_0$, $T$, and $\theta$.
\end{proposition}
\begin{proof} Without loss of generality, assume $T_0=0$.
For $k\ge \max_{\Gamma\times [0,T]} |p_t| = \max_{\Gamma\times [0,T]} |\psi_t|$ , let $ q^{(k)}=\max\{ q-k,0\}$
and $ S_k(t)=\{x \in U: q(x,t)>k\}$,
and $\chi_k(x,t)$  be the characteristic function of set $\{(x,t) \in U\times(0,T):  q(x,t)>k\}$.
On $S_k(t)$,  we have $(\nabla p)_t=\nabla q=\nabla q^{(k)}$. 

Let $\zeta=\zeta(t)$ be the cut-off function satisfying $\zeta(0)=0$. We will use test function $ q^{(k)}\zeta^2$, noting that $\nabla ( q^{(k)}\zeta^2) =\zeta^2 \nabla  q^{(k)}$. Multiplying \eqref{eqt} by $ q^{(k)}\zeta^2 $ and integrating the resultant on $U$, we get 
\begin{align*}
\frac 12\ddt \int_U |q^{(k)}\zeta|^2 dx & 
= \int_U |q^{(k)}|^2 \zeta \zeta_t  dx - \int_U (K(|\nabla p|))_t \nabla p \cdot \nabla q^{(k)}\zeta^2  dx - \int_U K(|\nabla p|) (\nabla p)_t \cdot  \nabla q^{(k)}\zeta^2 dx. 
\end{align*}
Note that 
$
(\nabla p)_t \cdot \nabla q^{(k)}\zeta^2= |\nabla (q^{(k)}\zeta)|^2.
$
Taking into account \eqref{K-est-2},   
\begin{align*}
|(K(|\nabla p|))_t\nabla p \cdot \nabla q^{(k)}\zeta^2|\
&=|K'(|\nabla p|)| \frac{|\nabla p\cdot \nabla p_t|}{|\nabla p|} |\nabla p \cdot \nabla q^{(k)}\zeta^2| 
\le a K(|\nabla p|) |\nabla q||\nabla q^{(k)}\zeta^2|.
\end{align*}
Note that  
$
|\nabla q| |\nabla q^{(k)}\zeta^2|= |\nabla (q^{(k)}\zeta)|^2.
$
It follows that
\begin{align*}
\ddt \int_U |q^{(k)} \zeta|^2 dx &+ (1-a) \int_U K(|\nabla p|) |\nabla (q^{(k)} \zeta)|^2 dx   \le 2\int_U |q^{(k)}|^2 \zeta  |\zeta_t| dx .
\end{align*}
Integrating this inequality from $0$ to $T$, we obtain 
\begin{align*}
 \max_{[0,T]} \int_U |q^{(k)}\zeta |^2 dx &+ \int_0^T \int_U K(|\nabla p|) |\nabla (q^{(k)}\zeta) |^2 dx dt \le C\int_0^T  \int_U |q^{(k)}|^2 \zeta |\zeta_t| dxdt .
 \end{align*}
 The last inequality uses the fact that function $K(\cdot)$ is bounded above.
 Applying Lemma \ref{WSobolev} to $q^{(k)} \zeta$ with $W=K(|\nabla p|)$ and $\varrho=\nu_0$ defined by \eqref{nu2}, we have 
  \begin{align*}
&   \norm{q^{(k)}\zeta}_{L^{\nu_0}(Q_T)}
 \le C \lambda^{1/\nu_0} 
  \cdot \left\{\max_{[0,T]} \int_U |q^{(k)}\zeta |^2 dx + \int_0^T \int_U K(|\nabla p|)|\nabla (q^{(k)}\zeta)|^2 dx dt \right\}^{1/2}.
 \end{align*} 
 and hence, 
  \beq\label{qzetabound}
  \begin{split}
  \norm{q^{(k)}\zeta}_{L^{\nu_0}(Q_T)} &\le C\lambda^{1/\nu_0}  \Big(\int_0^T  \int_U  |q^{(k)}|^2 |\zeta_t|\zeta dx dt \Big)^{1/2}.
   \end{split}
   \eeq
This is similar to inequality \eqref{umk-Ls*}. Then by following arguments of Theorem~\ref{GradUni} applied for $q^{(k)}$ instead of $u_m^{(k)}$, we obtain the estimate \eqref{ptbound}.    We omit the details.
\end{proof}

In the following estimates of $p_t$ we use the same notation as in sections \ref{PGradS} and \ref{PGradInfty}, particularly, \eqref{kap1}--\eqref{bK2},  \eqref{tilA1}, \eqref{tilA2}, \eqref{s2def}--\eqref{tilK2},  and also define new numbers
\begin{align*}
\kappa_4&=1+s_1+{\tilde s}_2 \kappa_1 (s_3-1),\quad
\kappa_5=({\tilde s}_2 z_3 +1)(s_3-1)+1,\\
\kappa_6&=({\tilde s}_2 z_3+\alpha/2)(s_3-1)+\alpha/2=\kappa_3(s_3-1)+\alpha/2.
\end{align*}

\begin{theorem}\label{ptthm}
 {\rm (i)} If $0<t\le 3$ then
 \begin{multline}\label{i66}
  \| p_t(t)\|_{L^\infty(U)}
  \le C  t^{-\kappa_4/2} \Big(1+\|\bar{p}_0\|_{L^\alpha}\Big)^{\kappa_5} \Big(1+ \int_U H(|\nabla p_0(x)|)dx \Big)^{1/2} \\
\cdot \Kzero(t)^{{\tilde s}_2 (s_3-1)z_3}\Kfirst(t)^{{\tilde s}_2 (s_3-1)} 
\Big(1+\int_0^t G_3(\tau)d\tau)\Big)^{s_3/2} \\
 \cdot \exp\Big\{C t^{-1/\delta_1}(1 + \|\bar{p}_0\|_{L^\alpha(U)})^{z_3} \Kzero(t)^{z_3}(1+\|\nabla \Psi\|_{L^{\alpha q_1}(U\times(0,t))}+\|\Psi_t\|_{L^{\alpha q_1}(U\times(0,t))})^{z_2}\Big\}
 +\sup_{ [0,t]} \|\psi_t\|_{L^\infty(\Gamma)}.
 \end{multline}
 
{\rm (ii)}  If $t>3$ then
\begin{multline}\label{i67}
\| p_t(t)\|_{L^\infty(U)}
\le C (1+\|\bar p_0\|_{L^\alpha})^{\kappa_6}
\cdot  \Kzero(t)^{\kappa_6} \Kfv(t)^{{\tilde s}_2 (s_3-1)} \Big( 1+\int_{t-2}^t G_3(\tau)d\tau\Big)^{s_3/2}\\
\cdot \exp\Big\{C' (1 + \|\bar{p}_0\|_{L^\alpha(U)}^{z_3})\Kzero(t)^{z_3}
 (1+ \|\nabla \Psi\|_{L^{\alpha q_1}( U\times (t-3,t))}+ \|\Psi_t\|_{L^{\alpha q_1}( U\times (t-3,t))})^{z_2}\Big\}
  +\sup_{ [t-1,t]} \|\psi_t\|_{L^\infty(\Gamma)}.
\end{multline}
\end{theorem}
\begin{proof}
{\rm (i)} Let $t\in(0,3]$. Applying \eqref{ptbound} with $T_0=0$, $T=t$ and $\theta=1/2$ gives
\beq \label{pt1}
\| p_t(t)\|_{L^\infty}
\le C\lambda^\frac{s_1}{2}  t^{-\frac{1+s_1}{2}} \| p_t\|_{L^2(U\times (0,t))}+\sup_{ [0,t]} \|\psi_t\|_{L^\infty},
\eeq
where $\lambda$ is defined by \eqref{lambda}. By Young's inequality, we have
\begin{align*}
\lambda^{s_1/2} & = \Big( \int_{t/4}^{t} \int_U (1+|\nabla p|)^{\frac {a s_0}{2-s_0}}dx d\tau  \Big)^\frac{\nu_1}{2}
\le \Big( \int_{t/4}^{t} \int_U (1+|\nabla p|)^{s_2} dx d\tau  \Big)^\frac{\nu_1}{2},
\end{align*}
where $\nu_1=(2-s_0)s_1/s_0= s_3-1$.
Combining with \eqref{es5} for $s=\tilde s_2$ yields
\begin{multline*}
\lambda^{s_1/2}\le C\Big[ t\cdot t^{-1-{\tilde s}_2 \kappa_1} (1 + \|\bar{p}_0\|_{L^\alpha(U)})^{2({\tilde s}_2 z_3+1)} 
\Kzero(t)^{2{\tilde s}_2 z_3}
\Kfirst(t)^{2{\tilde s}_2}
\Big( 1+\int_0^t G_1(\tau)d\tau\Big)\\
\cdot \exp\Big\{C' t^{-1/\delta_1}(1 + \|\bar{p}_0\|_{L^\alpha(U)})^{z_3} \Kzero(t)^{z_3} (1+\|\nabla \Psi\|_{L^{\alpha q_1}(U\times(0,t))}+\|\Psi_t\|_{L^{\alpha q_1}(U\times(0,t))})^{z_2}\Big\}\Big]^{\nu_1/2}.
\end{multline*}
Combining this with \eqref{pt1} and \eqref{intpt} gives
\begin{align*}
& \| p_t(t)\|_{L^\infty}
\le C t^{-\frac{1+s_1+{\tilde s}_2 \kappa_1\nu_1}{2}}(1+\|\bar p_0\|_{L^\alpha})^{({\tilde s}_2 z_3+1)\nu_1} 
 \Kzero(t)^{{\tilde s}_2 z_3\nu_1}
 \Kfirst(t)^{{\tilde s}_2 \nu_1} \cdot\Big( 1+\int_0^t G_1(\tau)d\tau\Big)^{\nu_1/2}\\
 &\quad \cdot \exp\Big\{C' t^{-1/\delta_1}(1 + \|\bar{p}_0\|_{L^\alpha(U)})^{z_3} \Kzero(t)^{z_3}(1+\|\nabla \Psi\|_{L^{\alpha q_1}(U\times(0,t))}+\|\Psi_t\|_{L^{\alpha q_1}(U\times(0,t))})^{z_2}\Big\}\\
&\quad \cdot\Big( \int_U [H(|\nabla p_0(x)|)+\bar p^2_0(x)] dx +\int_0^t G_3(\tau)d\tau +\int_0^t \int_U |\Psi_t(x,\tau)|^2dxd\tau\Big)^{1/2}  +\sup_{ [0,t]} \|\psi_t\|_{L^\infty}.
\end{align*}
Therefore,
\begin{align*}
& \| p_t(t)\|_{L^\infty}
  \le C  t^{-\kappa_4/2} \Big(1+\|\bar{p}_0\|_{L^\alpha}\Big)^{({\tilde s}_2 z_3+1)\nu_1+1} \Big(1+ \int_U H(|\nabla p_0(x)|)dx \Big)^{1/2}\\
&\quad \cdot   \Kzero(t)^{{\tilde s}_2 z_3\nu_1}\Kfirst(t)^{{\tilde s}_2 \nu_1} 
\Big(1+\int_0^t G_3(\tau)d\tau)\Big)^{(\nu_1+1)/2} \\
&\quad \cdot \exp\Big\{C t^{-1/\delta_1}(1 + \|\bar{p}_0\|_{L^\alpha(U)})^{z_3} \Kzero(t)^{z_3}(1+\|\nabla \Psi\|_{L^{\alpha q_1}(U\times(0,t))}+\|\Psi_t\|_{L^{\alpha q_1}(U\times(0,t))})^{z_2}\Big\}
 +\sup_{ [0,t]} \|\psi_t\|_{L^\infty}.
\end{align*}
Thus, we obtain \eqref{i66}.

{\rm (ii)} Now consider $t> 3$. Applying \eqref{ptbound} with $T_0=t-1$, $T=1$ and $\theta=1/2$, we obtain
\beq \label{pt2}
\| p_t(t)\|_{L^\infty}
\le C\lambda^\frac{s_1}{2}  \| p_t\|_{L^2(U\times (t-3/4,t))}+\sup_{ [t-1,t]} \|\psi_t\|_{L^\infty},
\eeq
where $\lambda= \Big( \int_{t-3/4}^t \int_U (1+|\nabla p|)^{\frac {a s_0}{2-s_0}}dx dt  \Big)^\frac{2-s_0}{s_0}$.
Then by Young's inequality  and \eqref{es6} for $s=\tilde s_2$,
\begin{align*}
\lambda^{s_1/2}
&\le C \Big( \int_{t-1}^{t} \int_U (1+|\nabla p|)^ {s_2} dx d\tau  \Big)^{\nu_1/2}\\
&\le C (1+\|\bar p_0\|_{L^\alpha})^{\kappa_3\nu_1} \Kzero(t)^{\kappa_3\nu_1} \Kfv(t)^{{\tilde s}_2 \nu_1}
\cdot \Big( 1+\int_{t-2}^t G_1(\tau)d\tau\Big)^{\nu_1/2}\\
&\quad \cdot \exp\Big\{C' (1 + \|\bar{p}_0\|_{L^\alpha(U)})^{z_3}\Kzero(t)^{z_3} 
 (1+ \|\nabla \Psi\|_{L^{\alpha q_1}( U\times (t-3,t))}+ \|\Psi_t\|_{L^{\alpha q_1}( U\times (t-3,t))})^{z_2}\Big\}.
\end{align*}
Combining this with \eqref{pt2} and \eqref{all2} gives
\begin{align*}
 &\| p_t(t)\|_{L^\infty(U)}
 \le C (1+\|\bar p_0\|_{L^\alpha})^{\kappa_3\nu_1} \Kzero(t)^{\kappa_3\nu_1} \Kfv(t)^{{\tilde s}_2 \nu_1}
\cdot \Big( 1+\int_{t-2}^t G_1(\tau)d\tau\Big)^{\nu_1/2}\\
&\quad \cdot \exp\Big\{C' (1 + \|\bar{p}_0\|_{L^\alpha(U)})^{z_3}\Kzero(t)^{z_3}
 (1+ \|\nabla \Psi\|_{L^{\alpha q_1}( U\times (t-3,t))}+ \|\Psi_t\|_{L^{\alpha q_1}( U\times (t-3,t))})^{z_2}\Big\}\\
&\quad  \cdot \Big(1+ \int_U|\bar{p}_0(x)|^{\alpha} dx
+ [ Env A(\alpha,t)]^\frac{\alpha}{\alpha-a} + \int_{t-2}^tG_3(\tau)d\tau+\int_{t-2}^t \int_U |\Psi_t|^2dxd\tau\Big )^{1/2} +\sup_{ [t-1,t]} \|\psi_t\|_{L^\infty}.
\end{align*}
Using the facts $G_1(t)\le G_3(t)$ and $\int_U |\Psi_t(x,t)|^2 dx\le G_3(t)$,  we infer that  
\begin{multline*}
\| p_t(t)\|_{L^\infty(U)}
\le C (1+\|\bar p_0\|_{L^\alpha})^{\kappa_3\nu_1+\alpha/2}
\Kzero(t)^{\kappa_3\nu_1+\alpha/2}\Kfv(t)^{{\tilde s}_2 \nu_1} \Big( 1+\int_{t-2}^t G_3(\tau)d\tau\Big)^{\nu_1/2+1/2}\\
\cdot \exp\Big\{C' (1 + \|\bar{p}_0\|_{L^\alpha(U)})^{z_3}\Kzero(t)^{z_3}
 (1+ \|\nabla \Psi\|_{L^{\alpha q_1}( U\times (t-3,t))}+ \|\Psi_t\|_{L^{\alpha q_1}( U\times (t-3,t))})^{z_2} \Big\}
+\sup_{ [t-1,t]} \|\psi_t\|_{L^\infty},\end{multline*}
and obtain \eqref{i67}. The proof is complete.
\end{proof}

For large time or asymptotic estimates, we have the following.

\begin{theorem}\label{ptlimthm}

 {\rm (i)} If $ A(\alpha)<\infty$ then
\begin{multline}\label{lim7} 
\limsup_{t\to\infty} \| p_t(t)\|_{L^\infty(U)}
 \le C \Azero^{\kappa_6}  \Amore^{{\tilde s}_2(s_3-1)}\Ath^{s_3/2} \\
 \cdot\exp\Big\{ C'\Azero^{z_3}(1+\limsup_{t\to\infty}( \|\nabla \Psi\|_{L^{\alpha q_1}( U\times (t-1,t))}+ \|\Psi_t(t)\|_{L^{\alpha q_1}( U\times (t-1,t))}))^{z_2}\Big\} 
+\limsup_{t\to\infty} \|\psi_t\|_{L^\infty(\Gamma)},
\end{multline}
where
$\Ath=1+ \limsup_{t\to\infty} \int_{t-1}^tG_3(\tau)d\tau.$

{\rm (ii)} If $\beta(\alpha)<\infty$ then there is $T>0$ such that for all $t>T$,
 \beq\label{large7}
\begin{aligned}
 & \| p_t(t)\|_{L^\infty(U)}
 \le C \Kfo(t)^{\kappa_6}  \Kfv(t)^{{\tilde s}_2 (s_3-1)} \Big(   1+\int_{t-2}^tG_3(\tau)d\tau\Big)^{s_3/2}\\
 &\quad  \cdot\exp\Big\{  C'\Kfo(t)^{z_3}(1+ \|\nabla \Psi\|_{L^{\alpha q_1}( U\times (t-3,t))}+ \|\Psi_t(t)\|_{L^{\alpha q_1}( U\times (t-3,t))})^{z_2}
\Big\}+\max_{ [t-1,t]} \|\psi_t\|_{L^\infty(\Gamma)}.
\end{aligned}
\eeq
 \end{theorem}
\begin{proof}
Again, let $\nu_1=s_3-1$.

(i)  
Taking  the limit superior of \eqref{pt2} as $t\to\infty$, and using \eqref{LsupGradp-s} with $s= s_2$, and \eqref{lim4}, we obtain
 \begin{multline*}
\limsup_{t\to\infty} \| p_t(t)\|_{L^\infty(U)}
 \le C \Big[ \Azero^{2\kappa_3}\Amore^{2{\tilde s}_2} \Afirst \\
\cdot\exp\Big\{  C'\Azero^{z_3}(1+\limsup_{t\to\infty}(\|\nabla \Psi\|_{L^{\alpha q_1}( U\times (t-1,t))}+ \|\Psi_t(t)\|_{L^{\alpha q_1}( U\times (t-1,t))}))^{z_2}
\Big\} \Big]^{\nu_1/2}\\
\quad \cdot \Big (\Azero^{\alpha} +\limsup_{t\to\infty} \int_{t-1}^tG_3(\tau)d\tau+\limsup_{t\to\infty}\int_{t-1}^t \int_U |\Psi_t(x,\tau)|^2dxd\tau\Big)^{1/2}
+\limsup_{t\to\infty} \|\psi_t\|_{L^\infty(\Gamma)}.
\end{multline*}
Simple manipulations yield
 \begin{multline}\label{last1}
\limsup_{t\to\infty} \| p_t(t)\|_{L^\infty(U)}
 \le C\Azero^{\kappa_3 \nu_1+\alpha/2} \Amore^{{\tilde s}_2 \nu_1}\Afirst^{\nu_1/2}
\cdot \Big(1+ \limsup_{t\to\infty} \int_{t-1}^tG_3(\tau)d\tau\Big)^{1/2}\\
 \cdot\exp\Big\{  C'\Azero^{z_3}(1+\limsup_{t\to\infty}(\|\nabla \Psi\|_{L^{\alpha q_1}( U\times (t-1,t))}+ \|\Psi_t(t)\|_{L^{\alpha q_1}( U\times (t-1,t))}))^{z_2}
\Big\} +\limsup_{t\to\infty} \|\psi_t\|_{L^\infty(\Gamma)}.
\end{multline}
Since $\Afirst\le \Ath$, we have
 $\Afirst^{\nu_1/2}(1+ \limsup_{t\to\infty} \int_{t-1}^tG_3(\tau)d\tau)^{1/2}\le \Ath^{(s_3-1)/2}\Ath^{1/2}=\Ath^{s_3/2}.$ Thus, inequality \eqref{lim7} follows \eqref{last1}.

(ii) Combining inequality \eqref{pt2} with estimate \eqref{Gradp-sL} for $s=\tilde s_2$ and estimate \eqref{large4}, we obtain for large $t$ that
 \begin{align*}
& \| p_t(t)\|_{L^\infty(U)}
 \le C \Big[
\Kfo(t)^{2\kappa_3}
\Kfv(t)^{2{\tilde s}_2}\Big (1+\int_{t-2}^tG_1(\tau)d\tau\Big )\\
&\cdot \exp\Big\{  C'\Kfo(t)^{z_3}(1+ \|\nabla \Psi\|_{L^{\alpha q_1}( U\times (t-3,t))}+\|\Psi_t(t)\|_{L^{\alpha q_1}( U\times (t-3,t))})^{z_2}
\Big\}  \Big]^{\nu_1/2}\\
& \cdot\Big (1+\beta(\alpha)^\frac {\alpha}{\alpha-2a} + \sup_{[t-1,t]}A(\alpha,\tau-1)^\frac{\alpha}{\alpha-a}
+\int_{t-2}^tG_3(\tau)d\tau  +\int_{t-2}^t \int_U |\Psi_t(x,\tau)|^2dxd\tau \Big )^{1/2}\\
&+\max_{[t-1,t]} \|\psi_t\|_{L^\infty(\Gamma)}.
\end{align*} 
Using similar calculations as in part (i) and  the fact $G_1\le G_3$, we obtain \eqref{large7}.
\end{proof}


\textbf{Acknowledgment.} The authors would like to thank Akif Ibragimov and  Alexander Nazarov for their helpful discussions.
L. H. acknowledges the support by NSF grant DMS-1412796.

\def\cprime{$'$}

\end{document}